\documentclass{amsart}

\usepackage{amssymb}
\numberwithin{equation}{section}

\newcounter{proofitem}

\newtheorem{thm}[subsubsection]{Theorem}
\newtheorem{cor}[subsubsection]{Corollary}
\newtheorem{lem}[subsubsection]{Lemma}
\newtheorem{lemdef}[subsubsection]{Lemma-Definition}
\newtheorem{prop}[subsubsection]{Proposition}

\newtheorem{case}[proofitem]{Case}

\theoremstyle{definition}
\newtheorem{defn}[subsubsection]{Definition}

\theoremstyle{remark}
\newtheorem{exs}[subsubsection]{Examples}

\newtheorem{rem}[subsubsection]{Remark}

\newcommand{\comment}[1]{}
\newcommand{\Ko}{K^\circ}
\newcommand{\Koo}{K^{\circ\circ}}

\newcommand{\Kt}{\widetilde K}

\newcommand{\Kalg}{K_{alg}}
\newcommand{\Kalgo}{K_{alg}^\circ}
\newcommand{\Kalgoo}{K_{alg}^{\circ\circ}}
\newcommand{\Kalgt}{\widetilde K_{alg}}
\newcommand{\cE}{\mathcal E}

\newcommand{\abs}[1]{\left\vert#1\right\vert}

\renewcommand{\epsilon}{\varepsilon}
\renewcommand{\phi}{\varphi}
\renewcommand{\emptyset}{\varnothing}

\let\Bbb\mathbb
\def\wedgem{\mathop{\wedge}}

\def\deg{{\rm deg}}
\def\ord{{\rm ord}}

\def\Hens{{\rm Hen}}

\newcommand{\Val}{\mathrm{Val}}

\let\cal\mathcal

\def\11{{\mathbf 1}}
\def\AA{{\mathbf A}}

\def\NN{{\mathbf N}}

\def\QQ{{\mathbf Q}}
\def\RR{{\mathbf R}}

\def\ZZ{{\mathbf Z}}

\def\cA{{\mathcal A}}
\def\cB{{\mathcal B}}

\def\cD{{\mathcal D}}
\def\cE{{\mathcal E}}
\def\cF{{\mathcal F}}

\def\cH{{\mathcal H}}

\def\cL{{\mathcal L}}

\def\cO{{\mathcal O}}

\def\cR{{\mathcal R}}
\def\cS{{\mathcal S}}
\def\cT{{\mathcal T}}
\def\cU{{\mathcal U}}
\def\cV{{\mathcal V}}

 \def\cU{{\cal U}}
 \def\cA{{\cal A}}
 \def\cB{{\cal B}}
 \def\cR{{\cal R}}
 \def\ra{\rangle}
 \def\la{\langle}
 \def\bF{\mathbb F}
 \def\bQ{\mathbb Q}
 \def\bZ{\mathbb Z}
 \def\bR{\mathbb R}
 \def\bC{\mathbb C}
 \def\bN{\mathbb N}

 \def\supp{\emph{supp}}

\def\Hens{{\rm Hen}}
\def\LHens{\cL_\Hens{ }}
\def\THens{\cT_\Hens{}}

\begin{document}

\author[Cluckers]{R.~Cluckers$^\mathrm{1}$}
\address{Katholieke Universiteit Leuven, Departement wiskunde,
Celestijnenlaan 200B, B-3001 Leu\-ven, Bel\-gium.}
\email{raf.cluckers@wis.kuleuven.be}
\urladdr{www.wis.kuleuven.be/algebra/Raf/}

\begin{abstract}
We present a unifying theory of fields with certain classes of
analytic functions, called fields with analytic structure. Both real
closed fields and Henselian valued fields are considered.
 For real
closed fields with analytic structure, $o$-minimality is shown. For
Henselian valued fields, both the model theory and the analytic
theory are developed. We give a list of examples that comprises, to
our knowledge, all principal, previously studied, analytic
structures on Henselian valued fields, as well as new ones. The
$b$-minimality is shown, as well as other properties useful for
motivic integration on valued fields. The paper is reminiscent of
[Denef, van den Dries, \textit{$p$-adic and real subanalytic sets.}
Ann.~of Math. (2) \textbf{128} (1988) 79--138], of [Cohen, Paul J.
\textit{Decision procedures for real and $p$-adic fields.}
Comm.~Pure Appl.~Math. \textbf{22} (1969) 131--151], and of
[Fresnel, van der Put, \textit{Rigid analytic geometry and its
applications.} Progress in Mathematics, \textbf{218} Birkh{\"a}user
(2004)], and unifies work by van den Dries, Haskell, Macintyre,
Macpherson, Marker, Robinson, and the authors.
\end{abstract}

\thanks{$^\mathrm{1}$ Supported as a postdoctoral fellow by the Fund for
Scientific Research - Flanders (Belgium) (F.W.O.) and partially by The
European Commission - Marie Curie European Individual Fellowship
with contract number HPMF CT 2005-007121 during the preparation of
this paper.}
\thanks{$^\mathrm{2}$ Supported in part by NSF grant DMS-0401175.}
\thanks{$^\mathrm{1,2}$ The authors acknowledge stimulating
conversations with Zach Robinson on some of the subject matter of
this paper
and thank the University of Leuven and Purdue University for support
and hospitality. The research has partially benefitted from the project ANR-06-BLAN-0183.}

\author[Lipshitz]{L.~Lipshitz$^\mathrm{2}$}
\address{Department of Mathematics, Purdue University, 150 North University Street, West Lafayette IN 47907-2067, USA}
\email{lipshitz@math.purdue.edu}
\urladdr{www.math.purdue.edu/$\thicksim$lipshitz/}



 \title{Fields with Analytic Structure}
 \maketitle

 \section{Introduction}
We begin with some background. Let $A$ be a Noetherian ring, $t\in
A$ and assume that $A$ is $t$-adically complete. The notion of a
valued field with analytic $A$-structure was introduced by van den
Dries in \cite{vdD}. The principal example is $A=\bZ[[t]]$, and the
corresponding ``analytic functions'' are the strictly convergent
power series over $A$ (i.e. the elements of ~$T_n(\bZ[[t]])=\{\sum
a_\nu\xi^\nu\colon a_\nu\to 0$ $t$-adically as~$|\nu|\to\infty\}$.)
The natural homomorphisms $\bZ[[t]]\to\bQ_p ,t \mapsto p,$ and
$\bZ[[t]]\to\bF_p((t))$, $\bZ\to\bZ/p\bZ$ give homomorphisms of
$T_n(\bZ[[t]])$ to $T_n(\bQ_p)=\bQ_p\la\xi\ra$ and
$T_n(\bF_p((t)))=\bF_p((t))\la\xi\ra$, the rings of strictly
convergent power series over $\bQ_p$ and $\bF_p((t))$ respectively,
and thus the fields $\bQ_p$, $\bF_p((t))$ and ultraproducts of these
fields are all structures in a natural way for the valued field
language with function symbols for the elements of $\bigcup_m T_m
(\bZ[[t]])$. This formalism was used by van den Dries to establish
an analytic analogue of the algebraic quantifier elimination theorem
of \cite{Pas1} and also an analytic version of the Ax-Kochen
Principle (when $A$ is a discrete valuation ring), namely, if $\cU$
is a nonprincipal ultrafilter on the set of primes then
$(\Pi_p\bQ_p)/\cU\equiv (\Pi_p\bF_p((t)))/\cU$, in this analytic
language.

 In \cite{DHM} results about $p$-adic subanalytic sets were established by studying nonstandard models of the theory of $\bQ_p$ in the language with symbols for the elements of $\bigcup_n T_n (\bQ_p)$, the strictly convergent power series over $\bQ_p$.
 This can be thought of as an investigation of fields with analytic $\bZ_p$-structure, i.e. taking $A=\bZ_p$, $t=p$.

In \cite{DMM1} it was observed that if $K$ is a maximally complete,
non-archimedean real closed field with divisible value group, and if
$f$ is an element of $\bR[[\xi]]$ with radius of convergence $>1$,
then $f$ extends naturally to an ``analytic'' function $I^n\to K$,
where $I=\{x\in K\colon -1\leq x\leq 1\}$. Hence if $\cA$ is the
ring of real power series with radius of convergence $>1$ then $K$
has analytic $\cA$-structure i.e.~this extension preserves all the
algebraic properties of the ring $\cA$. The real quantifier
elimination of \cite{DD} works in this context implying that $K$ is
elementarily equivalent to the subanalytic structure on $\RR$ and
hence $o$-minimal.
 See \cite{DMM2} and \cite{DMM3} for extensions.

Some of the results of \cite{LR1} also fit into this context. If
$K\subset K'$ are complete (rank one) valued fields,  then the
separated power series $S_{m,n}(E,K)$ define analytic functions on
$(K'^{\circ})^m\times(K'^{\circ\circ})^n$ and $(K_{alg}'^{\circ})^m
\times (K_{alg}'^{\circ\circ})^n$,  where $K'_{alg}$ is the
algebraic closure of $K'$, and furthermore $K'_{alg}$ has quantifier
elimination in the corresponding separated $K$-analytic language. In
\cite{LR2} similar results were established for  certain subrings of
the $S_{m,n}$, namely the elements of $S_{m,n}(E,K)$ that are
existentially definable over $T_{m+n}(K)$, and this was used to
establish quantifier simplification for $K'_{alg}$ in the
corresponding strictly convergent (or affinoid) analytic language
(i.e.~the language with function symbols for the elements of
$\bigcup_m T_m (K)$).

In \cite{LR3} the notion of analytic $\bZ[[t]]$-structure of
\cite{vdD} was extended to the separated context and this was used
to establish various uniformity results on quantifier elimination,
smooth stratification, {\L}ojasiewicz inequalities and topological
closure for rigid semianalytic and subanalytic sets for
algebraically closed valued fields. The notion of analytic
$\bigcup_{m,n}S_{m,n}(E,K)$-structure was also extended to
nonstandard models.

In \cite{CLR1}, sections 1 and 2, an, in some ways, more general
framework than that of \cite{vdD} for studying Henselian valued
fields with analytic structure in analytic Denef-Pas languages was
given, and used to prove a fairly general Cell Decomposition Theorem
which was applied to certain questions in motivic integration. It
required an analysis of terms and definable functions in one
variable.  This framework for Henselian valued fields with analytic structure has been
further developed and applied in \cite{Ce1}, \cite{Ce2} and \cite{Ce3}.

In \cite{LR4}  quantifier elimination and $o$-minimality for the
field $K=\bigcup_n\bR ((t^{1/n}))$ of real Puiseux series (or its
completion $\widehat{K}$) in an analytic language with function
symbols for $t$-adically ``overconvergent'' power series (for
example $\Sigma_n(n+1)! t^nx^n$ which converges on the disc $\{x
\colon \|x\| \leq \|t^{-\frac{1}{2}}\|\}$, and hence is $t$-adically
overconvergent on $K^\circ$) was established. These power series
define analytic functions $I^n\to K$ but do not live in any
$o$-minimal expansion of $\bR$ (see \cite{HP}). In \cite{CLR2} these
results were extended to larger classes of ``analytic'' functions on
more nonstandard real closed fields. These results also required
analyses of terms and definable functions in one variable.

All the above mentioned work involved the study of (usually)
nonstandard fields (either Henselian or real closed) in a language
with function symbols for a (frequently standard, but sometimes
nonstandard) family of analytic functions. Results are established
for the nonstandard fields using algebraic properties of the ring of
``functions'' in the language (for example parametrized Weierstrass
Preparation and Division, a Strong Noetherian Property and sometimes
compactness.) In this paper we endeavor to give a sufficiently
general framework for \emph{fields with analytic structure} to unify
the above-mentioned work and hopefully to facilitate further
applications.

In Section \ref{maxcom} we briefly review some facts about maximally
complete valued fields. In Section \ref{realAS} we give the basic
definitions and properties of (ordered) fields with real analytic
structure. In Section \ref{sepanstruct} we give the definitions and
basic properties of Henselian fields with separated analytic
structure.  Basic for model-theoretic results (such as
cell-decomposition, $o$-minimality) is an understanding of functions
of one variable defined in an arbitrary model. The basic results on
functions of one variable in the separated case are presented in
Section \ref{annuli}. These results also extend the classical
affinoid results when K is algebraically closed and complete (and of
rank 1) to the non-algebraically closed and the quasi-affinoid
cases.  The corresponding results on functions of one variable in
the real closed case are in \cite{CLR2}, section 3.  In Section 6 we
generalize the cell decompositions of \cite{CLR1} to Henselian
fields with analytic structure and also establish ``preservation of
balls'' and a Jacobian property for these structures, which are
useful for change of variables formulas for integrals.

Another axiomatic approach to analytic structures, close to the one
in \cite{DHM}, is given by Scanlon in \cite{ScanICM} where he also
studies liftings of the Frobenius.

\subsection*{Notation}
For $K$ a valued field, write $\Ko$ for its valuation ring with
maximal ideal $\Koo$, $\Kt$ its residue field and $\Kalg$ for its
algebraic closure. Usually $K$ is Henselian and then $\Kalg$ carries
a unique valuation extending that of $K$.  We denote the (multiplicative) norm on $K$ by $|\cdot|$.

 \section{Maximally Complete Fields}\label{maxcom}

 In this section we recall some of the properties of maximally complete valued fields, also called Malcev-Neumann fields.
 See \cite{Kap},  \cite{Po} and the references therein.
In the equicharacteristic case, the maximally complete field $K$
with residue field $\widetilde K$ and value group $\Gamma$ (most
often written additively, sometimes multiplicatively)
is
 \[
 \widetilde K((\Gamma))=\{\sum_{g\in I} a_g t^g\colon a_g\in\widetilde K,\ I\subset\Gamma\text{ well
 ordered}\}
 \]
In the mixed characteristic case (see \cite{Po}) we let $k$ be a
perfect field of characteristic $p$, $R$ a complete discrete
valuation ring of characteristic zero with $\widetilde R=k$ and let
$E\subset R$ be a set of multiplicative representatives of $k$ in
$R$ (i.e.~$E$ contains $0$ and one element from each nonzero
equivalence class of $\widetilde R = R/R^\circ$ where $R^\circ$ is the maximal ideal
of $R$). Choose an embedding of $\bZ\hookrightarrow\Gamma$. The
maximally complete field of mixed characteristic is constructed in
\cite{Po}. It is
\[
R((\Gamma))=\{\sum_{g\in I} a_g t^g\colon a_g\in E,\
I\subset\Gamma\text{ well ordered}\}.
\]
(The ``coefficients'' $a_g$ are added and multiplied according the
addition and multiplication in $R$, using the fact that each element
of $R$ has a unique representation of the form $\sum_{g\in\bN}{ a_g
t^g}$, $a_g\in E$.)

In both the equicharacteristic and the mixed characteristic cases we
can write an arbitrary element of $\widetilde K((\Gamma))$ or
$R((\Gamma))$ as $a=\Sigma_{g\in\Gamma} a_g t^g$ where the sum is
over all of $\Gamma$, but we require that $I_a=\{g\colon a_g\neq
0\}$ be a well ordered subset of $\Gamma$.
 We call $I_a$ the support of $a$, denoted $\supp(a)$.

The (multiplicative) norm $|\cdot|$ on $R((\Gamma))$ is
$|\sum_{g\in I}{a_gt^g}| = t^{g_0}$ where $g_0$ is the smallest $g
\in I$ with $a_g \neq 0$.

 \section{Real closed fields with analytic structure}\label{realAS}

As we mentioned in the introduction, non-archimedean real closed
fields with various analytic structures were considered in
\cite{DMM1},\cite{DMM2},\cite{DMM3}, \cite{LR4} and \cite{CLR2}. In
this section we establish a general framework for fields with real
analytic structure.

\subsection{Definitions}Let $A_{n,\alpha}, n\in\bN$, $\alpha\in\bR$, $\alpha>0$ be the ring
of real power series in $(\xi_1,\ldots,\xi_n)$ (i.e.~elements of
$\bR[[\xi_1,\ldots,\xi_n]]$) with radius of convergence $>\alpha$, and let $\Gamma$ be an ordered abelian group.

In analogy with the notation of Section \ref{maxcom} we define
$$
A_{n,\alpha}((\Gamma)):= \Big\{\sum_{g \in I}{f_gt^g}\colon f_g \in
A_{n,\alpha} \text{ and } I \subset \Gamma \text{ well
ordered}\Big\}.
$$
Clearly $A_{0,\alpha}((\Gamma)) = \bR((\Gamma)) \subset
A_{n,\alpha}((\Gamma))$.
Note that $\bR((\Gamma))$ inherits a natural order from the order on
$\bR$ and on $\Gamma$. If $\Gamma$ is divisible, then
$\bR((\Gamma))$ is real closed.
We shall denote the absolute value arising from this order by $|\cdot|$, and denote
the norm  on $\bR((\Gamma))$ by $\|\cdot\|$ to distinguish it from $|\cdot|$.  This norm extends to the
\emph{gauss-norm} on $A_{n,\alpha}((\Gamma))$, also denoted
$\|\cdot\|$, defined by $\|\sum_{g\in I}{f_gt^g}\| = t^{g_0}$ where
$g_0$ is the smallest $g \in I$ with $f_g \neq 0$.  We call
$f_{g_0}$ the \emph{top slice} of $f$, and call $f$ \emph{regular in
$\xi_n$ of degree $s$} at a point $c$ in $[-\alpha, \alpha]^n$ if $f_0$ is regular in
the classical sense in $\xi_n$ of degree $s$ at the point $c^\circ$, the closest element of
$\bR^n$ to $c$.

\begin{defn}[Real
Weierstrass system]\label{realWS} Let $\cB = \{B_{n,\alpha}\colon
n\in\bN,\ \alpha\in\bR_+\}$ be a family of $\bR$-algebras satisfying
$$
\bR[\xi_1,\dots,\xi_n] \subset B_{n,\alpha} \subset
A_{n,\alpha}((\Gamma))
$$
for each $n$ and $\alpha$. We call the family ${\cB}$ a \emph{real
Weierstrass system} if the following conditions (a) and (b) are
satisfied:

\item[(a)]
\begin{itemize}
\item[(i)] If $m\leq m'$ and $\alpha'\leq\alpha$ then $B_{m,\alpha}\subset
B_{m',\alpha'}$ and $B_{0,\alpha} = B_{0,\alpha'}$.  (We allow the
possibility that $B_0:= B_{0,\alpha}$ is a proper $\bR$-subalgebra
of $\bR((\Gamma))$.)

\item[(ii)] If $f\in B_{m+n,\alpha}$ and $f=\sum_{\mu} {\overline{f}_\mu
(\xi_1,\ldots,\xi_m)\eta^\mu}$ where
$\eta=(\xi_{m+1},\ldots,\xi_{m+n})$, then the $\overline{f}_\mu\in
B_{m,\alpha}$.\

\item[(iii)] If $f\in B_{m,\alpha}$, $a\in (-\alpha,\alpha)^n\cap \bR^n$, and
$r\in\bR_+$, then $f(r\cdot(\xi+a))$ belongs to $B_{m,\delta}$, with
$\delta:= \min((\alpha-a)/r,(\alpha+a)/r)$, and where
$f(r\cdot(\xi+a))$ is considered naturally as an element of
$A_{n,\delta}((\Gamma))$.

\item[(iv)] If $f\in B_{n,\alpha}$ then $af \in B_{n,\alpha}$ for some $a$ in
$B_0$ satisfying $\|af\| = 1$ with $\|\cdot\|$ the
gauss-norm on $A_{n,\alpha}((\Gamma))$.\
\end{itemize}

 \item[(b)]\emph{Weierstrass Division}:
\label{WPT} If $f \in B_{n,\alpha}$ with $\|f\|=1$ is regular in
$\xi_n$ of degree $s$ at 0, there is a $\delta > 0, \delta \in \bR$
such that  if $g\in B_{n,\alpha}$, then there
are unique $Q\in B_{n,\delta}$ and
$R_0(\xi'),\ldots,R_{s-1}(\xi')\in B_{n-1,\delta}$, with $\|Q\|,
\|R_i\| \leq \|g\|$, such
that
$$
g=Qf+R_0(\xi')+ R_1(\xi')\xi_n+\ldots+ R_{s-1}(\xi')\xi^{s-1}_n.
$$
\end{defn}

\begin{defn}\label{realstrong}
If a real Weierstrass system $\cB:=\{B_{n,\alpha}\colon n\in\bN,\
\alpha\in\bR_+\}$ satisfies in addition the following condition (c),
then call $\cB$ a \emph{strong}, real Weierstrass system. \
\item[(c)]If $f(\xi,\eta_1,\eta_2) \in B_{n+2,\alpha}$ there are
  $f_1(\xi,\eta_1,\eta_3), f_2(\xi,\eta_2,\eta_3)$ and
  $Q(\xi,\eta_1,\eta_2,\eta_3) \in B_{n+3,\alpha}$ such that
  $$
  f(\xi,\eta_1,\eta_2) = f_1(\xi,\eta_1,\eta_3) +\eta_2f_2(\xi,\eta_2,\eta_3) + Q\cdot(\eta_1\eta_2 - \eta_3).
$$
\end{defn}

\begin{rem}[Weierstrass Preparation]\label{realWP}Suppose that the
family $\cB = \{B_{n,\alpha}\}$ is a real Weierstrass system. If $f
\in B_{n,\alpha}$ with $\|f\|=1$ is regular in $\xi_n$ of degree $s$
at 0, then there is a $\delta \in \bR$, $\delta > 0$, such that we
can write uniquely
$$
f=[\xi_n^s+A_1(\xi')\xi_n^{s-1}+\ldots+A_s(\xi')]U(\xi)
$$
where:
\begin{eqnarray*}
&&A_1,\ldots,A_s,\in B_{n-1,\delta}, \text{ and } U\in B_{n,\delta} \text{ is a unit},\\
&&\|A_1\|,\ldots,\|A_s\|,\|U\|\leq 1 \text{ and}\\
&&\|A_1(0)\|,\ldots,\|A_s(0)\|<1 \text{ and } \|{U}(0)\|=1.\
\end{eqnarray*}
This can be seen by taking $g=\xi_n^s$ in axiom (b) and using the case $s=0$ to see that if $\|Q\| = 1$ and $Q$ is regular of degree $0$ at $0$, then $Q$ is a unit. (The special case $n=s=0$ shows that if
$a \in B_0$, $\|a\| = 1$, then $a$ is a unit in $B_0$, and hence, using (a)(iv), that $B_0$ is a subfield of $\bR((\Gamma))$.)
\end{rem}

\begin{rem} \label{realcomp}

(i) Condition (c) of Definition \ref{realstrong} is used in the
proof of $o$-minimality (see Theorem \ref{realqe}.) This condition
does not follow from the other conditions, as can be seen in Example
\ref{realex}(7) below. However, every real Weierstrass system can be
extended to a strong real Weierstrass system since the Weierstrass
system $\{A_{n,\alpha}((\Gamma))\}$ is strong, see Example
\ref{realex} (6). To prove $o$--minimality for real closed fields
with (not necessarily strong) analytic structure it is sufficient to
prove $o$--minimality for real closed fields with strong analytic
structure. The proof of $o$-minimality for real closed fields with
strong analytic structure is given in \cite{CLR2}, Section 3, where
Condition (c) is used (implicitly), roughly speaking to write $f(x,
{\frac{1}{x}}) = g(x) + {\frac{1}{x}}h({\frac{1}{x}})$.

(ii) Weierstrass Division guarantees that if $f \in
B_{n,\alpha}$ and $\xi_i$ divides $f$ as an element of $A((\Gamma))$
(i.e. $f(\xi_1,\dots,\xi_{i-1},0,\xi_{i+1}, \dots ,\xi_n) = 0$) then
there is a $\delta >0$ and an element $g \in B_{n,\delta}$ such that
$f = \xi_ig$.   It is not the case that if $f \in A_{n,\alpha}$ has
no zero in the (real) polydisc $[-\alpha,\alpha]^n$, then $f$ is a
unit in $B_{n,\alpha}$.  To see this consider $f = 1 + x^2 \in B_{1,1}$.  The
units in $B_{n,\alpha}$ are the functions that have no zeros in the
\emph{complex} polydisc $\{x \in \bC \colon |x| \leq \beta \}^n$,
for some $\alpha < \beta \in \bR$. However, if $f \in B_{n,\alpha}$
and $f_0$, the top slice of $f$ has no zero in $[-\alpha,\alpha]^n$
then $[-\alpha,\alpha]^n$ can be covered by finitely many (smaller)
polydiscs on each of which $f_0$ is a unit.  It then  follows from
conditions (a) and (b) of Definition \ref{realWS} that $f$ is a unit
in the rings $B_{n,\delta}$ on each of these smaller polydiscs. If
$f \in B_{n,\alpha}$ and $f_0$ is not $0$ at $0$, then $f$ is
regular of degree $0$ at $0$, and hence by Weierstrass Division, a
unit.  In other words, if  $f \in B_{n,\alpha}$ satisfies $f_0(0)
\neq 0$ then $f$ is a unit in $B_{n,\delta}$ for some $\delta \in
\bR_+$.

\item[(iii)] It follows from Condition (a)(ii) of Definition \ref{realWS} that if
$f\in B_{m+n,\alpha}$ and $f=\sum_{\mu} {\overline{f}_\mu
(\xi_1,\ldots,\xi_m)\eta^\mu}$, then
$$
\sum_{\substack {\mu_i>d \text{ for}\\ i=1,\cdots,n}}{\overline{f}_\mu(\xi_1,\ldots,\xi_m)\eta^\mu} \in B_{m+n,\alpha}.
$$
\item[(iv)] The proofs in \cite{CLR2}, Section 2 (in particular of
Theorem 2.5), with very minor modifications, show that the rings
$B_{n,\alpha} = A_{n,\alpha}((\Gamma))$ satisfy conditions (a) and
(b) of Definition \ref{realWS}.  Condition (c) is immediate. We
shall refer to this Weierstrass System as $\cA((\Gamma))$ and call
it the \emph{full Weierstrass system based on $\Gamma$.}  It is the
largest Weierstrass system corresponding to $\Gamma$.  The axioms
allow us to ``evaluate" elements of $A_{n,\alpha}((\Gamma))$ at
arguments $a \in B_0^n \cap [-\alpha, \alpha]^n$.  By axiom (a)(iii)
it is sufficient to consider the case $a^\circ = 0$, i.e. $a=
(a_1,\cdots,a_n)$ with the $a_i$ all infinitesimal (i.e. $\|a_i\| <
1$.)  Then, by Weierstrass Division,
$$
f = b + \sum_iQ_i \cdot (\xi_i - a_i)
$$
for a unique $b \in \bR((\Gamma))$, since $\xi_i - a_i$ is regular of degree one in $\xi_i$ at $0$.
 Let $J\subset \Gamma$ be well ordered and satisfy $J+J = J$
 and $\supp(f), \supp(a_i) \subset J$.  Then an
 induction on $J$ shows that $b = \sum_{g \in I}{f_g(a)t^g}$,
 where $f_g(\xi) = \sum_\nu a_{g\nu}\xi^\nu$ and
 $f_g(a) = \sum_\nu a_{g\nu}a^\nu$.  A similar argument
 allows us to ``compose" elements of $A((\Gamma))$.
More precisely, if $f\in
B_{m,\alpha}$ and $g_1,\ldots,g_m\in B_{n,\gamma}$ with
$(g_i)_0(0)=a_i$, where $(g_i)_0$ is the top slice of $g_i$, $|a_i|<\alpha$, and $\|g_i\|\leq 1$ for
$i=1,\ldots,m,$ then there is a $\beta\in\bR$, $\beta>0$ such that
$f(g_1(\xi),\ldots,g_m(\xi)) $ is in $B_{m,\beta}$.

\item[(v)] By (iii)  $A((\Gamma))$ is also closed under the taking of
natural derivatives in $A((\Gamma))$. More precisely, taking $m=1$
for notational convenience, composing $f\in B_{1,\alpha}$ with
$x\mapsto x+y$, writing $f(x+y)=f_0(x)+f_1(x)y+\ldots$ and defining
$f'$ as $f_1$, by construction one sees that $f'$ is in
$B_{1,\delta}$, for some $0 < \delta \leq \alpha$, giving a
derivation $B_{1,\delta}\to B_{1,\delta}$ which extends the natural
derivation on $A_{1,\delta}$.

\item[(vi)] Since all the above mentioned data are unique in
$\cA((\Gamma))$, and by the axioms exist in any Weierstrass system
$\cB$ contained in $\cA((\Gamma))$, they are also unique in $\cB$.
\end{rem}

\subsection{The Strong Noetherian Property}We prove a Strong Noetherian Property (Theorem \ref{SNP}) for real Weierstrass
systems, closely following \cite{CLR2} Lemma 2.9 and Theorem 2.10.
Since this is a fundamental property, and the ideas of the proof are
used again in section \ref{SNPsection}, we give complete proofs. We
write $B_{m,\alpha}^\circ := \{f \in B_{m,\alpha} \colon \|f\| \leq
1\}$.

The following Lemma is used to prove Theorem \ref{SNP}.
\begin{lem}\label{lemforSNP}
Let $\cB = \{B_{m,\alpha}\}$ be a real Weierstrass system and let
$f(\xi,\eta) =\sum_{\mu}\overline{f}_{\mu}(\xi)\eta^{\mu}\in
B_{m+n,\alpha}$. Then the $\overline{f}_{\mu} \in B_{m,\alpha}$,  by
Definition \ref{realWS}(a)(ii). There is an integer  $d \in \bN$,  a
constant $\beta\in\bR$, $0 < \beta \leq \alpha$,  and $g_\mu \in
B^\circ_{m+n,\beta}$ for $|\mu|<d$, such that, considering $f \in
B_{m+n,\beta}$,
$$
f = \sum_{|\mu| < d} {\overline{f}_\mu(\xi) g_\mu(\xi,\eta)}.
$$
\end{lem}
\begin{proof}
By Property (a)(iv) of Definition \ref{realWS}, we may assume that $\|f\| = 1$. Choose a $\nu_0$ such that
$\|\overline{f}_{\nu_0} \| = 1.$ Making an $\bR$-linear change of
variables (allowed by Remark
\ref{realcomp}), and shrinking $\alpha$ if necessary, we may assume
that $\overline{f}_{\nu_0}$ is regular in $\xi_m$ at $0$ of degree
$s$, say.  Write $\xi'$ for $(\xi_1,\dots,\xi_{m-1})$. By
Weierstrass Division  there is a $\beta > 0$ and there are
$Q(\xi,\eta) \in B_{m+n,\beta}$ and
$$
R(\xi,\eta) = R_0(\xi',\eta)+
\cdots + R_{s-1}(\xi',\eta)\xi_m^{s-1} \in B_{m+n-1,\beta}[\xi_m]
$$
such that
$$
f(\xi,\eta) = \overline{f}_{\nu_0}(\xi) Q(\xi,\eta) + R(\xi,\eta).
$$
By induction on m, we may write
$$
R_0 = \sum_{|\mu|<d}{\overline{R}_{0\mu}(\xi')
g_\mu(\xi',\eta)},
$$
for some $d \in \bN$, some $ \beta > 0$ and $g_\mu(\xi',\eta) \in
B_{m+n-1,\beta}^\circ.$ Writing $R = \sum_\nu{\overline
R_\nu(\xi)\eta^\nu}$, observe that each $\overline{R}_\nu$ is an
$B_{m,\beta}^\circ$-linear combination of the $\overline{f}_{\nu },$
since, taking the coefficient of $\eta^\nu$ on both sides of the
equation $f(\xi,\eta) = \overline{f}_{\nu_0}(\xi) Q(\xi,\eta) +
R(\xi,\eta),$ we have
$$
\overline{f}_\nu = \overline{f}_{\nu_0}\overline{Q}_\nu +
\overline{R}_\nu.
$$
Consider
\begin{eqnarray*}
f -  \overline{f}_{\nu_0} Q -
\sum_{|\mu|<d}{\overline{R}_{\mu}(\xi)g_\mu(\xi',\eta)}
&=:& S_1\xi_m + S_2{\xi_m}^2 + \cdots +S_{s-1} \xi_m^{s-1}\\
&=& \xi_m[S_1 + S_2\xi_m + \cdots + S_{s-1}\xi_m^{s-2}]\\
&=:& \xi_m \cdot S,\, \text{ say,}
\end{eqnarray*}
where the $S_i \in B_{m+n-1,\beta}^\circ.$ Again, observe that each
$\overline{S}_\nu$ is an $B_{m,\beta}^\circ-$linear combination of
the $\overline{f}_{\nu'}.$  Complete the proof by induction on $s$,
working with $S$ instead of $f$.
\end{proof}

\begin{thm}[Strong Noetherian Property]\label{SNP}
Let $\cB = \{B_{m,\alpha}\}$ be a real Weierstrass system and let
$f(\xi,\eta)
=\sum_{\mu}\overline{f}_{\mu}(\xi)\eta^{\mu}\in{B}_{m+n,\alpha}$.
Then the $\overline{f}_{\mu} \in {B}_{m,\alpha}$ and there is an
integer  $d \in \bN$,  a constant $\beta>0$, $ \beta\in\bR$ and
units $U_\mu(\xi,\eta) \in {B}^\circ_{m+n,\beta}$ for $|\mu|<d$,
such that, considering $f \in B_{m+n,\delta}$,
$$
f = \sum_{\mu\in J} {\overline{f}_\mu(\xi) \eta^\mu
U_\mu(\xi,\eta)},
$$
where $J$ is a subset of $\{0,1,\ldots,d\}^n$.
\end{thm}
\begin{proof}
It is sufficient to show that there are an integer $d$, a set
$J\subset\{0,1,\ldots,d\}^n$, and $g_\mu \in {B}_{m+n,\beta}^\circ$
such that
 \begin{equation}\label{eqfg}
 f = \sum_{\mu\in J}
{\overline{f}_\mu(\xi) \eta^\mu g_\mu(\xi,\eta)},
\end{equation}
since then, rearranging the sum if necessary, we may assume that
each $g_\mu$ is of the form $1 + h_\mu$ where $h_\mu \in
(\eta){B}^\circ_{m+n,\beta}.$ Shrinking $\beta$ if necessary will
guarantee that the $g_\mu$ are units. But then it is in fact
sufficient to prove (\ref{eqfg}) for $f$ replaced by
$$
f_{I_i}:=\sum_{\mu\in I_i}\overline{f}_{\mu}(\xi)\eta^{\mu}
$$
for each $I_i$ in a finite partition $\{I_i\}$ of $\NN^n$ and to
show that $f_{I_i}$ is in ${B}_{m+n,\beta}$.

By Lemma \ref{lemforSNP} there is an integer  $d \in \bN$,  a
constant $\beta \in \bR$, $ 0 < \beta \leq \alpha$ and $g_\mu \in
B^\circ_{m+n,\beta}$ for $|\mu|\leq d$, such that,
$$
f = \sum_{|\mu| \leq d} {\overline{f}_\mu(\xi) g_\mu(\xi,\eta)}.
$$
Rearranging, we may assume for $\nu,\mu\in\{1,\ldots,d\}^n$ that
$(\bar g_{\mu})_\nu$ equals $1$ if $\mu=\nu$ and that it equals $0$
otherwise.

Focus on $f_{I_1}(\xi,\eta)$, defined as above by
$$
f_{I_1}(\xi,\eta) = \sum_{\mu\in
I_1}\overline{f}_{\mu}(\xi)\eta^{\mu}
$$
with
$$
I_1:=\{1,\ldots,d\}^n\cup\{\mu\mid \mu_i\geq d\mbox{ for all $i$}\}
$$
and note that
\begin{equation}\label{eqhatf}
f_{I_1}(\xi,\eta) = \sum_{|\mu| \leq d} {\overline{f}_\mu(\xi)
g_{\mu,I_1}(\xi,\eta)}
\end{equation}
with $g_{\mu, I_1}(\xi,\eta)\in B^\circ_{m+n,\beta}$ defined by the
corresponding sum
$$
g_{\mu,I_1}(\xi,\eta) = \sum_{\nu\in
I_1}\overline{g}_{\mu,\nu}(\xi)\eta^{\nu}.
$$
The $g_{\mu,I_1}$ and $f_{I_1}$ are in
$B^\circ_{m+n,\beta}$.  This follows from the axioms and induction on $n$. (See Remark \ref{realcomp}.)

It is now clear that $g_{\mu,I_1}$ is of the form
$\eta^{\mu}(1+h_\mu)$ where $h_\mu\in (\eta)B^\circ_{m+n,\beta}$.

One now completes the proof by noting that $f-f_{I_1}$ is a finite sum of terms
of the form $f_{I_j}$ for $j>0$ and $\{I_i\}_i$ a finite partition
of $\NN^n$  and where each $f_{I_j}$ is in $B^\circ_{m+n,\beta}$ is
of the form $\eta_i^\ell q(\xi,\eta')$ where $\eta'$ is
$(\eta_1,\ldots,\eta_{i-1},\eta_{i+1},\ldots,\eta_{n})$ and $q$ is
in $\cR_{m+n-1,\beta}^\circ$. These terms can be handled by
induction on $n$.
\end{proof}

\begin{defn}\label{realansys}
Let ${\cB}=\{B_{n,\alpha}\}$ be a real Weierstrass system. Let $K$
be an ordered field containing $B_0$ as an ordered subfield.  For
each $n\in\bN$, $\alpha\in\bR$, $\alpha>1$ let $\sigma_{n,\alpha}$
be an $\bR$-algebra homomorphism from $B_{n,\alpha}$ to the ring of
$K$ valued functions on $[-1,1]^n$, compatible with the inclusions
$B_{n,\alpha}\subset B_{n,\beta}$ for $\beta<\alpha$, and respecting
the translation conditions of Definition \ref{realWS}(a)(iii). Write
$\sigma_n$ for the induced homomorphism on $\bigcup_{\alpha>1}
B_{n,\alpha}$. Suppose that the maps $\sigma_n$ satisfy
\begin{itemize}
 \item[(i)] $\sigma_0$ is the inclusion $B_0 \subset K$
 \item[(ii)]$\sigma_{m}(\xi_i)$ is the $i$-th coordinate function
on $(\Ko)^m$, $i=1,\dots,m$, and
\item[(iii)]$\sigma_{m+1}$ extends $\sigma_{m}$, where we
identify in the obvious way functions on $(\Ko)^m$ with functions on
$(\Ko)^{m+1}$ that do not depend on the last coordinate.
\end{itemize}
Then we call the family $\sigma:=\{\sigma_{n,\alpha}\}$ a \emph{real
analytic ${\cB}$-structure on $K$}.
\end{defn}

\subsection{Examples of real analytic structures}\label{realex}
\noindent
(1) \ Take  $\Gamma = \{0\}$ and $B_{n,\alpha}=A_{n,\alpha}((\{0\})) = A_{n,\alpha}$.
Then by the observations in \cite{DMM1} every complete or maximally complete real closed valued field containing $\bR$ has analytic $\cB$-structure.  In particular, $\bR$ has analytic $\cB$--structure, so $\bR$ with the subanalytic structure studied in \cite{DD} is covered by our definitions.\\

(2) \ Take $\Gamma =\bQ$ and for all $\alpha > 0$ let $B_{n,\alpha}:=B_n:=$
\[
\big\{ \sum_{\gamma_i\in \bQ} t^{\gamma_i}  p_i(\xi)\colon
p_i(\xi)\in\bR[\xi] \text{ and }\gamma_i-\varepsilon
\text{deg}(p_i)\to\infty \text{ for some } \varepsilon>0 \big\}
\]
$\subsetneq A_{n,\alpha}((\bQ))$ (in other words, the $t$-adically
overconvergent power series, $K\langle\langle\xi\rangle\rangle$,
where $K$ is the completion of the field of real Puiseux series).
Then we are in the context of considering the field $K$ with the
rings of $t$-adically overconvergent power series. This is the case
considered in \cite{LR4}.  Observe that in this example
$\bR((\Gamma)) \not\subset K$.
\\

(3) \ If in the context of (2) we take $\cB' = \{B'_{n,\alpha} \}$, with $B'_{n,\alpha} := B'_n :=$
\[
\big\{ \sum_{\gamma_i\in \bZ[{\frac{1}{m}}]}{ t^{\gamma_i}
p_i(\xi)}\colon p_i(\xi)\in\bR[\xi] \text{ and }\gamma_i-\varepsilon
\text{deg}(p_i)\to\infty \text{ for some } \varepsilon>0, m \in \bN
\big\}
\]
we may consider the field of real Puiseux series (rather than its
completion) with real analytic $\cB'$-structure.  Hence this is an
example of a field with real analytic structure that is not
complete.
\\

(4) \ Take $\Gamma= \bQ$ and let
\[
B_{n,\alpha}=\big\{ \sum t^{\gamma_i} f_{\gamma_i} (\xi)\colon
f_{\gamma_i}\in A_{n,\alpha}, \bQ \ni \gamma_i \to \infty\big\}
\]
$\subsetneq A_{n,\alpha}((\bQ))$. Then we are in the context of
Section 2 of \cite{CLR2} where we considered the completion of the
field of Puiseux series with real analytic $\cB$-structure. If we
take
$$
B'_{n,\alpha}=\big\{ \sum_{\gamma_i\in \bZ[{1 \over m}]} t^{\gamma_i} f_{\gamma_i} (\xi)\colon
f_{\gamma_i}\in A_{n,\alpha}, \gamma_i \to \infty \text{, for some } m \in \bN \big\}
$$
and $\cB':=\{B'_{n,\alpha}\}$, then the field of real Puiseux series has analytic $\cB'$-structure.\\

(5) \ Take $\Gamma =\bQ^n$ with $\bQ^n$ ordered lexicographically,  and
\[
B_{n,\alpha}=\big\{\sum_{g\in I}t^g f_g(\xi)\colon f_g\in
A_{n,\alpha}, I \subset \bQ^n\text{ well ordered}\big\}.
\]
Then we are in the context of Section 4 of \cite{CLR2}. Note that even for $n=1$ there are more functions in this analytic structure than the one of example (4).\\

(6) \ (cf. Remark \ref{realcomp}.) Take $\Gamma$ arbitrary and $B_{n,\alpha}=A_{n,\alpha}((\Gamma))$.
Then $\cA((\Gamma)):=\{B_{n,\alpha}\}$ is a strong real Weierstrass
system. The proofs of \cite{CLR2} Section 2, with very minor
modifications, show that this family is a strong real Weierstrass
system and that $\bR((\Gamma))$ has real analytic
$\cA((\Gamma))$-structure. We call this the \emph{full} real
analytic
$\bR((\Gamma))$-structure.\\

(7) \ If we take $\Gamma = \{0\}$ and $B_{n,\alpha}$ the ring of \emph{algebraic} power series with radius of convergence $> \alpha$, then the $B_{n,\alpha}$ satisfy conditions (a) and (b) of Definition \ref{realWS}, but not condition (c).  Indeed, if $f = \sum a_{ij}\eta_1^i\eta_2^j$ and $f_1$ and $f_2$ are as in condition (c), then $f_1(0,\eta_3) = \sum a_{ii}\eta_3^i$ is the ``diagonal" of $f$.  Take
$$
f = (1-4\eta_1)^{- {1\over 2} }(1-4\eta_2)^{- {1\over 2} } \in A_{2,\alpha}
$$
for $\alpha < 1/4$.  Then $f$ is algebraic, but the diagonal of $f$ is $\sum_i \binom{2i}{ i}^2\eta_3^i$, which is the elliptic integral ${2 \over \pi} \int_{0}^\frac{\pi}{2} \frac{dt}{\sqrt{1-16\eta_3sin^2t}}$ and \emph{not} algebraic. (Cf. \cite{DL}.)\\

\subsection{Model theoretic results}
\begin{defn}\label{reallanguage}
Let $\cB:=\{B_{n,\alpha}\}$ be a real Weierstrass system. Define the
language $\cL_\cB$ to be the language of ordered fields, $\langle
+,-,\cdot,^{-1},<\rangle$, together with symbols for all elements of
$B_{n,\alpha}$ for all $n$ and $\alpha>1$.
\end{defn}

\begin{rem}\label{scale}
Let $\cB:=\{B_{n,\alpha}\}$ be a real Weierstrass system and
$\sigma$ a real analytic ${\cB}$-structure on a field $K$. Then for
each $\alpha$ with $1\geq \alpha>0$, and each $f\in B_{n,\alpha}$,
$f$ yields a unique definable function (given by a term) from, for
example, $[-\alpha/2,\alpha/2]$ to $K$, by rescaling as in axiom
(a).
\end{rem}

\begin{thm}[Quantifier Elimination and $o$-minimality]\label{realqe}
Let $K$ be a real closed field with a real analytic $\cB$-structure
for some real Weierstrass system ${\cB}$. Then $K$ has quantifier
elimination in the language $\cL_\cB$ (with the natural
interpretation) and is $o$-minimal in $\cL_\cB$.
\end{thm}
\begin{proof}
The proof of the quantifier elimination for $\cL_\cB$ is a minor
modification of the last section of \cite{DD}; see \cite{DMM1},
\cite{CLR2}, \cite{LR4} for similar such modifications. By this
quantifier elimination, the natural theory of real closed fields
with real analytic $\cB$-structure is complete in the language
$\cL_\cB$. Hence, it is enough to prove $o$-minimality for an
expansion of one model. We prove $o$-minimality of $K:=\bR((G))$ in
the language $\cL_{\cB'}$ with $G$ the divisible closure of $\Gamma$
and where we take $B'_{n,\alpha}:=A_{n,\alpha}((G))$ to form our
real Weierstrass system $\cB'$, cf. Example \ref{realex}(6). Since
$B_{n,\alpha}\subset B'_{n,\alpha}$ for each $n,\alpha$, it is
enough to prove that $K$ is $o$-minimal in $\cL_{\cB'}$.
 Now one can analyze definable functions in one variable (using
annuli and the fact that $\cB'$ is strong) exactly as in
\cite{CLR2}, Section 3, yielding $o$-minimality as in \cite{CLR2}.
The strongness assumption (Definition \ref{realstrong}) is used
implicitly in \cite{CLR2} Section 3. (This analysis is similar to,
but simpler than the one given in section \ref{annuli} below for the
separated case, since one is much closer to the algebraically closed
case.)
\end{proof}

 \section{Henselian fields with analytic structure}\label{sepanstruct}

In this section we present a general theory of (Henselian) fields
with analytic structure, generalizing the presentations in
\cite{vdD}, \cite{DHM}, \cite{LR3} and \cite{CLR1}. We distinguish
between separated analytic structures and strictly convergent
analytic structures.  In section \ref{sepan} we present  the
definitions of \emph{separated Weierstrass systems} and
\emph{Henselian fields with separated analytic structure.}  In
section \ref{SNPsection} we discuss the \emph{strong noetherian
property for separated Weierstrass systems}, and in section
\ref{strictlyconv} we define \emph{strictly convergent Weierstrass
systems} and \emph{fields with strictly convergent analytic
structure}.  In section \ref{sepex} we give examples of Henselian
fields with analytic structure.  In section \ref{properties} we
develop properties of Henselian fields with analytic structure.   In
section \ref{annuli} we will analyze one variable terms (functions)
of the theory of Henselian fields with separated analytic structure.

In this whole section, let $A$ be a commutative ring with unit with
a fixed proper ideal $I $ of $A$, where proper means $I\not=A$.
(Hence, $I=0$ is allowed.) Write $A^\circ:=I$ and $\widetilde {A} :=
A/I$.

 \subsection{Henselian fields with separated analytic structure}\label{sepan}

We consider polynomial rings and power series rings in two kinds of
variables, written $\xi_i$ and $\rho_j$, which play  different
roles. Roughly speaking, we think of the $\xi_i$ as varying over the
valuation ring (or the closed unit disc) $\Ko$ of a valued field
$K$, and the variables $\rho_j$ as varying over  the maximal ideal
(or the open unit disc) $\Koo$ of $\Ko$. We use the terminology
``separated" in analogy to the rings of separated power series whose
theory was developed in \cite{LL1}, \cite{B}, \cite{LL2}, and
especially \cite{LR1}. Let $A$ and $I \varsubsetneq A$ be as at the
beginning of section \ref{sepanstruct}.

A first instance where these variables play a different roles is:

 \begin{defn}[Regular]
\label{regular} Let $f$ be a power series in
$A[[\xi_1,\ldots,\xi_m,\rho_1,\ldots,\rho_n]]$, and let $J$ be the
ideal
$$
J:=\{\sum_{\mu,\nu} a_{\mu,\nu} \xi^\mu\rho^\nu \in A[[\xi,\rho]] :
a_{\mu,\nu} \in I\}
$$
of $A[[\xi,\rho]]$, where $A$ and $I$ are as at the beginning of
section \ref{sepanstruct}.

 \item[(i)]
$f$ is called {\em regular in $\xi_m$ of degree $d$} when $f$ is
congruent in $A[[\xi,\rho]]$ to a monic polynomial in $\xi_m$ of
degree $d$, modulo the ideal $B_1 \subset A[[\xi,\rho]]$, with
$$
B_1:=J+(\rho)A[[\xi,\rho]].$$
 \item[(ii)]
$f$ is called {\em regular in $\rho_n$ of degree $d$} when $f$ is
congruent in $A[[\xi,\rho]]$ to $\rho_n^d$, modulo the ideal
$B_2 \subset A[[\xi,\rho]]$, with
$$
B_2 := J+(\rho_1,\ldots,\rho_{n-1},\rho_n^{d+1})A[[\xi,\rho]].
$$
\end{defn}

 \begin{defn}[$(A,I)$--System]\label{defAmn}
  Let $m \leq m'$ and $n \leq n'$ be natural numbers, and $\xi=(\xi_1,\ldots,\xi_m)$, $\xi'=(\xi_1,\ldots,\xi_{m'})$, $\xi''=(\xi_{m+1},\ldots,\xi_{m'})$, $\rho=(\rho_1,\ldots,\rho_{n})$,
   $\rho'=(\rho_1,\ldots,\rho_{n'})$, and
 $\rho''=(\rho_{n+1},\ldots,\rho_{n'})$ be variables.
A system $\cA=\{A_{m,n}\}_{m,n \in \bN}$ of $A$-algebras $A_{m,n}$,
satisfying,  for all $m \leq m'$ and $n \leq n'$,
 \begin{itemize}
\item[(i)]$ A_{0,0} = A,$\
 \item[(ii)]
$A_{m,n}\subset A[[\xi,\rho]],$\
 \item[(iii)]
$ A_{m,n} [\xi'', \rho''] \subset A_{m',n'},$\
 \item[(iv)]the image $(A_{m,n})\;\widetilde{}$
 of $A_{m,n}$ under the residue map $\, \widetilde{}:A[[\xi,\rho]] \to \widetilde {A}[[\xi,\rho]]$ is a subring of $
 \widetilde{A}[\xi][[\rho]]$, and\
\item[(v)] if
 $f \in A_{m',n'}$, say $f =
\sum_{\mu\nu}\overline{f}_{\mu\nu}(\xi,\rho)(\xi'')^\mu(\rho'')^\nu$,
then the $\overline{f}_{\mu\nu}$ are in $A_{m,n},$\
\end{itemize}
is called a \emph{separated
 $(A,I)$-system.}

\end{defn}

\begin{defn}[Pre-Weierstrass system]\label{good}Let $\cA=\{A_{m,n}\}_{m,n \in \bN}$ be a separated $(A,I)$-system.
Then $\cA$ is called a \emph{separated pre-Weierstrass system} when
the two usual Weierstrass Division Theorems hold in the $A_{m,n}$,
namely, for $f,g\in A_{m,n}$:

{\em (a)}  If $f$ is regular in $\xi_m$ of degree $d$, then there
exist uniquely determined elements $q\in A_{m,n}$ and $r\in
A_{m-1,n}[\xi_m]$ of degree at most $d-1$ such that $g=qf+r$.

{\em (b)} If $f$ is regular in $\rho_n$ of degree $d$, then there
exist uniquely determined elements $q\in A_{m,n}$ and $r\in
A_{m,n-1}[\rho_n]$ of degree at most $d-1$ such that $g=qf+r$.

Sometimes $\cA$ is said to be \emph{over $(A,I)$} to specify that
$\cA$ is a separated $(A,I)$-system.

\end{defn}

In fact, since we allow $A$ to be quite general, we need to be able to work locally, using rings of
fractions:

\begin{defn}[Rings of fractions]\label{frac}
Let $\cA=\{A_{m,n}\}_{m,n \in \bN}$ be a separated
 $(A,I)$-system. Inductively define the concept that an $A$-algebra $C$ is a \emph{ring of
$\cA$-fractions with proper ideal $C^\circ$} and with \emph{rings
$C_{m,n}$ of separated power series over $C$} by

\item[(i)] The ring $A$ is
a ring of $\cA$-fractions with ideal $A^\circ=I$ and with rings of
separated power series the $A_{m,n}$ from the system $\cA$.

\item[(ii)]
If $B$ is a  ring of $\cA$-fractions and $d$ in $B$ satisfies
$C^\circ \not=C$, with
$$C:= B/dB,$$
$$C^\circ := B^\circ/dB,$$
 then $C$ is a ring of $\cA$-fractions with proper ideal $C^\circ$ and
$C_{m,n}:= B_{m,n}/dB_{m,n}$.

\item[(iii)] If $B$ is a ring of $\cA$-fractions and $c,d$ in $B$
satisfy $C^\circ\not = C$ with
 $$C=B\langle {c \over d} \rangle := B{}_{1,0}/(d\xi_1 -
c),$$
 $$C^\circ:=(B^\circ)B\langle {c \over d}\rangle$$
 then $C$ is a ring of $\cA$-fractions with proper ideal
$C^\circ$ and $C_{m,n}:=
 B_{m+1,n}/(d\xi_1 - c)$.

\item[(iv)] If $B$ is a ring of $\cA$-fractions and $c,d$ in $B$ satisfy $C^\circ\not = C$, with $$C= B[[{c \over d}]]_s:=B_{0,1}/(d\rho_1 -
c),$$ $$C^\circ:=(B^\circ, \rho_1) B_{0,1}/(d\rho_1 - c),$$ and
$(B^\circ, \rho_1)$ the ideal generated by $B^\circ$ and $\rho_1$,
then $C$ is a ring of $\cA$-fractions with proper ideal $C^\circ$
and $C_{m,n}:=B_{m,n+1}/(d\rho_1 - c)$.

In all cases define $C^\circ_{m,n}$ as $(C^\circ, \rho)C_{m,n}$.
 \end{defn}

\begin{defn}[Weierstrass system]
\label{goodgood} Let $\cA$ be a separated pre-Weierstrass system.
 Call $\cA$ a \emph{separated Weierstrass system} if it satisfies (c)
  below for any ring $C$ of $\cA$-fractions.

\item[(c)] If $f = \sum_{\mu,\nu}\overline{c}_{\mu\nu}\xi^\mu\rho^\nu$
is in $C_{m,n}$
 with the $\overline{c}_{\mu\nu} \in C$, then there is a finite set
 $J \subset \bN^{m+n}$ and for each $(\mu,\nu) \in J$ there is a
 $g_{\mu\nu} \in C_{m,n}^\circ$ such that
 $$
 f = \sum_{(\mu,\nu) \in J} \overline{c}_{\mu\nu}\xi^\mu\rho^\nu(1+g_{\mu\nu}).
 $$

Sometimes $\cA$ is said to be \emph{over $(A,I)$} to specify that
$\cA$ is a separated $(A,I)$-system.
 \end{defn}

\begin{defn}[Analytic structure]
\label{sep} Let $\cA = \{A_{m,n}\}$ be  a separated Weierstrass
system, and let $K$ be a valued field. A \emph{separated analytic
$\cA$-structure on $K$} is a collection of homomorphisms
$\{\sigma_{m,n}\}_{m,n\in\NN}$, such that, for each $m,n\geq 0$,
$\sigma_{m,n}$ is a homomorphism from $A_{m,n}$ to the ring of
$K^\circ$ valued functions on $(K^\circ)^m\times (K^{\circ\circ})^n$
and such that:
 \begin{itemize}
 \item[(1)] $I\subset\sigma_{0,0}^{-1} (K^{\circ\circ})$,
 \item[(2)]$\sigma_{m,n}(\xi_i)=$ the $i$-th coordinate function
on $(\Ko)^m\times(\Koo)^n$, $i=1,\dots,m$, and
$\sigma_{m,n}(\rho_j)=$ the $(m+j)$-th coordinate function on
$(\Ko)^m\times(\Koo)^n$, $j=1,\dots,n$, and
 \item[(3)]$\sigma_{m,n+1}$ extends $\sigma_{m,n}$, where we
identify in the obvious way functions on $(\Ko)^m\times(\Koo)^n$
with functions on $(\Ko)^m\times(\Koo)^{n+1}$ that do not depend on
the last coordinate, and $\sigma_{m+1,n}$ extends $\sigma_{m,n}$
similarly.
 \end{itemize}
 \end{defn}

We have given the basic definitions of this section. A reader who
wants to skip the proofs and the analysis of analytic structures can
proceed directly with the examples in section \ref{sepex}, with
section \ref{properties}, and with the (model theoretic) results of
section \ref{seccell}, having as well a look at section
\ref{annuli}.
In section \ref{strictlyconv} we give the basic definitions of
strictly convergent analytic structures, which are simpler but less
powerful.

\begin{rem}[Units]\label{unit}
If $\cA=\{A_{m,n}\}_{m,n \in \bN}$ is a separated  pre-Weierstrass
system and $f = 1 + g \in A_{m,n}$, with $g \in A_{m,n}^\circ$, then
$f$ is regular of degree $0$ and, by Weierstrass Division, $f$ is a
unit in $A_{m,n}$. Moreover, $f\in A_{m,n}$ is a unit if and only if
$f$ is of the form $c+g$ for some unit $c\in A$ and some $g\in
A_{m,n}^\circ$. Indeed, since $f$ is a unit there exists $h$ in
$A_{m,n}$ such that $fh=1$, hence, $\widetilde f\widetilde h=1$ in
$\widetilde A$, with $\widetilde{}$ as in Definition \ref{defAmn}
(iv). Hence, $\widetilde f$ is a unit in $\widetilde A_{m,n}$, hence
$\widetilde f$ is a unit in $\widetilde A[x][[\rho]]$ and
$\widetilde f$ is in $ \widetilde A+\rho \widetilde A[x][[\rho]]$.
\end{rem}

\begin{rem}[Noetherianness]

(i) If $I$ is a finitely generated ideal, then, in Definition
\ref{regular}, one has that $B_1 = (I, \rho)A[[\xi,\rho]]$ and $B_2
= (I,\rho_1,\dots,\rho_{n-1}, \rho_n^{d+1})A[[\xi,\rho]]$.

(ii) While we did not require $A$ to be Noetherian, (see
\ref{sepex}(7) for an example) it follows from (c) of Definition
\ref{goodgood} that if $\cA$ is a separated Weierstrass system and
$f = \sum_{\mu,\nu}a_{\mu\nu}\xi^\mu\rho^\nu \in A_{m,n}$ then the
ideal
of $A $
generated by the $a_{\mu \nu}$ is
finitely generated.

(iii)We refer to condition (c) of Definition \ref{goodgood} as a
\emph{Strong Noetherian Property} as it implies that all the
coefficients of $f$ can be written as linear combinations of
finitely many, and if the coefficient is ``small", the corresponding
coefficients of the linear combination are also small.  Example
\ref{sepex} (7) below shows that this does not require the ring $A$
to be Noetherian.  Similarly, we refer to Theorem \ref{SepSNP} as a
Strong Noetherian Property, even though it does not imply that the
rings $A_{m,n}$ are Noetherian.
\end{rem}

As usual, Weierstrass preparation is a consequence of Weierstrass Division.

\begin{rem}[Weierstrass Preparation]\label{weierstrass_preparation}
Let the family $\{A_{m,n}\}$ be a separated Weierstrass system. With
the notation of Definition \ref{good}, we obtain the following for
$f$ in $A_{m,n}$:

{\em (i)} If $f$ is regular in $\xi_m$ of degree $d$, then there
exist: a unique unit $u$ of $A_{m,n}$ and a unique monic polynomial
$P\in A_{m-1,n}[\xi_m]$ of degree $d$ such that $f=u\cdot P$.

{\em (ii)} If $f$ is regular in $\rho_n$ of degree $d$, then there
exist: a unique unit $u$ of $A_{m,n}$ and a unique monic polynomial
$P\in A_{m,n-1}[\rho_n]$ of degree $d$ such that $f=u\cdot P$; in
addition, $P$ is regular in $\rho_n$ of degree $d$.

This can be seen, by dividing $\xi_m^d$ (respectively, $\rho_n^d$)
by $f\in A_{m,n}$, as in \cite{LR1}, Corollary~2.3.3.
\end{rem}

For $F$ a valued field and $f=\sum_{\mu,\nu}
c_{\mu\nu}\xi^\mu\rho^\nu$ in $F[[\xi,\rho]]$, the \emph{gauss-norm}
of $f$ is written $\|f\|$ and defined as $ \sup_{\mu\nu} |
c_{\mu\nu}| $ if this supremum
exists in $|F|$
and is not defined otherwise.

\begin{rem}[Gauss-norm]\label{Gauss-norm}
In the case that $A=F^\circ$ and $I = F^{\circ\circ}$ with $F$ a
valued field, condition (c) of Definition \ref{goodgood} guarantees
that the gauss-norm on $A_{m,n}$ is defined and moreover, if $f\in
A_{m,n}$ is nonzero then there is $c \in F$ such that $cf \in
A_{m,n}$ and $\|cf\| = 1$. Conversely, in this case ($A=F^\circ,
I=F^{\circ\circ}$) if the $A_{m,n}$ satisfy definitions \ref{defAmn}
and \ref{good}, and for every $0 \neq f \in A_{m,n}$ there is an
element $c \in F$ such that $cf \in A_{m,n}$ and $\|cf\| = 1$, then
condition (c) of Definition \ref{goodgood} follows from conditions
(a), (b) and Definition \ref{defAmn}.  This fact (which we do not
use) can be proved along the lines of the proof of Theorem
\ref{SepSNP}.
\end{rem}

\begin{rem}
(i) As in the real case, it follows by Weierstrass Division that if
$\xi_i$ (respectively $\rho_j$) divides $f \in A_{m,n}$ in
$A[[\xi,\rho]]$, then $\xi_i$ (respectively $\rho_j$) divides $f$ in
$A_{m,n}$.

(ii) As in the real case, if $f =
\sum_{\mu\nu}\overline{f}_{\mu\nu}(\xi,\rho)(\xi'')^\mu(\rho'')^\nu
\in A_{m',n'}$ and
$$
I_1:= \{(\mu,\nu) \colon \mu_i \geq d, \nu_j \geq d, \text{ for all
} i,j\},
$$
then
$$
f_1:= \sum_{(\mu,\nu) \in
I_1}\overline{f}_{\mu\nu}(\xi,\rho)(\xi'')^\mu(\rho'')^\nu
$$
is an element of $A_{m',n'}$.  Similarly, for each $i,j$ and $\ell$,
$$
\sum_{\mu,\nu \text{ with } \mu_i =
\ell}\overline{f}_{\mu\nu}(\xi,\rho)(\xi'')^\mu(\rho'')^\nu \quad
\text{  and  } \sum_{\mu,\nu \text{ with } \nu_j =
\ell}\overline{f}_{\mu\nu}(\xi,\rho)(\xi'')^\mu(\rho'')^\nu
$$
are elements of $A_{m',n'}$.  By part (i) above, we can divide
these series by $\xi_i^\ell$ (respectively, $\rho_j^\ell$.)

(iii) Weierstrass Division guarantees that the rings $A_{m,n}$ are
closed under Weierstrass changes of variables among the $\xi_i$ and
among the $\rho_j$, but not in general under changes of variables
that mix the $\xi_i$ and the $\rho_j$.
\end{rem}

In section \ref{Strongsystems} we will consider a variant of
separated Weierstrass systems, by requiring some additional
conditions. We will term these \emph{``strong separated Weierstrass
systems"}.

\begin{lem}\label{(vii)}
If $\cA$ is a (strong) separated Weierstrass system, and $C$ is a
ring of $\cA$-fractions, then the family $\{C_{m,n}\}$ is  a
separated $(C,C^\circ)$-system which is a (strong) separated
Weierstrass system (cf.~section \ref{Strongsystems} for ``strong").
\end{lem}
\begin{proof}
Since axiom (c) for $\{C_{m,n}\}$ follows at once from (c) for $\cA$
and from (v) of Definition \ref{defAmn}, only axioms (a) and (b) for
the $C_{m,n}$ need proof, which we do by induction on the definition
of $C$. Suppose that $C$ is $A\langle {c \over d} \rangle$ and $f
\in C_{m,n}$ is regular in $\xi_m$ of degree $s$. Then
$$
f \equiv  \xi_m^s + F_1\xi_m^{s-1} + \cdots + F_s \text{ mod } (I,
\rho, c-\eta d).
$$
Let $F \in C_{m+1,n}$ satisfy $f \equiv F$ mod $(c-\eta d)$. It may
be that $F$ is not regular in $\xi_m$ of degree $s$.  However, using
the condition $(A_{m+1,n})\; \widetilde{} \subset
\widetilde{A}[\xi,\eta][[\rho]]$ of Definition \ref{defAmn}, we see
that there is a finite sum $\sum_{i=1}^\ell {(c-\eta d)b_{i}
\xi_m^{s+i}}$ such that, taking $G = F -  \sum_{i=1}^\ell {(c-\eta
d)b_{i}\xi_m^{s+i}}$, we have $f \equiv G$ mod $(I, \rho, c-\eta d)$
and $G$ regular in $\xi_m$ of degree $s$.  The other case is
similar.
\end{proof}

\subsection{The Strong Noetherian property for separated Weierstrass
systems}\label{SNPsection}

Axiom (c) of Definition \ref{goodgood} is a kind of Noetherian
property. In order to fully exploit it towards quantifier
elimination, we have to work locally using the rings of fractions,
defined in the previous subsection, and Laurent rings, defined
below. The aim of this subsection is Theorem \ref{SepSNP}, which
uses the full strength of the formalisms of section \ref{sepan} and
of this section.  (See also Remark \ref{remSepSNP}.)

 In this subsection, $\cA$ is a separated Weierstrass system, as
always over $(A,I)$.

First we elaborate some more on rings of $\cA$-fractions.

\begin{defn}[Defining formula]\label{defform}
Let $C$ be a ring of $\cA$-fractions. Call an expression
 a \emph{defining formula for $C$} when it can
inductively be obtained by the following steps

\item[(i)] The expression $(1= 1)$
is a defining formula for $A$.

\item[(ii)] In case (ii) of Definition \ref{frac}, if $\varphi_B$ is a defining
formula for $B$, then
$$\phi_{B} \wedge (d=0)$$
is a defining formula for $C$.

\item[(iii)] In case (iii) of Definition \ref{frac}, if
$\varphi_B$ is a defining formula for $B$,
then
$$
\phi_{B}\wedge (|c|\leq|d|) \wedge (d  \neq
 0)$$
is a defining formula for $C$.

\item[(iv)] In case (iv) of Definition \ref{frac},
if $\varphi_B$ is a defining formula for $B$, then
 $$
\phi_B \wedge (|c| < |d|)\wedge (c\neq0)
 $$
 is a defining formula for
$C$.
\end{defn}

\begin{defn}[System of rings of fractions]\label{sysfrac}
Let $\cA$ be a separated
 $(A,I)$-system. If $C$ is a ring of $\cA$-fractions and $c, d
\in C$, then let $\cD_{c,d}(C)$ be the set of rings of fractions
among $C/dC, C\langle {c \over d} \rangle, C[[{d \over c}]]_s $ if
this set is nonempty and let $\cD_{c,d}(C)$ be $\{C\}$ otherwise.
Define the concept of a \emph{system of rings of $\cA$-fractions}
inductively as follows. $\cF = \{A\}$ is a system of rings of
$\cA$-fractions. If $\cF_0$ is a system of rings of $\cA$-fractions
and $C \in \cF_0$, $ c, d \in C$, then
$$\cF:=(\cF_0 \setminus \{C\}) \cup \cD_{c,d}(C)$$ is a system of rings of
$\cA$-fractions.
 \end{defn}

 \begin{defn}\label{cV}
 Let $\cV$ be the theory of the valuation rings of valued fields in the language of valued rings. For
 any commutative
 ring $C$ with unit and with fixed proper ideal $C^\circ$ (that is, $C^\circ\not =C$),
 define
 $$
 \cV(C) := \cV \cup \{|a|\leq 1: a\in C\} \cup \{|b|<1: b \in C^\circ\}.
 $$
\end{defn}

\begin{lemdef}\label{lemCV}
 With the notation from Definition \ref{cV} and with a system $\cF$ of rings of $\cA$-fractions, it follows that
 $$
 \cV(A) \vdash \bigvee_{C \in \cF} \phi_C
 $$
 and
 $$
  \cV(A) \vdash \bigwedge_{C \neq C' \in \cF} \neg(\phi_C \wedge \phi_{C'}).
  $$
  If moreover $\cA$ is a separated Weierstrass system
   and $\sigma$ is a separated analytic $\cA$-structure on $K$, then there is exactly one $C \in \cF$ such that
  $\Ko \models \varphi_C$ under the interpretation provided by
  $\sigma$. We call such $\varphi_C$ \emph{compatible with
  $\sigma$}. We call a ring of $\cA$-fractions $C$ compatible with
  $\sigma $ when it has a defining formula that is compatible with
  $\sigma$.
\end{lemdef}
\begin{proof}
This holds by the definitions.
\end{proof}

Next we generalize the notion of rings of $\cA$-fractions to that of
Laurent rings. This notion will mainly be used for Theorem
\ref{SepSNP}.

 \begin{defn}[Laurent rings]\label{Laurentfrac}
Let $C$ be a ring of $\cA$-fractions. Inductively define the concept
that a $C_{m,n}$-algebra $C'$ is a \emph{Laurent ring over
$C_{m,n}$} with \emph{proper ideal $C'{}^\circ$}
as follows.

\item[(i)] $C_{m,n}$ is
a Laurent ring over $C_{m,n}$ with proper ideal $C^\circ_{m,n}$
(cf.~Definition \ref{frac}).

\item[(ii)]
Let $J$ be an ideal of $C_{m+M,n+N}$. If $B = C_{m+M,n+N}/J$ is a
Laurent ring over $C_{m,n}$, if $f\in C_{m+M,n+N}$, and if
$C'{}^\circ\not = C'$, with
 $$
C'= B\langle {1\over f} \rangle := C_{m+M+1,n+N}/(J, \xi_{m+M+1}f -
1),
$$
$$
C'{}^\circ:= (C^\circ + (\rho_1,\ldots,\rho_{n+N}))
C_{m+M+1,n+N}/(J, \xi_{m+M+1}f - 1),
$$
then $C'$ is a Laurent ring over $C_{m,n}$ with proper ideal
$C'{}^\circ$.

\item[(iii)] Let $J$ be an ideal of $C_{m+M,n+N}$.
If $B= C_{m+M,n+N}/J$ is a Laurent ring over $C_{m,n}$, if $f\in
C_{m+M,n+N}$, and if $C'{}^\circ\not = C'$, with
$$
C'=B[[f]]_s := C_{m+M,n+N+1}/(J, \rho_{n+N+1} - f),
$$
$$
C'{}^\circ := (C^\circ + (\rho_1,\ldots,\rho_{n+N+1}))
C_{m+M,n+N+1}/(J, \rho_{n+N+1} - f),
$$
then $C'$ is a Laurent ring over $C_{m,n}$ with proper ideal
$C'{}^\circ$.
\end{defn}
\begin{defn}\label{not} Let $C$ be a ring of $\cA$-fractions.
For any Laurent ring $C'$ over $C_{m,n}$ there exists an ideal $J$
of some $C_{M+m,N+n}$ such that $C' = C_{M+m,N+n}/J$. Define then
 $$C'_{m_1,n_1} = C_{M+m+m_1,
N+n+n_1}/JC_{M+m+m_1, N+n+n_1}
 $$
and
 $$(C'_{m_1,n_1})^\circ :=
(C^\circ , (\rho_1,\ldots,\rho_{N+n+n_1}))C'_{m_1,n_1}.
 $$
 \end{defn}

The following is the analogue for Laurent rings of Lemma
\ref{(vii)}. We will not need it.  The proof is similar to that of Lemma \ref{(vii)}.
\begin{lem}\label{(vii)bis}
If $\cA$ is a  (strong) separated Weierstrass system, and $B$ is a
Laurent ring over $C_{m,n}$ for some ring of $\cA$-fractions $C$,
then the family $\{B_{m,n}\}$ is a (strong) separated Weierstrass
system over $(B,B^\circ)$ (cf.~section \ref{Strongsystems} for the definition of
``strong").
\end{lem}

 \begin{defn}[Defining formula]\label{Laurentfracdefform} Let $C$ be a ring of $\cA$-fractions.
Call an expression a \emph{defining formula for $B$}, with $B$ a
Laurent ring over $C_{m,n}$, when it can inductively be obtained by
the following steps

\item[(i)] the expression $(1=1)$
is a defining formula for $C_{m,n}$.

\item[(ii)] In case (ii) of Definition \ref{Laurentfrac}, if
$\phi_{C'}$ is a defining formula of $C'$, then
$$
\phi_{C'} \wedge (|f| \geq 1)
$$
is a defining formula for $B:=C'\langle{1 \over f} \rangle$.

\item[(iii)]
In case (iii) of Definition \ref{Laurentfrac}, if $\phi_{C'}$ is a
defining formula of $C'$, then
$$
\phi_{C'} \wedge (|f| < 1)
$$
is a defining formula for $B:=C'[[f]]_s$.

 \end{defn}

 \begin{defn}[Covering family]\label{coveringf} Let $C$ be a ring of $\cA$-fractions.
 Note that the theory $\cV(C)$ is well defined
by Definition \ref{cV} and Lemma \ref{(vii)}. We call a finite
family $\cF$ of Laurent rings over $C_{m,n}$ a \emph{covering
family} if
 $$
\cV(C) \vdash \bigvee_{B\in \cF}\varphi_{B}.
 $$
 The family is a \emph{disjointly
covering family} if in addition $\neg(\phi_{B} \wedge \phi_{B'})$ is
a theorem of $\cV(C)$ for all $B\neq B' \in \cF$.
 \end{defn}

By construction, we have the following:

\begin{lem}
Let $\sigma$ be a separated analytic $\cA$--structure on $K$ and let $C$ be
a ring of $\cA$-fractions compatible with $\sigma$. Any defining formula
$\varphi_B$ of any Laurent ring $B$ over $C_{m,n}$ defines in a
natural way a subset $X_{\varphi_B}$ of $(\Ko)^{m}\times(\Koo)^{n}$,
which may be empty.
 Moreover, for any  covering family $\cF$ of Laurent rings $B$ over
$C_{m,n}$ with formulas $\varphi_B$, the union of the sets
$X_{\varphi_B}$ equals $(\Ko)^{m}\times(\Koo)^{n}$.
\end{lem}

 \begin{defn}[Units]\label{strongu}
Let $C$ be a Laurent ring over $A_{m,n}$. (Any Laurent ring over a
ring of $\cA$-fractions is a Laurent ring over some $A_{m,n}$.) We
call $f \in C$ a \emph{$C$-unit} if
$$
\cV(C), \varphi_C \vdash |f| \geq 1.
$$
 \end{defn}

We recall the following definition from \cite{LL2} section 3.12,
which is used in several proofs.

\begin{defn}[Preregular]\label{prereg}
Let $C$ be a Laurent ring over $A_{m,n}$ and,
using the notation of Definition \ref{defAmn}, let $f =
\sum_{\mu\nu}{c_{\mu\nu}(\xi'')^\mu(\rho'')^\nu} \in C_{m'-m,n'-n}$
where the $c_{\mu\nu} \in C$.  We call $f$ \emph{preregular in
$(\xi'',\rho'')$ of degree $(\mu_0,\nu_0)$} if $c_{\mu_0 \nu_0} = 1$
and $c_{\mu \nu} \in (C_{m'-m,n'-n})^\circ$ for all $\nu$
lexicographically $< \nu_0$ and for all $(\mu,\nu_0)$ with $\mu$
lexicographically $> \mu_0$.
\end{defn}

\begin{rem}
If $f$ is preregular of degree $(\mu_0,0)$ then a Weierstrass change
of variables among the $\xi_i, i=m+1,\dots,m'$ will make $f$ regular in $\xi_{m'}$.
Similarly, if $f$ is preregular of degree $(0, \nu_0)$ then a
Weierstrass change of variables among the $\rho_j, j=n+1,\dots,n'$ will make $f$
regular in $\rho_{n'}$.
\end{rem}

The following lemma is needed in the proof of Theorem \ref{SepSNP}.
The proofs of this lemma and Theorem \ref{SepSNP} are similar to the
proofs of Lemma \ref{lemforSNP} and Theorem \ref{SNP}, with some
additional complications.

\begin{lem}\label{lemSepSNP} Let $\{A_{m,n}\}$ be a separated Weierstrass system and,
using the notation of Definition \ref{defAmn}, let
$$
f = \sum_{\mu,\nu}\overline{f}_{\mu\nu}(\xi,\rho)(\xi'')^\mu
(\rho'')^\nu \in A_{m',n'}
$$
where the $\overline{f}_{\mu\nu}(\xi,\rho) \in A_{m,n}$. There is a
system $\cF$ of rings of $\cA$-fractions, and for each $A' \in \cF$
there is a finite, disjointly covering family of Laurent rings $C$
over $A'_{m,n}$, such that for each $C$ there is a finite set $J_C$,
and $C$-units $u_{C\mu\nu} $  and functions $g_{C\mu\nu} \in
C_{m'-m,n'-n}$ for $(\mu,\nu) \in J_C$ such that
$$
f = \sum_{(\mu,\nu) \in J_C} \overline f_{\mu\nu}
u_{C\mu\nu}g_{C\mu\nu}
$$
as an element of $C_{m'-m,n'-n}$.

In the case that $A$ is a valuation ring with maximal ideal $A^\circ$ we can take $\cF = \{A\}$.
In the case that $m=0$ or $n=0$ we can take the family of Laurent rings corresponding
to $A' \in \cF$ to be just $\{A'_{m,n}\}$.
\end{lem}

\begin{proof}  Let $f$ be as above, say  (using the notation of Definition \ref{defAmn})
$$
f=
\sum_{\mu'\nu'}{\overline{a}_{\mu'\nu'}(\xi')^{\mu'}(\rho')^{\nu'}}
$$
 with the $\overline{a}_{\mu'\nu'}\in A$.  Then by condition (c) (Definition \ref{goodgood})  we have that
 $$
 f = \sum_{(\mu',\nu') \in J }{\overline{a}_{\mu'\nu'}(\xi')^{\mu'}(\rho')^{\nu'}(1 + g_{\mu',\nu'})}
 $$
 with the $ g_{\mu',\nu'} \in A_{m',n'}^\circ$ for some finite $J \subset \bN^{m'+n'}$.
 Hence, splitting into finitely many cases corresponding to a system $\cF$ of rings of $\cA$-fractions, and considering each $A' \in \cF$ separately, we may assume that $\overline{a}_{\mu_0'\nu_0'} = 1$ and that $f$ is preregular in $(\xi',\rho')$ of degree $(\mu_0',\nu_0')$.
( In the case that $A$ is a valuation ring, there is no need to split up into cases to find, and ``factor out", the ``dominant'' coefficient $\overline{a}_{\mu_0'\nu_0'}$.)
 Let $\mu_0' = (\mu_0'',\mu_0)$ and
 $\nu_0' = (\nu_0'',\nu_0)$ and write
 $$
 f= \sum_{(\mu,\nu)}{\overline{f}_{\mu\nu}(\xi,\rho)(\xi'')^{\mu}(\rho'')^{\nu}}
 $$
 with the $\overline{f}_{\mu,\nu} \in A_{m,n}$.  Then $\overline{f}_{\mu_0\nu_0}$ is preregular in $(\xi,\rho)$ of degree $(\mu_0'',\nu_0'')$ and, writing
 $$
\overline{f}_{\mu_0\nu_0}=
\sum_{\nu''}{\overline{f}_{\mu_0\nu_0\nu''}(\xi)(\rho)^{\nu''}},
$$
we see that $\overline{f}_{\mu_0\nu_0\nu_0''}(\xi)$ is preregular in
$\xi$ of degree $\mu_0''$.

If $m=0$ or $n=0$, a Weierstrass change of variables among the $\rho_j$
(respectively the $\xi_i$) will make $\overline{f}_{\mu_0\nu_0}$ regular in $\rho_n$
(respectively $\xi_m$) and no ``splitting up'' into Laurent rings is needed before doing Weierstrass division by $\overline{f}_{\mu_0\nu_0}$ (below).
In the general case, consider the two Laurent rings $C_1$ and $C_2$ defined by the
conditions $|\overline{f}_{\mu_0\nu_0\nu''}| < 1$ and
$|\overline{f}_{\mu_0\nu_0\nu''}| \geq 1$, respectively.

On $C_1$, using $\lambda$ to denote the new variable (of the second kind, i.e. a ``$\rho$" variable)
$\overline{f}_{\mu_0\nu_0\nu''} - \lambda = 0$.  After a Weierstrass
change of variables among $\xi_1,\cdots,\xi_m$, we may assume that
$\overline{f}_{\mu_0\nu_0\nu''} - \lambda$ is regular in $\xi_m$
of degree $s$, say.  Then, in $(C_1)_{m'-m,n'-n}$ we have that
$$
f = R_0(\hat{\xi},\xi'',\rho,\rho'',\lambda) +
\xi_mR_1(\hat{\xi},\xi'',\rho,\rho'',\lambda) + \cdots +
\xi_m^{s-1}R_{s-1}(\hat{\xi},\xi'',\rho,\rho'',\lambda),
$$
where $\hat{\xi} := (\xi_1,\cdots, \xi_{m-1})$.  We now complete the
proof in this case by induction on $s$ and induction on $(m,n)$,
ordered lexicographically, as in the proof of Lemma \ref{lemforSNP}.
($(m,n)$ has been reduced to $(m-1,n+1)$.)

On $C_2$, using $\eta$ for the new variable,
$\eta\overline{f}_{\mu_0\nu_0\nu''} - 1 = 0$. After a Weierstrass
change of variables among $\xi_1,\cdots,\xi_m,\eta$, we may assume
that $\eta\overline{f}_{\mu_0\nu_0\nu''} - 1$ is regular in $\eta$
of degree $s_1$, say.   Considering $\eta f$ and $\eta f_{\mu_0
\nu_0}$, after replacing the coefficient $\eta \overline{f}_{\mu_0
\nu_0 \nu_0''}$ of $\rho^{\nu_0''}$ by $1$, $\eta
\overline{f}_{\mu_0\nu_0}$ is preregular in $\rho$ of degree
$\nu_0''$. After a Weierstrass change of variables among $\rho_1,
\cdots, \rho_m$, we may assume that $\eta \overline{f}_{\mu_0
\nu_0}$ is regular in $\rho_n$, of degree $s_2$, say. Doing
Weierstrass division twice, once by
$\eta\overline{f}_{\mu_0\nu_0\nu''} - 1$ and once by $\eta
\overline{f}_{\mu_0\nu_0}$, we have in $(C_2)_{m'-m,n'-n}$ that
$$
\eta f = Q\cdot \eta \overline{f}_{\mu_0\nu_0} +
\sum_{i<s_1,j<s_2}{R_{ij}(\xi,\xi'',\hat\rho,\rho'')\eta^i\rho_n^j},
$$
where $\hat\rho = (\rho_1, \cdots, \rho_{n-1})$.  One  again
completes the proof by induction on $s_1s_2$ and on $(m,n)$, ordered
lexicographically, as in the proof of Lemma \ref{lemforSNP}.
($(m,n)$ has been reduced to $(m,n-1)$.)
\end{proof}

We continue to use the notation of Definition \ref{defAmn}.

\begin{thm}[The Strong Noetherian Property for separated Weierstrass systems]\label{SepSNP}
Let $\{A_{m,n}\}$ be a separated Weierstrass system and let
$$
f = \sum_{\mu,\nu}\overline{f}_{\mu\nu}(\xi,\rho)(\xi'')^\mu
(\rho'')^\nu \in A_{m',n'}
$$
where the $\overline{f}_{\mu\nu}(\xi,\rho) \in A_{m,n}$. There is a
finite system $\cF$ of rings of $A$-fractions, and for each $A' \in
\cF$ there is a finite, disjointly covering family of Laurent rings
$C$ over $A'_{m,n}$, for each $C$ there is a finite set $J_C$ and
$C$-units $u_{C\mu\nu} \in C$  and functions $h_{C\mu\nu} \in
C_{m'-m,n'-n}^\circ$ for $(\mu,\nu) \in J_C$ such that
$$
f = \sum_{(\mu,\nu) \in J_C} \overline f_{\mu\nu}(\xi,\rho)
(\xi'')^\mu(\rho'')^\nu u_{C\mu\nu}(1+h_{C\mu\nu})
$$
as an element of $C_{m'-m,n'-n}$.

When $A$ is a valuation ring we can take the system $\cF = \{A\}$.
If  $m=0$ or $n=0$ we can take the family of Laurent rings
corresponding to $A' \in \cF$ to be just $\{A'_{m,n}\}$. Hence, in
the case that $A$ is a valuation ring and $mn=0$ no (nontrivial)
rings of $A$-fractions or Laurent rings are needed.
\end{thm}
\begin{proof}  By Lemma \ref{lemSepSNP} we may assume we have written
$$
f = \sum_{(\mu,\nu) \in J_C} \overline f_{\mu\nu}
u_{C\mu\nu}h_{C\mu\nu}
$$
with the $h_{C\mu\nu} \in C_{m',n'}$. As in the proof of Theorem
\ref{SNP}, choosing $d$ large enough, writing $k$ for $m'-m+n'-n$, and taking
$$
I_1:=\{1,\ldots,d\}^k\cup\{(\mu,\nu)\mid \mu_i\geq d, \nu_j\geq d \mbox{ for all $i,j$}\},
$$
we have that
\begin{eqnarray*}
f_{I_1} &:=& \sum_{(\mu,\nu)\in I_1}{\overline f_{\mu\nu}(\xi,\rho)(\xi'')^\mu(\rho'')^\nu}\\
& =& \sum_{(\mu,\nu) \in \{1,\ldots,d\}^k} \overline f_{\mu\nu}
(\xi'')^\mu(\rho'')^\nu u_{C\mu\nu}(1+g_{C\mu\nu}),
\end{eqnarray*}
where the $g_{C\mu\nu} \in (\xi'',C_{m',n'}^\circ)C_{m',n'}$.  Since
each $g_{C\mu\nu}  \text{ mod } (C_{m',n'})^\circ$ is a polynomial
in $\xi''$ (cf. Definition \ref{defAmn}), further increasing $d$, we
may assume that each $g_{C\mu\nu} \in (C_{m',n'})^\circ$.  The proof
is now completed exactly as in the proof of Theorem \ref{SNP}, by induction
on $m'-m + n'-n$.
\end{proof}

\begin{rem}\label{remSepSNP}  In many examples, for example when $A$ is Noetherian and complete in its $I$-adic topology and $A_{m,n} := A\langle\xi\rangle[[\rho]]$ (cf. \cite{CLR1}, section 2), or $A_{m,n} = S_{m,n}^\circ(E,K)$ (cf. \cite{LR1}) we do not have to break up into pieces using rings of $\cA$-fractions and Laurent rings, and the following stronger statement is true (using the notation of Definition \ref{defAmn}):
\begin{itemize}
\item[] Let $f\in A_{m',n'}$ and  write $ f=\sum_{\mu,\nu}
\overline{f}_{\mu\nu} (\xi,\rho)(\xi'')^\mu(\rho'')^\nu$ , where the
$\overline{f}_{\mu\nu}\in A_{m,n}$.  There is a finite set
$J\subset\bN^{m'-m+n'-n}$ and units of the form $1+g_{\mu\nu}$ with
$g_{\mu\nu} \in A_{m',n'}^\circ$, such that
 \begin{eqnarray}
 f=\sum_{(\mu,\nu)\in J} \overline{f}_{\mu\nu}(\xi,\rho)(\xi'')^\mu(\rho'')^\nu(1+g_{\mu\nu}). \label{SSNP}
 \end{eqnarray}
  \end{itemize}
 The Strong Noetherian Properties of Definition \ref{goodgood} and Theorem \ref{SepSNP} follow immediately from this property.
Our treatment
would be a less general but also simpler were we to take this
condition as an axiom replacing the weaker
axiom (c) of Definition \ref{goodgood}. We would then not have to prove Theorem \ref{SepSNP}.
 \end{rem}

\subsection{Strictly convergent analytic structures}\label{strictlyconv}

We consider polynomial rings and power series rings in one kind of
variables, written $\xi_i$, usually variants of Tate rings, hence
the terminology ``strictly convergent''.  In the separated case of the previous
 two subsections the $\rho$ variables (varying over $\Koo$) were used to witness strict inequalities.  Furthermore,
the second  Weierstrass division axiom (Definition \ref{good}(b))
enforced some additional completeness on the ring $A$, on rings of
$\cA$-fractions, and on fields $K$ with analytic $\cA$-structure. In
the strictly convergent case we will allow two possibilities which
we distinguish by use of a designated element $\pi$ of $A$:

(i) $\pi \not = 1$ and in the interpretations given by fields $K$
with analytic $\cA$-structure, $\pi$ is interpreted as a prime
element of $\Ko$, i.e. $\sigma(\pi)$ is an element of smallest
positive order (see Definition \ref{seps}(i)), and (ii) that the
strictly convergent analytic structure is the strictly convergent
part of a separated analytic structure and $\pi = 1$. Except for
this complication we follow the development of the previous two
sections fairly closely.

Case (ii) is treated in Definition \ref{seps}(ii).  We first focus
on case (i). Let $A$ and $I$ be as in the beginning of section
\ref{sepanstruct}. Let $\pi$ be a fixed element of $A$.

 \begin{defn}[System]\label{defAmns}
 A system $\cA=\{A_{m}\}_{m\in \bN}$ of $A$-algebras $A_{m}$,
satisfying, for all $m \leq m'$:
\begin{itemize}
\item[(i)]
$ A_{0} = A,$
\item[(ii)]
$A_{m}\subset A[[\xi_1,\ldots,\xi_m]],$
\item[(iii)]
$A_{m} [\xi_{m+1},\ldots,\xi_{m'}] \subset A_{m'}$,
\item[(iv)]
the image $(A_{m})\;\widetilde{}$ of $A_{m}$ under the residue map
$\; \widetilde{}:A[[\xi_1,\ldots,\xi_m]] \to \widetilde
{A}[[\xi_1,\ldots,\xi_m]]$ is $
 \widetilde{A}[\xi]$, and
 \item[(v)]If  $f \in A_{m'}$, say $f =
\sum_{\mu}\overline{f}_{\mu}(\xi)(\xi'')^\mu$, then the
$\overline{f}_{\mu}$  are in $A_{m}$,
\end{itemize}
is called a \emph{strictly convergent
 $(A,I)$-system.}
\end{defn}

 \begin{defn}[Regular]
\label{regulars} Let $\cA=\{A_{m}\}_{m \in \bN}$  be a strictly
convergent $(A,I)$ system. A power series $f$ in $A_{m}$ is called
{\em regular in $\xi_m$ of degree $d$} when $f$ is congruent in
$A[[\xi]]$, modulo the ideal $\{\sum_{\mu} a_{\mu} \xi^\mu : a_{\mu}
\in I\}$, to a monic polynomial in $\xi_m$ of degree $d$.
\end{defn}

\begin{defn}[Pre-Weierstrass system]\label{goods}
Let $\cA=\{A_{m}\}_{m \in \bN}$ be a strictly convergent
$(A,I)$-system. Then $\cA$ is called a \emph{strictly convergent
pre-Weierstrass system} when the usual Weierstrass Division Theorem
holds in the $A_{m}$, namely, for $f,g\in A_{m}$:

(a) if $f$ is regular in $\xi_m$ of degree $d$, then there exist
uniquely determined elements $q\in A_{m}$ and $r\in A_{m-1}[\xi_m]$
of degree at most $d-1$ such that $g=qf+r$.
\end{defn}

In fact, as in the separated case, we need to be able to work locally, using rings of
fractions:

\begin{defn}[Rings of fractions]\label{fracsshort}
Let $\cA=\{A_{m}\}_{m \in \bN}$ be a strictly convergent
$(A,I)$-system. Inductively define the concept that \emph{$C$ is a
ring of $\cA$-fractions} with ideal $C^\circ$ and rings of strictly
convergent functions $C_m$, $m\geq 0$ as follows:

\item[(i)] The ring $A$ is
 a ring of $\cA$-fractions, with ideal $I = A^\circ$ and rings of strictly convergent functions the $A_m$.

\item[(ii)]
If $B$ is a  ring of $\cA$-fractions and $d\in B$  satisfies
$C^\circ\not =C$, with
$$C:= B/dB,$$
$$C^\circ := B^\circ \cdot C,$$
then $C$ is a ring of $\cA$-fractions, with ideal $C^\circ$ and
strictly convergent functions $C_m=B_m/dB_m$.

\item[(iii)] If $B$ is a  ring of $\cA$-fractions and $c,d\in B$ satisfy $C^\circ\not =C$, with
$$C:=   B\langle {c \over d} \rangle := B{}_{1}/(d\xi_1 - c),$$
$$C^\circ := B^\circ \cdot C,$$
then $C$  is a ring of $\cA$-fractions with ideal $C^\circ$ and
strictly convergent functions $C_m= B{}_{m+1}/(d\xi_1 - c)$.

\item[(iv)] If $B$ is a  ring of $\cA$-fractions  and $c,d\in B$ satisfy $C^\circ\not =C$, with
$$C:=   B\langle{c \over \pi \cdot d} \rangle := B{}_{1}/(d\cdot \pi \cdot \xi_1 - c),$$
$$C^\circ := B^\circ \cdot C,$$
then $C$ is a ring of $\cA$-fractions with ideal $C^\circ$ and
strictly convergent functions $C_m= B{}_{m+1}/(\pi \cdot d \cdot
\xi_1 - c)$.
 \end{defn}

Note that part (iv) above differs from the definition we made in the
separated case; the notation $B\langle{c \over \pi \cdot d} \rangle$
for the ring is reminiscent of the inequality $c < d$ which in our
case (namely $\pi\not=1$) is equivalent to $c \leq \pi \cdot d$.

\begin{defn}[Weierstrass system]
\label{goodgoods} Let $\cA$ be a strictly convergent pre-Weierstrass
system.
 Call $\cA$ a \emph{strictly convergent Weierstrass system} if it satisfies  (c)
below for every ring $C$ of $\cA$-fractions.
\item[(c)] If $f = \sum_{\mu}\overline{c}_{\mu}\xi^\mu$
is in $C_m$
 with the $\overline{c}_{\mu} \in C$, then there is a finite set
 $J \subset \bN^{m}$ and for $\mu \in J$ there is
 $g_{\mu} \in C^\circ_{m}$ such that
 $$
 f = \sum_{\mu \in J} \overline{c}_{\mu}\xi^\mu(1+g_{\mu}).
 $$
\end{defn}

\begin{defn}[Analytic structure]
\label{seps}

\item[(i)] ($\pi \neq 1$) Let $\cA = \{A_{m}\}$ be a strictly
convergent Weierstrass system, and let $K$ be a valued field. A
\emph{strictly convergent analytic $\cA$-structure on $K$} is a
collection of homomorphisms $\{\sigma_{m}\}_{m\in\NN}$, such that,
for each $m\geq 0$, $\sigma_{m}$ is a homomorphism from $A_{m}$ to
the ring of $K^\circ$-valued functions on $(K^\circ)^m$ satisfying:
 \begin{itemize}
 \item[(1)] $I\subset\sigma_{0}^{-1} (K^{\circ\circ})$ and $\sigma_0(\pi)$ is a prime element of $\Ko$,
 \item[(2)]$\sigma_{m}(\xi_i)=$ the $i$-th coordinate function
on $(\Ko)^m$, $i=1,\dots,m$,  and
 \item[(3)] $\sigma_{m+1}$ extends $\sigma_{m}$ where we
identify in the obvious way functions on $(\Ko)^m$ with functions on
$(\Ko)^{m+1}$ that do not depend on the last coordinate.

 \end{itemize}

\item[(ii)] ($\pi = 1$) Let $\cA = \{A_{m,n}\}$ be a separated
Weierstrass system and $K$ a field with separated analytic
$\cA$-structure $\sigma = \{\sigma_{m,n}\}$. We will define the
\emph{strictly convergent analytic structure $\bar \sigma$ on $K$
associated to $\sigma$}. Note that $\sigma$ induces morphisms
$\sigma_{C_{m,0}}$ from $C_{m,0}$ to the ring of $K^\circ$ valued
functions on $(K^\circ)^m$, for $C$ any ring of $\cA$-fractions
which is compatible with $\sigma$ (cf.~Definitions \ref{frac},
\ref{lemCV}). Then $\bar \sigma$ is defined as the collection of
homomorphisms $\sigma_{C_{m,0}}$ for all $m\geq 0$ and all rings $C$
of $\cA$-fractions which are compatible with $\sigma$.

  \end{defn}

In (ii) of Definition \ref{seps}, the analytic structure does not
come from a strictly convergent $(A,I)$-system, but from a richer
system of the rings $C_{m,0}$ for many $C$. However, for the
separated analytic structure $\sigma'$ with $\sigma'$ as in
Definition \ref{expand}, $\bar \sigma'$ contains no more data than
the collection of the maps $\sigma'_{m,0}$. With a slight abuse of
terminology we will also denote in this case the collection of the
maps $\sigma'_{m,0}$ by $\bar\sigma'$ and call $\bar\sigma'$
\emph{strictly convergent analytic structure on $K$ associated to
$\sigma'$}.

In this paper we focus on case (i), ($\pi\not =1$), of Definition
\ref{seps}; studying case (ii) is much more subtle and will require
new techniques.

\begin{defn}[Defining formula]\label{defforms}
Let $C$ be a ring of $\cA$-fractions (Definition \ref{fracsshort}). Call an expression
 a \emph{defining formula for $C$} when it can
inductively be obtained by the following steps

\item[(i)] The expression $(1= 1)$
is a defining formula for $A$.

\item[(ii)] In case (ii) of Definition \ref{fracsshort}, if $\varphi_B$ is a defining
formula for $B$, then
$$\phi_{B} \wedge (d=0)$$
is a defining formula for $C$.

\item[(iii)] In case (iii) of Definition \ref{fracsshort}, if
$\varphi_B$ is a defining formula for $B$,
then
$$
\phi_{B}\wedge (|c|\leq|d|) \wedge (d  \neq
 0)$$
is a defining formula for $C$.

\item[(iv)] In case (iv) of Definition \ref{fracsshort},
if $\varphi_B$ is a defining formula for $B$, then
 $$
\phi_B \wedge (|c| \leq |\pi \cdot d|) \wedge (c\neq0)
 $$
 is a defining formula for
$C$.
\end{defn}

 \begin{defn}[System of rings of fractions]\label{sysfracs}
Let $\cA=\{A_{m}\}_{m \in \bN}$ be a strictly convergent
$(A,I)$-system. If $C$ is a ring of $\cA$-fractions and $c, d
\in C$, then let $\cD_{c,d}(C)$ be the set of rings of fractions
among $C/dC, C\langle {c \over d} \rangle, C\langle{d \over \pi \cdot c}\rangle $ if
this set is nonempty and let $\cD_{c,d}(C)$ be $\{C\}$ otherwise.
Define the concept of a \emph{system of rings of
$\cA$-fractions} inductively as follows. $\cF = \{A\}$ is a system
of rings of $\cA$-fractions.  If $\cF_0$ is a system of rings of
$\cA$-fractions, and $C \in \cF_0$, $0 \neq c, 0 \neq d \in C$, then
$$
\cF:=(\cF_0 \setminus \{C\}) \cup  \cD_{c,d}(C)
$$
is a system of rings of $\cA$-fractions.
 \end{defn}

 \begin{rem} Let $\cV'$ be the theory of the valuation rings of valued fields with prime element $\pi$, and let $\cA$ be a
 strictly convergent $(A,I)$-system and $\cF$ a system of rings of $\cA$-fractions. Let
 $$
 \cV'(A) := \cV' \cup \{|a|\leq 1: a\in A\} \cup \{|b| \leq |\pi|: b \in A^\circ\}.
 $$
 Then
 $$
 \cV'(A) \vdash \bigvee_{C \in \cF} \phi_C
 $$
 and
 $$
  \cV'(A) \vdash \bigwedge_{C \neq C' \in \cF} \neg(\phi_C \wedge \phi_{C'}).
  $$
Hence, if $\sigma$ is a strictly convergent analytic $\cA$-structure
on the field $K$, there will be exactly one $C \in \cF$ such that
$\Ko \models \phi_C$ under the interpretation provided by
  $\sigma$. As in the separated case, we call such $C$ \emph{compatible with $\sigma$}.
 \end{rem}

In the strictly convergent case $\pi \neq 1$ we do not need Laurent rings in
the formulation of the Strong Noetherian Property.

\begin{thm}[The Strictly Convergent Strong Noetherian Property, $\pi \neq 1$]\label{SepSNPs}
Let $\{A_{m}\}$ be a strictly convergent Weierstrass system, let $m
\leq m'$, $\xi = (\xi_1, \dots , \xi_m), \; \xi' = (\xi_1, \dots ,
\xi_{m'}), \;  \xi'' = (\xi_{m+1}, \dots , \xi_{m'})$ and
$$
f = \sum_{\mu}\overline{f}_{\mu}(\xi)(\xi'')^\mu  \in A_{m'}
$$
where the $\overline{f}_{\mu}(\xi) \in A_{m}$. There is a (finite)
system $\cF$ of rings of $A$-fractions,  such for each $C \in \cF$
there is a finite set $J_C$   and functions $h_{\mu} \in
C_{m}^\circ$ for $\mu \in J_C$ such that
$$
f = \sum_{\mu \in J_C} \overline f_{\mu}(\xi) (\xi'')^\mu(1+h_{\mu})
$$
as an element of $C_{m'}$.
\end{thm}

\begin{proof}
This is proved similarly to the special case $n = 0$ of Theorem
\ref{SepSNP}.
\end{proof}

\subsection{Examples of analytic structures}

\label{sepex}
 \noindent
(1) \ Definition \ref{sep} generalizes the notions of analytic
structures in \cite{vdD} and \cite{CLR1}. Even more, if $A$ is
noetherian, complete and separated in its $I$-adic topology, and we
take $A_{m,n}:=A\la\xi\ra [[\rho]]$, then the conditions (a)-(c) of
Definition \ref{good} are satisfied; (c) is immediate and (a) and
(b) are shown in \cite{CLR1}. (Only the sub-case $I=0$ is not taken
care of in \cite{CLR1}, but can be treated similarly). (In
\cite{CLR1} only the case $C=A$ is treated, but the general case
also follows easily from the Strong Noetherian Property of
\cite{CLR1}.  Indeed, the stronger property of Remark
\ref{remSepSNP} is satisfied in this example.)

(2) \ Definitions \ref{sep} and \ref{seps} generalize the notion of analytic
structure in \cite{DHM}, since that is a special instance of the
analytic structure of \cite{vdD} and \cite{CLR1}. To recall, take
$F=\bQ_p$, $T_{m}^\circ=\bZ_p\la\xi\ra$ and $\cS=\{T_{m}^\circ\}$,
or equivalently take $A= \bZ_p, I = p\bZ_p$, $t=p$ and
$\cA=\{A_{m,0}\}$, where $A_{m,0}:= \bZ_p\langle \xi \rangle$. In
\cite{DHM}, $p$-adically closed fields with strictly convergent
$\cS$-structure are studied, in particular, these fields turn out to
be elementarily equivalent to $\QQ_p$ with subanalytic structure.

(3) \ Definitions \ref{sep} and \ref{seps} generalize the notion of
analytic structure (implicit) in \cite{LR1} - let $K$ be a complete,
rank one valued field and take $S^\circ_{m,n}=S^\circ_{m,n}(E,K)$,
$\cS =\{S_{m,n}^\circ \}$  as defined in \cite{LR1}, Definition
2.1.1. If $K'$ is any complete field containing $K$ (or an algebraic
extension of such a field) then $K'$ has separated analytic
$\cS$-structure, with $\sigma$ defined naturally by the inclusion $K
\subset K'$. If $K^*$ is a (non-standard) model of the theory of
$K'$ in the valued field language with function symbols for the
elements of $\cS$ then $K^*$ has separated analytic $\cS$-structure,
as shown in~\cite{LR3}. We also showed in \cite{LR3} that if $K'$ is
a maximally complete field extending $K$ then $K'$ has a separated
analytic $\cS$-structure. (Similar statements hold for strictly
convergent analytic $\cS$-structure).

In \cite{LR2} we established that certain subrings $\cE_{m,n}^\circ$ of
the $S_{m,n}^\circ$, namely the elements $f$ of $S_{m,n}^\circ(E,K)$ such that
$f$ and all its (Hasse) derivatives are existentially definable over
$T_{m+n}(K)$, form a separated Weierstrass system.  This was used in
that paper to prove a quantifier simplification theorem (\cite{LR2} Corollary 4.5) for
$K'_{alg}$ in the language with function symbols for the elements of
\;$\bigcup_nT_n(K)$.  That quantifier simplification theorem (and extensions)
thus follow from Theorem \ref{QES} below.

(4) \ We give another example which does not fall under the scope of
previous papers. Let $F$ be a fixed maximally complete field  with
value group $\Gamma$ (see Section \ref{maxcom} above.)  In the
equicharacteristic case let $E$ be a copy of the residue field in
$F$ and in the mixed characteristic case let $E$ be a set of
multiplicative representatives of $\widetilde F$ as in Section
\ref{maxcom}. For a ring $B\subset F^\circ$ let $B_0=B\cap E$, and define
 $$
 B\la\xi\ra :=\big\{\sum_{g\in I}{ t^g p_g(\xi)}\colon t^g\in B, p_g(\xi)\in B_0[\xi] \text{ and } I\subset\Gamma \text{ is well ordered }\big\}.
 $$
Let $\cB$ be the family of subrings $B$ of $F^\circ$ satisfying
$E\subset B$ and $\supp(B)$ well ordered. Take
$S^\circ_{m,n}(\cB)=\bigcup_{B\in \cB}B\la\xi\ra[[\rho]]$ and
$S_{m,n}(\cB):= F\otimes_{F^\circ}S^\circ_{m,n}(\cB).$ Then we call
$S_{m,n}(\cB)$ the \emph{full ring of separated power series over
$F$} and $T_m({\cB}) := S_{m,0}(\cB)$ the \emph{full ring of
strictly convergent power series over $F$}.  The elements of
$S_{m,n}(\cB)$ naturally define functions $(F_{alg}^\circ)^m \times
(F_{alg}^{\circ\circ})^n \to F_{alg}$, where $F_{alg}$ is the
algebraic closure of $F$, or its maximal completion, cf. \cite{LR3},
section 5. The family $\cS:=\{S^\circ_{m,n}(\cB)\}$ is a strong
Weierstrass system, yielding a separated analytic $\cS$-structure on
$F$, or on $F_{alg}$. Indeed, most of the structure theorems that
hold for the Tate rings in the classical affinoid case (\cite{BGR})
or the rings of separated power series in the quasi-affinoid case
(\cite{LR1}) will also hold for these rings, providing a basis for
affinoid or quasi-affinoid algebra and geometry over maximally
complete fields.

(5) \ See Theorem \ref{constants} and Theorem \ref{algextn}
below for further natural examples obtained by extending analytic
structures by putting in constants from a model, resp.~by going to
algebraic extensions of a model. (``Model'' meaning here field with
analytic structure.)

(6) \ The case of trivially valued field $F$ does not fall under the
scope of the previously mentioned papers (as far as we know), but
might be interesting for the study of tame analytic integrals
(generalizing \cite{D3}, \cite{Pas1}). When $F$ is a trivially
valued field, the family $\{B_{m,n}\}$ with $B_{m,n}=F[\xi][[\rho]]$
is a strong separated Weierstrass system, since $F[\xi][[\rho]]$
equals $F[[\rho]]\langle\xi\rangle$, the $\rho$-adically strictly
convergent power series in $\xi$ over $F[[\rho]]$, cf.~\cite{LR1} or
\cite{CLR1}.  The conditions of Definition \ref{defAmn} are
immediate, as are the Weierstrass Division Theorems, (a) and (b)
(see~\cite{LR1}.)  $B_{m,n}$ is Noetherian since a power series ring
over a Noetherian ring is Noetherian. Property (c) is satisfied
since $F[\xi][[\rho]]$ is the ``full'' power series ring.  A field
$K$ with analytic $\{B_{m,n}\}$-structure need not be trivially
valued.

(7) \ (An example where $A$ is neither Noetherian, nor a valuation
ring.) Let $L$ be a field and $A_i:=L[x_1, \cdots , x_i]$, $I_i :=
(x_1, \cdots , x_i)A_i$ and $A:=\bigcup_i A_i$, $I:= \bigcup_i I_i$,
Then $A_{m,n} := \bigcup_i A_i\langle\xi\rangle[[\rho]]$ is a (strong)
separated Weierstrass system, though $A$ is neither noetherian, nor
a valuation ring.

(8) \ A simple compactness type argument shows that there is a strong
Weierstrass system $\cA' = \{A'_{m,n}\}$ with all the $A'_{m,n}
\subset \bZ[[t]]\langle\xi\rangle[[\rho]]$ \emph{countable}.  This
example is not covered by the treatments of the previous papers.

(9) \ We construct a separated Weierstrass system which is not
strong (i.e.~not satisfying Definition \ref{stronggood}) (Cf.
Example \ref{realex}(7)). Let $B_{m,n}'$ be the algebraic closure of
$\bQ[\xi,\rho]$ in $\bQ[\xi][[\rho]]$  (cf. Example (6).) Then the
family $\{B_{m,n}'\}$ is a separated Weierstrass system. Namely, the
conditions of Definition \ref{defAmn} are easy, Weierstrass Division
(axioms (a) and (b)) can be done with algebraic data, and condition
(c) follows from axioms (a) and (b) in this case with $F=\bQ$ by a
proof similar to the proof of Theorem \ref{SepSNP}. However,
the family $\{B_{m,n}'\}$ does not satisfy condition
\ref{stronggood}(v), as can be seen as follows. Let $f(\xi_1,\xi_2)$
be an algebraic power series whose diagonal is not algebraic (see
for example  \ref{realex}(7)). If
$$
f(\rho\xi_1,\rho\xi_2)=f_1(\rho,\xi_1)+ f_2(\rho,\xi_2) \bmod
(\xi_1\xi_2-1),
$$
then
$$f_1(0,\rho)+f_2(0,\rho)=g(\rho^2),
$$
where $g$ is the diagonal of $f$, and where $\rho$ is a single
variable. But $g(\rho^2)$ is not algebraic.

(10) \ Let $K$ be a complete valued field, and let
$K\langle\langle\xi\rangle\rangle$ be the ring of overconvergent
power series over $K$ (i.e. the elements of $K\langle\xi\rangle$
with radius of convergence $>1$.) Let $\cB$ be a family of quasi-noetherian subrings of $\Ko$ as in
Definition 2.1.1 of \cite{LR1}.
For $\gamma > 1$ let
\begin{eqnarray*}
B\langle \xi \rangle [[\rho]]^{(\gamma)} &:=& \Big\{\sum_{\mu\nu}{a_{\mu\nu}(\xi)^\mu(\rho)^\nu} \colon \exists k \in \bN( |\mu| > k \rightarrow  \\
& & \qquad \qquad \qquad \qquad |a_{\mu\nu}| < \gamma ^{-|\mu|} \text{ or }
|\nu| > (\gamma - 1)|\mu|) \Big\}\\
& \subset& B\langle \xi \rangle[[\rho]]
\end{eqnarray*}
be the subring of  \emph{ $\gamma$-overconvergent power series with
coefficients from $B$}. Define
$$
S_{m,n}^{\circ}(E,K)^{over} := \bigcup_
{\substack{ B \in \cB  \\ 1 < \gamma }} {B \langle \xi \rangle [[\rho]]}^{(\gamma)}.
$$
Then  $\cS^{over} := \{S_{m,n}^{ \circ}(E,K)^{over}\}$ is a
separated Weierstrass system, and $K$ has separated analytic
$\cS^{over}$-structure and strictly convergent analytic
$\{K\langle\langle\xi\rangle\rangle\}$-structure (Definition
\ref{seps}(ii)). Hence some of the results of \cite{S}  fit into our
context.  We leave the verification that  $\cS^{over} $ is a
(strong) separated Weierstrass system to the reader. This
construction can be extended in various nonstandard directions, for
example, to maximally complete fields (cf. example (4) above.)

(11) \ If $\cA$ is a (strong) Weierstrass system, and $J \subset A$
is an ideal (with $(I,J) \subset A$ proper), then $\cA/J :=
\{A_{m,n}/JA_{m,n}\}$ is also a (strong) Weierstrass system, with
$(A/J)^\circ = (I,J)/J$.  The conditions in Definition \ref{defAmn}
and Definition \ref{goodgood} are immediate, as are the conditions
of Definition \ref{stronggood} if $\cA$ is strong.  The conditions
of \ref{good} (Weierstrass Division) are not quite immediate as we
may have $f \in A_{m,n}/JA_{m,n}$ regular in $\xi_m$ (respectively
$\rho_n$) of degree $s$, and $f \equiv F$ mod $JA_{m,n}$, $F \in
A_{m,n}$ without $F$ being regular in $\xi_m$ (respectively
$\rho_n$) of degree $s$.  However,  as in the proof of Lemma
\ref{(vii)}, we can modify $F$ to find a $G \in A_{m,n}$ with $f
\equiv G$ and $G$ regular in $\xi_m$ (respectively $\rho_n$) of
degree $s$.

(12) \ Further important examples are provided by Definition
\ref{expand} (see Theorem \ref{constants}) and Theorem
\ref{algextn}.

(13) \ Every Henselian valued field carries an (algebraic) analytic
structure all the functions of which are definable, see section \ref{algebraic}.
Hence the (algebraic) theory of Henselian valued fields is included in our
formalism of (Henselian) fields with analytic structure.\\

In examples (1) and (2) of strictly convergent analytic structures
$(\pi\not=1)$, the strictly convergent analytic structure gives the
same family of analytic functions as a natural separated analytic
structure. Hence, all results on separated analytic structures
apply. We don't know any example where this is not the case. We will
for this reason from section \ref{annuli} on focus on separated
analytic structures.

\subsection{Properties of analytic structures} \label{properties} In this subsection we develop some of the basic properties of separated and strictly convergent analytic structures.

\begin{rem}
If $K$ is trivially valued (i.e. $K^{\circ\circ} = (0)$) then $I
\subset \ker\sigma_{0}$ and the analytic structure collapses to the
usual algebraic structure given by polynomials.  See Example
\ref{sepex}(11). The domain of the $\rho$ variables is $\{0\}$ if
$\Koo = (0)$.
\end{rem}

\begin{rem}\label{remcomp} Considering  $A_{m,n} \subset A[[\xi,\rho]]$,
we do not know a priori that ``composition" or ``substitution'' by
elements of $A_{M,N}$ for $\xi$-variables and elements of
$(A_{M,N})^\circ$ for $\rho$-variables in the rings $A_{m,n}$ makes
sense. However, Properties (a) and (b) (Weierstrass Division) of
Definition \ref{good}  (or Property (a) of Definition \ref{goods})
allow us to \emph{define} composition in these rings. For example,
dividing $f(\xi_1)$ by $\xi_1 - g(\xi_2)$ gives $f(\xi_1) = Q\cdot
(\xi_1 - g(\xi_2)) + h(\xi_2)$ and we can define $f(g(\xi_2))$ to be
$h(\xi_2)$. Definition \ref{good} (or Definition \ref{goods})
guarantees that this actually is composition on the ``top slice". In
all the standard examples (Examples \ref{sepex}) composition defined
in this way actually is power series composition.  In a field with
analytic structure (i.e. after applying $\sigma$) this ``defined"
composition becomes actual composition.  In the above example,
$\sigma(f(\xi_1)) = \sigma(Q)(\xi_1 - \sigma(g(\xi_2))) +
\sigma(h(\xi_2))$, so $\sigma(h) = \sigma(f)\circ \sigma(g)$.
\end{rem}

\begin{prop}
Analytic $\cA$-structures preserve composition.  More precisely, if
$\cA = \{A_{m,n}\}$ and $f\in A_{m,n}$, $\alpha_1,\dots,\alpha_m\in
A_{M,N}$, $\beta_1,\dots,\beta_n\in (A_{M,N})^{\circ}$,  then
$g:=f(\alpha,\beta)$ is in $A_{M,N}$ and
$\sigma(g)=(\sigma(f))(\sigma(\alpha),\sigma(\beta))$. \
\end{prop}
\begin{proof} As in \cite{CLR1}, Proposition 2.8, and the above
remark, this follows from Weierstrass Division.
\end{proof}

\begin{prop}\label{zero}
In a nontrivially valued field with analytic $\cA$-structure, the
image of a power series is the zero function if and only if the
image of each of its coefficients is zero.  More precisely, for $K$
a nontrivially valued field,

{\em (i)} Let $\sigma$ be a separated analytic $\cA$-structure on
$K$. Then
$$
\ker\sigma_{m,n}= \{\sum_{\mu,\nu} a_{\mu,\nu} \xi^\mu\rho^\nu \in
A_{m,n} : a_{\mu,\nu} \in \ker\sigma_{0,0} \}.
$$ Furthermore, with the notation of Definition
\ref{defAmn}(v), if
 $f(\xi',\rho')=\sum_{\mu,\nu}\overline{f}_{\mu\nu}(\xi,\rho)(\xi'')^\mu(\rho'')^\nu$
 and $a \in (\Ko)^m$, $b \in (\Koo)^n$, then we have that the function
 $$
(\Ko)^{m'-m}\times (\Koo)^{n'-n}\to \Ko:(c,d)\mapsto
\sigma_{m',n'}(f)(a,c,b,d)
 $$
is the zero function exactly when
$\sigma_{m,n}(\overline{f}_{\mu\nu})(a,b) = 0$ for all $\mu,\nu$.

{\em (ii) ($\pi\not=1$)} Let $\sigma$ be a strictly convergent
analytic $\cA$-structure on $K$. Then
$$
\ker\sigma_m= \{\sum_{\mu} a_{\mu} \xi^\mu \in A_m : a_{\mu} \in
\ker\sigma_{0}\}.
$$
Furthermore, with the notation of Definition \ref{defAmns}(v), if
 $f(\xi')=\sum_{\mu}\overline{f}_{\mu}(\xi)(\xi'')^\mu$
 and $a \in (\Ko)^m$, then the function
 $$(\Ko)^{m'-m}\to\Ko: c\mapsto \sigma_{m'}(f)(a,c)$$
is the zero function exactly when $\sigma_{m}(\overline{f}_{\mu})(a)
= 0$ for all $\mu$.
\end{prop}

\begin{proof}
This follows easily from the Strong Noetherian Property and
Weierstrass Preparation. Case (i) in the case that $K$ is
nontrivially valued is given in detail in \cite{CLR1} Proposition
2.10.
\end{proof}

\begin{rem}
Proposition \ref{zero} does not always hold when $K$ is trivially
valued. Namely, in case (i), the function $\sigma(\rho_1)$ is the
zero function. In case (ii), if $K$ is a finite field there is a
nonzero monic polynomial $p(\xi_1) \in A_{1}$ such that $\sigma(p)$
is the zero function. There is no problem in the construction of
Remark \ref{constantsbis} however.
\end{rem}

When considering a particular field $K$ with analytic
$\cA$-structure, it is no loss of generality to assume that
$\ker\sigma_{0,0} = (0)$ (See Example \ref{sepex}(11)).  Indeed we
can replace $A$ by $A/\ker\sigma_{0,0}$  to get an equivalent
analytic structure on $K$ with this property.  In the case that
$\ker\sigma_{0,0} = (0)$ we may consider $A_{0,0}= A$ to be a
subring of $K^\circ$. It is convenient to extend the Weierstrass
system by suitably adjoining the elements of $K$ to $A$:

\begin{defn}[Extension of parameters]\label{expand}
(i) Let $\cA=\{A_{m,n}\}$ be a separated Weierstrass system and let
$K$ be a nontrivially valued valued field with analytic
$\cA$-structure $\{\sigma_{m,n}\}$. Assume that $\ker\sigma_{0,0} =
(0)$, so we may consider $A$ as a subring of $K^\circ$ and $A_{m,n}$
as a subring of $\Ko[[\xi,\rho]]$. With the notation of Definition
\ref{defAmn}(v) and with $M=m'-m$, $N=n'-n$, if $f =
\sum_{\mu,\nu}{\overline{f}_{\mu\nu}(\xi,\rho)(\xi'')^\mu(\rho'')^\nu}$,
$a \in (K^\circ)^m$ and $b \in (K^{\circ\circ})^n$,  we write
$f(a,\xi'',b,\rho'')$ for the power series
$$\sum_{\mu,\nu}{\sigma_{m,n}(\overline{f}_{\mu\nu})(a,b)(\xi'')^\mu(\rho'')^\nu}\
\mbox{ in } K^\circ[[\xi'',\rho'']].$$  Then by Proposition
\ref{zero}
the function
$$(\Ko)^{M}\times (\Koo)^{N}\to \Ko:(c,d)\mapsto (\sigma f)(a,c,b,d)
$$
only depends on the power series $f(a,\xi'',b,\rho'')$ and we denote
this function by
$$
\sigma'_{M,N}(f(a,\xi'',b,\rho''))
$$
Define the subring $A_{M,N}(K)$ of $K^\circ[[\xi'',\rho'']]$ by
\begin{eqnarray*}
A_{M,N}(K)&:=& \bigcup_{m,n \in \bN}\{f(a,\xi'',b,\rho'')\colon f
\in A_{m',n'},\ a\in (K^\circ)^m,\ b\in (K^{\circ\circ})^n\},
\end{eqnarray*}
and define
$$\cA(K):=\{A_{M,N}(K)\}_{M,N}.$$ Then,
$\sigma'_{M,N}$ is a homomorphism from $A_{M,N}(K)$ to the ring of
functions $(K^\circ)^M \times (K^{\circ\circ})^N \to K^\circ$, for
each $M,N$.

(ii)($\pi\not=1$) Let $\cA=\{A_{m}\}$ be a strictly convergent
Weierstrass system and  assume that $K$ has analytic $\cA$-structure
$\{\sigma_{m}\}$ and that $\ker\sigma_{0} = (0)$ so $A = A_{0}
\subset K^\circ$. With the notation of Definition \ref{defAmns} (v) and
with $M=m'-m$, if $f =
\sum_{\mu}{\overline{f}_{\mu}(\xi)(\xi'')^\mu}$ and $a \in
(K^\circ)^m$, define the power series $f(a,\xi'')$ as
$\sum_{\mu}{\sigma_{m}(\overline{f}_{\mu})(a)(\xi'')^\mu} \in
K^\circ[[\xi'']]$.  Then by Proposition \ref{zero} the function $ \sigma'_{M}(f(a,\xi'')) $
defined by
$$
(\Ko)^{M}\to\Ko: c\mapsto \sigma_{m'}(f)(a,c)
$$
only depends on the power series $f(a,\xi'')$. Define
\begin{eqnarray*}
A_{M}(K)&:=& \bigcup_{m \in \bN}\{f(a,\xi'')\colon f \in A_{m+M},\ a\in (K^\circ)^m\}\\
&\subset& K^\circ[[\xi'']]\\
\cA(K)&:=&\{A_{M}(K)\}.
\end{eqnarray*}
Then $\sigma'_{M}$ is a homomorphism from $A_{M}(K)$ to the ring of
functions $(K^\circ)^M  \to K^\circ$.
\end{defn}

We have

\begin{thm} \label{constants}(i) Let $K$ and $\cA$ be as in Definition \ref{expand}(i).
Then $\cA(K)$ is a separated Weierstrass system over $(\Ko,\Koo)$
and $K$ has separated analytic $\cA(K)$-structure via the
homomorphisms $\{\sigma'_{M,N}\}$.

 The family $\{A_{M,0}(K)\}$ is a strictly convergent Weierstrass system and $\{\sigma'_{M,0}\}$ provides K with
a strictly convergent analytic structure as described in Definition
\ref{seps}(ii). (No rings of $\Ko$-fractions are needed as $\Ko$ is
a valuation ring.)
 If $\cA$ is a strong  Weierstrass system, so is $\cA(K)$.

 (ii) Let $K$ and $\cA$ be as in Definition \ref{expand}(ii).
Then $\cA(K)$ is a strictly convergent Weierstrass system and $K$
has strictly convergent analytic $\cA(K)$-structure via the
homomorphisms $\{\sigma'_{M}\}$. If $\cA$ is a strong  Weierstrass
system, so is $\cA(K)$.
\end{thm}

\begin{proof}
The various properties for $\cA(K)$ follow using the Weierstrass
Division Theorem and the Strong Noetherian Property and the
corresponding properties for $\cA$.  (The case when $A_{m,n} =
A\langle\xi\rangle[[\rho]]$, with $A$ noetherian and complete in its
$I$-adic topology is given in \cite{CLR1}, Lemma-Definition 2.12.)
\end{proof}

\begin{rem}[More general constants]\label{constantsbis}
Let $K$ and $\cA$ be as in Definition \ref{expand}(i), and use its
notation. For $K'$ any subfield of $K$, one can define
$$
A_{M,N}(K') := \bigcup_{m,n \in \bN}\{f(a,\xi'',b,\rho'')\colon f
\in A_{m',n'},\ a\in (K'{}^\circ)^m,\ b\in (K'^{\circ\circ})^n\}.
$$
Similarly as in Theorem \ref{constants}, if $K$ has analytic
$\cA$-structure, then it has analytic $\cA(K')$ structure.
\end{rem}

We restate the special case $m=0$ or $n=0$ of the Strong Noetherian Property (Theorem \ref{SepSNP})
for $\cA(K)$, that we will need in Section \ref{annuli}.

\begin{cor}\label{VSSNP} Let $\cA = \{A_{m,n}\}$ be a separated Weierstrass system and let
nontrivially valued field $K$ have separated analytic
$\cA$-structure $\sigma$ with $\ker(\sigma_{0,0}) = (0)$. If $f \in
A_{m+M,n+N}(K)$ and $m=0$ or $n=0$, we can write
$$
 f=\sum_{(\mu,\nu)\in J} \overline{f}_{\mu\nu}(\xi,\rho)(\xi'')^\mu(\rho'')^\nu(1+g_{\mu\nu})
$$
where the $g_{\mu\nu} \in A_{m+M,n+N}^\circ$  and $J\subset\bN^{M+N}$ is a finite set.
\end{cor}

We restate some results from \cite{CLR1} whose proofs extend without difficulty to our current more general setting.

\begin{prop} (Proposition 2.17 of \cite{CLR1})
\label{Henselian} \item[(i)]  Let $\cA$ be a separated Weierstrass
system and let $K$ be a valued field with separated analytic
${\cA}$-structure; then $\Ko$ is a Henselian valuation ring.

\item[(ii)($\pi\not=1$)] Let $\cA$ be a strictly convergent Weierstrass
system and let $K$ be a valued field with strictly convergent
analytic ${\cA}$-structure.
Then $\Ko$ is a Henselian valuation ring.

\end{prop}

The following theorem permits us to work over any algebraic
extension of the domain of an analytic $\cA$-structure. Its proof
is the same as that of Theorem 2.18 of \cite{CLR1}.

\begin{thm}\label{algextn}
Let $K$ be a valued field and let $K'$ be an algebraic extension of
$K$.\\
\indent (i) Let $\cA$ be a separated Weierstrass system. Suppose that $K$ has a separated analytic ${\cA}$-structure
$\sigma =\{\sigma_{m,n}\}$. Then there is a unique extension of
$\sigma$ to a separated analytic $\cA$-structure $\tau$ on $K'$.

(ii) ($\pi\not=1$)] Let $\cA$ be a strictly convergent
Weierstrass system. Suppose that $K$ has a strictly convergent
analytic $\cA$-structure and that $\sigma(\pi)$ remains prime in
$(K')^\circ$. Then there is a unique extension of $\sigma$ to a
strictly convergent analytic $\cA$-structure $\tau$ on $K'$.

\end{thm}

\begin{rem}
(i) Let $C$ be a ring of $\cA$-fractions. Let  $\sigma$ be a separated
analytic $\cA$--structure on $K$ and let $B$ be a Laurent ring over
$C_{m,n}$ with defining formula $\varphi_B$. Then the formula
$\varphi_B$ defines in a natural way a subset $\cU_{\varphi_B}$ of
$(\Kalgo)^{m}\times(\Kalgoo)^{n}$ by~Theorem \ref{algextn}. If the
set $\cU_{\varphi_B}$ is nonempty, then it is called a \emph{Laurent
subdomain} in the terminology of \cite{BGR}, Definition 7.2.3.2, and
it is called a \emph{$\Kalgo$-subdomain} in the terminology of
\cite{LR1}, Definition 5.3.3. Moreover, $B$ defines in a natural way
a ring of (analytic) functions on $U_{\phi_{B}}$, via $\sigma$. We
will not study these rings of functions in detail.

(ii) Let  $\sigma$ be a separated analytic $\cA$--structure on $K$ and $C$
a ring of $\cA$-fractions which is compatible with $\sigma$. Let $B$
be a Laurent ring over $C_{m,n}$ with a defining formula
$\varphi_B$. Then $\cU_{\varphi_B}$ as in (i) is nonempty if and
only if $\varphi_B$ can be obtained inductively as in Definition
\ref{Laurentfracdefform} such that for each step of type (ii) there
exists $x$ in $(\Kalgo)^{m+M}\times(\Kalgoo)^{n+N}$ with
$|\sigma(f)(x)|\geq 1$ and for each step of type (iii) there exists
$x$ in $(\Kalgo)^{m+M}\times(\Kalgoo)^{n+N}$ with $|\sigma(f)(x)|<
1$. We will not use this property.
\end{rem}

\begin{rem} \label{nonHenselian} The conclusion of Proposition \ref{Henselian}(ii), and hence also the conclusions of Theorem \ref{algextn}(ii), can be false without the assumption that $\ord(\Koo)$ has  minimal element $\ord(\sigma_0(b))$.  To see this consider $A_m = T^\circ_m(\bC_p)$, the strictly convergent power series of gauss-norm $\leq 1$ over $\bC_p$, the completion of the algebraic closure of $\bQ_p$.  Let $\bC_p^*$ be a nonprincipal ultrapower of $\bC_p$, and let $b \in \bC_p^*$ satisfy $1-
 {1 \over n} < |b| < 1$ for all $n \in \bN$.  Let $K$ be the smallest substructure of
$\bC_p^*$ containing $b$ and closed under $+,\cdot,(\cdot)^{-1}$ and
all the functions of $\bigcup_m A_m$.  Then, by definition, $K$ has
analytic $\{A_m\}$-structure in the sense that the elements of $A_m$
define functions on $K$ in a natural way. Let $p^{1 \over \infty}
\subsetneq \Koo$ be the ideal in $\Ko$ generated by $\{p^{1 \over
n}: n\in \bN\}$.  One can see by induction on terms that $\Ko/p^{1
\over \infty}\Ko = \{{f(b) \over g(b)} : f,g \in \overline \bF_p[b]
\text{ and } g(0) \neq 0 \}$. Here $\overline \bF_p$ is the
algebraic closure of the $p$-element field. $\Ko$ is not Henselian -
consider $f(x) = 1 - x - bx^2$, which is regular in $x$ of degree
$1$ but has no zero in $\Ko$.  Indeed it has no zero in $\Ko/p^{1
\over \infty}\Ko$.
\end{rem}

\begin{defn}
Let $\cA = \{A_{m,n}\}$ be a separated or strictly convergent
Weierstrass system.  Let $\cL = \langle0,1,+,\cdot,^{-1},
\overline{\cdot}, |\cdot|\rangle$ be the language of valued fields
(the symbol $\overline{\cdot}$ denotes multiplication on the value
group.) Let $\cL_{\cA}$ be $\cL$ augmented with function symbols for
the elements of  $\bigcup_{m,n} A_{m,n}$.  We call $\cL_{\cA}$ the
\emph{language of valued fields with analytic $\cA$-structure}. Let
$K$ be a valued field with analytic $\cA$-structure.  Interpret $f
\in A_{m,n}$ as $\sigma(f) : (\Ko)^m \times(\Koo)^n \mapsto \Ko$,
extended by zero outside its domain, so $K$ becomes an
$\cL_{\cA}$-structure.
\end{defn}

The following theorem extends \cite{LL2} Theorem 3.8.2, \cite{LR2} Theorem 4.2 and \cite{LR3} Theorem 4.2.

 \begin{thm}(Quantifier Elimination.)\label{QES}  Let $K$ be an algebraically closed valued field with
separated analytic $\cA$-structure. Let $\cL_{\cA}$ be as above.
Then $K$ admits elimination of quantifiers in $\cL_{\cA}$.
\end{thm}

\begin{proof}  The proofs of the above cited theorems, which are based on Weierstrass Division and the Strong Noetherian Property, work with very minor modifications.
\end{proof}

Note that Theorem \ref{QES} does not impose any condition on the
characteristic of $K$. (In fact, everything up to the end of section
\ref{annuli} is for all characteristics.) In section \ref{seccell},
Theorem \ref{mt}, we will give a quantifier elimination statement in
the case that $K$ has characteristic zero but is not necessarily
algebraically closed.

\subsection{The algebraic case}\label{algebraic}
In this subsection we show that every Henselian field carries a
\emph{definable} (algebraic) analytic structure, so our theory of
Henselian fields with analytic structure contains the first order
theory of Henselian fields as a special case.  Along the way we
establish some results about algebraic strictly convergent and
separated power series over a discretely valued Henselian field that
may be of independent interest.  The results of this section are not
used in the rest of the paper.

Let $R$ be a field or an excellent Henselian discrete valuation
ring, with prime $\frak p$. Let $\xi = (\xi_1, \dots , \xi_m)$.  Let
$\widehat R$ denote the $\frak p$--adic completion of $R$ and
$\widehat R \langle \xi \rangle$ the ring of strictly convergent
power series in $\xi$ with coefficients from $\widehat R$.  Let
$R\langle \xi \rangle_{alg}$ denote the algebraic closure of
$R[\xi]$ in $\widehat R \langle \xi \rangle$, i.e. all the elements
of  $\widehat R \langle \xi \rangle$ that are algebraic over
$R[\xi]$.  (The condition of excellence for $R$ is the same as
assuming that $\widehat R$ is separable over $R$, i.e. if $ y_1,
\dots , y_n \in \widehat R$ then the field extension $Q(R[y_1, \dots
, y_n])$ over $Q(R)$ has a  separating transcendence base.  Here
$Q(A)$ denotes the quotient field of $A$.)

\begin{thm}(Artin Approximation for $R \langle \xi \rangle_{alg}$)\label{affapprox}
Let $\eta = (\eta_1, \dots , \eta_N)$, and let $F_i(\xi,\eta) \in R[\xi,\eta]$ for $i=1, \dots, d$
be a finite set of polynomials, and let $c \in \bN$ and
$\overline \eta(\xi) \in (\widehat R\langle\xi\rangle)^M$ satisfy $F_i(\xi, \overline \eta(\xi)) = 0$
for $i=1,\dots,d$.  Then there exist $\widetilde \eta(\xi) \in (R\langle\xi\rangle_{alg})^M$ satisfying
 $F_i(\xi, \widetilde \eta(\xi)) = 0$ for $i=1,\dots,d$ and
 $\widetilde \eta(\xi) \equiv \overline \eta (\xi) \mod \frak p^c$.
 Moreover, there are $\widetilde \zeta_1, \dots, \widetilde \zeta_n \in R\langle\xi\rangle_{alg}$ such that
 $R[\xi, \widetilde \eta, \widetilde \zeta_1, \dots, \widetilde \zeta_n]$ is a Henselian extension of $R[\xi]$, i.e. the Jacobian of this extension is a unit -- there are $G_j \in R[\xi,\eta,\zeta]$, $j=1,\dots, N+n$ such that
 $\big({\partial G \over \partial \eta \partial \zeta}\big)(\widetilde \eta, \widetilde \zeta)$ is a unit.
\end{thm}

\begin{proof}  The proof is the same as that of \cite{Ar}
in the local case, except that one uses the strictly convergent
Weierstrass Preparation and Division Theorems instead of the local
versions.  For an elementary exposition of the N\'eron
desingularization see \cite{DL2} section 3. (Alternatively, one
could follow the slightly different proof given in \cite{Bo}.)
\end{proof}

The quasi-affinoid or separated case is more complicated.  First we introduce some notation.
Let $\xi = (\xi_1,\dots,\xi_m)$, $\eta =(\eta_1,\dots,\eta_M)$, $\rho = (\rho_1,\dots,\rho_n)$ and
$\lambda = (\lambda_1,\dots,\lambda_N)$.
$$
S_{m,n}^\circ(\widehat R) := \widehat R \langle\xi\rangle[[\rho]]
$$
and $S_{m,n}^\circ(R)_{alg}$ is the algebraic closure of $R[\xi,\rho]$ in $S_{m,n}^\circ(\widehat R)$.

\begin{defn}An \emph{algebraic Laurent domain} (in $(F_{alg}^\circ)^m \times  (F_{alg}^{\circ\circ})^n$, where
$F$ is the quotient field of $R$ or $\widehat R$) is a domain defined inductively by a sequence
of inequalities of the form $|p_i(\xi)| \Box_i 1$ where each $p_i$ is a polynomial
with $\|p_i\| = 1$ and each $\Box_i \in \{ \geq, <\}$.  For algebraic Laurent domains $\cU$ we define the ring of $\widehat R$--analytic functions $\cO_{\widehat R}(\cU)$ on $\cU$ inductively as follows.  If
$\cO_{\widehat R}(\cU) = S_{m',n'}^\circ(\widehat R)/I$ and $p(\xi)$ is a polynomial with $\|p\| = 1$ and
\begin{eqnarray*}
\cU'&:=& \cU \cap \{(\xi,\rho): |p(\xi)| \geq 1 \}\\
\cU''&:=& \cU \cap \{(\xi,\rho): |p(\xi)| < 1 \}
\end{eqnarray*}
then
\begin{eqnarray*}
\cO_{\widehat R}(\cU') &:=& S^\circ_{m'+1,n'}(\widehat R)/ (I, p(\xi)\xi_{m'+1} - 1)\\
\cO_{\widehat R}^\circ(\cU') &:=& (\frak p, \rho_1, \dots , \rho_{n'})\cO_{\widehat R}(\cU')\\
\cO_{\widehat R}(\cU'') &:=& S^\circ_{m',n'+1}(\widehat R)/ (I, p(\xi)- \rho_{n'+1})\\
\cO_{\widehat R}^\circ(\cU'') &:=& (\frak p, \rho_1, \dots , \rho_{n'+1})\cO_{\widehat R}(\cU'').\\
\end{eqnarray*}
For $\cU$ an algebraic Laurent domain we define the ring of \emph{algebraic} functions $\cO_R(\cU)_{alg}$ to be the algebraic closure of $R[\xi,\rho]$ in $\cO_{\widehat R}(\cU)$, and
$$
\cO_R^\circ(\cU)_{alg} := \cO_R(\cU)_{alg} \cap \cO_{\widehat R}^\circ(\cU).
$$
A finite family $\{\cU_i\}$ of algebraic Laurent domains is called \emph{covering} if
$\bigcup_i\cU_i = (F_{alg}^\circ)^m \times  (F_{alg}^{\circ\circ})^n$.
\end{defn}

\begin{thm}(Artin approximation for $S_{m,n}^\circ(R)_{alg}$)\label{quasiaffapprox}
Let  $F_i(\xi,\eta,\rho,\lambda) \in R[\xi,\eta,\rho,\lambda]$ for $i = 1, \dots , d$, let $c \in \bN$ and suppose that
$\overline \eta(\xi,\rho) \in (S_{m,n}^\circ(\widehat R) )^M$,
$\overline \lambda(\xi,\rho) \in (\frak p, \rho)(S_{m,n}^\circ(\widehat R) )^N$ satisfy
$$
F_i(\xi,\rho, \overline \eta(\xi,\rho), \overline \lambda(\xi,\rho)) = 0
$$
for $i=1,\dots,d$.  Then there is a finite covering family $\{\cU_i\}$, $i=1,\dots,e$ of algebraic Laurent domains and for each $i$ there are
\begin{eqnarray*}
\widetilde \eta^{(i)}(\xi,\rho) &\in& (\cO_R(\cU_i)_{alg})^M\\
\widetilde \lambda^{(i)}(\xi,\rho) &\in& (\cO_R^\circ(\cU_i)_{alg})^N
\end{eqnarray*}
such that
\begin{eqnarray*}
\widetilde \eta^{(i)}(\xi,\rho) &\equiv& \overline \eta^{(i)}(\xi,\rho) \mod (\frak p, \rho)^c \\
\widetilde \lambda^{(i)}(\xi,\rho) &\equiv& \overline \lambda^{(i)}(\xi,\rho)  \mod (\frak p, \rho)^c \text{ and }\\
F_j(\xi,\rho,\widetilde \eta^{(i)}, \widetilde \lambda^{(i)}) &=& 0 \quad\quad \text{ for }  j=1,\dots,d.
\end{eqnarray*}

Moreover, for each $i$ there are $\widetilde \zeta^{(i)} =
(\widetilde \zeta_1^{(i)},\dots,\widetilde \zeta_L^{(i)}) \in
(\cO_R(\cU_i)_{alg})^L$ and $\widetilde \sigma^{(i)} = (\widetilde
\sigma_1^{(i)},\dots,\widetilde \sigma_{L'}^{(i)}) \in
(\cO^\circ_R(\cU_i)_{alg})^{L'}$ such that
$\cO_R(\cU_i)_{alg}[\widetilde \eta, \widetilde \lambda,\widetilde
\zeta, \widetilde \sigma]$ is a Henselian extension of
$\cO_R(\cU_i)_{alg}$, i.e. the Jacobian of the extension is a unit
--  there are $G_j \in R[\xi,\eta,\lambda,\zeta,\sigma]$,
$j=1,\dots, M+N+L+L'$ such that
 $\Big({\partial G \over \partial \eta \partial \lambda \partial \zeta \partial \sigma}\Big)(\widetilde \eta, \widetilde \zeta, \widetilde \lambda,\widetilde \sigma)$ is a unit.
\end{thm}
\begin{proof}
The proof, including the N\'eron $\frak p$--desingularization, is the same as in the affinoid case, except for the following complication.  As in the affinoid case we reduce to the situation (using the notation of \cite{Ar}) that
$$
F(\xi,\rho,\overline \eta,\overline \lambda) = 0
$$
and
$$
\overline \delta = \delta(\xi,\rho,\overline \eta, \overline \lambda) \not\equiv0 \mod \frak p
$$
where $\delta$ is the determinant of a suitable minor of the Jacobian matrix.
It may not be possible to make $\overline \delta$ regular by permissible changes of variables among the $\xi_i$ and among the $\rho_j$. However, $\overline \delta$ will be \emph{preregular} (see Definition \ref{prereg}) of degree $(\mu_0,\nu_0)$, say.  Writing
$$
\overline \delta(\xi,\rho) = \sum_\nu{\overline \delta _\nu(\xi)\rho^\nu}
$$
we break up into two algebraic Laurent domains as follows:

(i) $\cU_1$ defined by $|\overline \delta_{\nu_0}(\xi)| \geq 1$ i.e. $\overline \delta_{\nu_0} \xi_{m+1} - 1 = 0$, and

(ii) $\cU_2$ defined by $|\overline \delta_{\nu_0}(\xi)| < 1$ i.e. $\overline \delta_{\nu_0}  - \rho_{n+1} = 0$.

Case (i).  After Weierstrass changes of variables among the $\xi_i$ and among the $\rho_j$, we may
assume that $\overline \delta_{\nu_0} \xi_{m+1} - 1$ is regular in $\xi_{m+1}$ of degree $s$ say, and that
$\xi_{m+1}\overline \delta$ is regular in $\rho_n$.  We are thus reduced to seeking solutions in
$\cO_R(\cU_1)_{alg}$ (which is a finite $S_{m,n}^\circ(R)_{alg}$--algebra) with $\overline \delta$ regular in
$\rho_n$ and we proceed as in the affinoid, or local cases by induction on $n$.  The reduction from a finite
$S_{m,n-1}^\circ(R)_{alg}$--algebra to $S_{m,n-1}^\circ(R)_{alg}$ is immediate as in \cite{Ar}.

Case (ii)  After a Weierstrass change of variables among the $\xi_i$, we may assume that
$\overline \delta_{\nu_0}  - \rho_{n+1}$ is regular in $\xi_m$.  Hence  we are reduced to the case of a finite algebra over $S^\circ_{m-1,n+1}(R)_{alg}$, and hence to the case of $S^\circ_{m-1,n+1}(R)_{alg}$, and the proof follows by induction on the pairs $(m,n)$ ordered lexicographically.  The case $m=0$ is like the affinoid case -- $\overline \delta$ can be made regular in $\rho_n$ by a Weierstrass change of variables among the $\rho_j$.
\end{proof}

\begin{rem}\label{apprem} The condition on the Jacobian in Theorem \ref{quasiaffapprox} (and Theorem \ref{affapprox}) shows that all the algebraic power series (i.e. power series in $\cO_R^\circ(\cU)_{alg}$ for $\cU$ a Laurent domain) define functions on any Henselian field containing $R$ that are existentially definable in the language of Henselian fields.
\end{rem}

\begin{thm}\label{AWP}$S_{m,n}^\circ(R)_{alg}$ satisfies the two Weierstrass Preparation and Division Theorems (Definition \ref{good}).  In other words, if $f, g \in S_{m,n}^\circ(R)_{alg}$ and $f$ is regular (in $\xi_m$ or $\rho_n$), then the associated Weierstrass data are algebraic.
\end{thm}
\begin{proof}
First we prove that the data associated with Weierstrass Preparation applied to a regular algebraic power series are algebraic.  Let $f(\xi,\rho) \in S_{m,n}^\circ(R)_{alg}$ be regular in $\xi_m$ of degree $s$ say.  (The case that $f$ is regular in $\rho_n$ is similar.)  Let $\xi' = (\xi_1,\dots,\xi_{m-1})$.  Then there is an
$F(\xi,\rho,Y) \in R[\xi,\rho,Y]\setminus \{0\}$ such that $F(\xi,\rho,f(\xi,\rho))=0$.  We may assume that
$F(\xi,\rho,0) \neq 0$.  Let
$$
f(\xi,\rho) = U(\xi,\rho)[\xi_m^s + A_{s-1}(\xi',\rho)\xi_m^{s-1} + \cdots + A_0(\xi',\rho)]
$$
with the $A_j \in S^\circ_{m-1,n}(\widehat R)$ and $U \in S^\circ_{m,n}(\widehat R)$.  Let the $\overline \xi_{m,j}$ be the zeros of $\xi_m^s + A_{s-1}(\xi',\rho)\xi^{s-1} + \cdots + A_0(\xi',\rho)$ in the algebraic closure of the quotient field of $S^\circ_{m,n}(\widehat R)$.  Then the $\overline \xi_{m,j}$ all satisfy
$F(\xi', \overline \xi_{m,j}, \rho,0) = 0$, and hence are algebraic.  The $A_i(\xi',\rho)$ are symmetric functions of the $\overline \xi_{m,j}$ and hence also are algebraic, i.e. in $S_{m,n}^\circ(R)_{alg}$.  It follows immediately that $U(\xi,\rho)$ is also algebraic.

Next we show that the Weierstrass data associated with Weierstrass Division of an algebraic power series by a Weierstrass polynomial in $\xi_m$ are algebraic.  Let $g \in S_{m,n}^\circ(R)_{alg}$, and
$f= \xi_m^s + A_{s-1}(\xi',\rho)\xi^{s-1} + \cdots + A_0(\xi',\rho) =:W$
with the $A_j \in S^\circ_{m-1,n}(R)_{alg}$.  We have that
$$
g(\xi,\rho) = Q(\xi,\rho)\cdot W + \sum_{i=0}^{s-1}{r_i(\xi',\rho) \xi_m^i}
$$
with the $r_i(\xi',\rho) \in S^\circ_{m-1,n}(\widehat R)$.

Assume first that $W$ is irreducible and separable and let $\overline \xi_{m,j}$, $j=1,\dots,s$ be the zeros of $W$ in the algebraic closure of the quotient field of $S^\circ_{m,n}(\widehat R)$.
Since $g(\xi,\rho)$ is algebraic, so is $g(\xi', \overline \xi_{m,j}, \rho)$ for each $j$.
Then we have for $j=1,\dots,s$ that
$$
g(\xi',\overline \xi_{m,j},\rho) = \sum_{i=0}^{s-1}{r_i(\xi',\rho)\overline \xi_{m,j}^i}.
$$
Considering these as equations in the unknowns $r_i$, we see that the coefficient matrix is nonsingular, and hence that the $r_i$ are algebraic.

In the non-irreducible case, we do Weierstrass Division successively by the irreducible factors of $W$, observing that they will also be regular in $\xi_m$.

Finally, we must consider the case that $W$ is irreducible, but not separable, i.e. that $W$ is a polynomial in $\xi_m^p$, where $p$ is the characteristic of $R$, say $W(\xi', \xi_m,\rho) = W_1((\xi', \xi_m^p,\rho)$.  We proceed by induction on the degree of $W$.  $W_1$ has degree ${s \over p} = s' < s$ and is regular of degree $s'$.  Write
$$
g = \sum_{i=0}^{p-1}{g_i(\xi',\xi_m^p,\rho)\xi_m^i}.
$$
The $g_i \in S^\circ_{m,n}(R)_{alg}$ and hence by induction there are $r_{i,j} \in S^\circ_{m-1,n}(R)_{alg}$
such that
$$
g_i(\xi',\eta,\rho) = Q_i(\xi',\eta,\rho)W_1(\xi',\eta,\rho) + \sum_{j=0}^{s'-1}{r_{i,j}(\xi',\rho) \eta^j}.
$$
Replacing $\eta$ by $\xi_m^p$ we see that
\begin{eqnarray*}
g(\xi',\xi_m,\rho) &=& \sum_{i=0}^{p-1}{g_i(\xi',\xi_m^p,\rho)\xi_m^i}\\
&=& \big(\sum_{i=0}^{p-1}{Q_i(\xi',\xi_m^p,\rho)\xi_m^i\Big) \cdot W} + \sum_{i=0}^{p-1}\sum_{j=1}^{s'-1}{r_{i,j}(\xi',\rho)\xi_m^{p\cdot j +i}}.
\end{eqnarray*}
\end{proof}

From Theorem \ref{AWP} and Remark \ref{apprem} we have

\begin{cor}If $K$ is Henselian and $R \subset \Ko$ is either a field or an excellent discrete valuation ring, then
\item[(i)] $\{S^\circ_{m,n}(R)_{alg}\}_{m,n}$ is a separated Weierstrass system.
\item[(ii)] The elements of $S^\circ_{m,n}(R)_{alg}$ define analytic functions
$(\Ko)^m\times(\Koo)^n \to \Ko$. Indeed, these functions are existentially definable uniformly in $K$.  Hence
$K$ has a natural analytic $\{S^\circ_{m,n}(R)_{alg}\}_{m,n}$--structure.
\item[(iii)] Taking $R= \bQ$,  $\bZ_p^{alg}$ the algebraic $p$--adic integers, and $\bF_p$ the $p$--element field, gives prime (i.e. smallest) analytic structures on all Henselian fields, depending only on the characteristic and residue characteristic of the field.
\end{cor}

\section{Subdomains of $K^\circ_{alg}$ and their rings of separated analytic functions}\label{annuli}

In this section, we develop the basis of a theory of analytic
functions on a $K$-annulus (an irreducible $K$-domain in $\Kalgo$
 in the terminology of
\cite{LR1}), when $K$ carries a separated analytic
$\cA(K)$-structure. Here $\cA=\{A_{m,n}\}$ is a (fixed) separated
Weierstrass system and $K$ is a field with separated analytic
$\cA$-structure. We assume in this section that $\ker(\sigma_{0,0})
= (0)$ and hence that $K$ has analytic $\cA(K)$-structure,
cf.~Definition \ref{expand}, Theorem \ref{constants}. Our main
result is that any analytic function on any $K$-annulus equals a
unit times a rational function, cf.~Theorem \ref{ML}. From this it
follows that if $\tau$ is a term in the language of valued fields
with analytic $\cA(K)$--structure, then there is a cover of
$K^\circ_{alg}$ by finitely many $K$--annuli on each of which $\tau$
is a rational function times a strong unit, cf. Theorem \ref{terms}.
This is what is needed for most model-theoretic applications.  In
section \ref{Strongsystems} we prove some stronger results under
stronger assumptions.

The results in this section extend some the results of \cite{CLR1},
Section 3, as well as some of those of \cite{DHM}. They also extend
the classical results of \cite{FP}, Section 2.2, to the
quasi-affinoid and the non-algebraically closed cases. An
alternative definition of analytic structures is given in
\cite{ScanICM} where a statement similar to Theorem \ref{terms} is
an axiom (more or less instead of our axiom (c) which seems  to us
more readily verifiable for the examples in section \ref{sepex}).

We present a more complete theory than the minimum needed for the
above-mentioned result about terms. The reader interested only in
that application should consult Remark \ref{minimum} for the
shortest path to that result.  The complications in the proofs are a
consequence of the lack of completeness -- completeness is a key
ingredient of the proofs in the classical (affinoid) algebraically
closed case. It is also a key ingredient in the development of
affinoid algebra and geometry (\cite{BGR}) and quasi-affinoid
algebra and geometry (\cite{LR1}).  In section \ref{Strongsystems}
we assume a weak consequence of completeness (that holds in most of
the standard examples)  give somewhat easier proofs, and also
establish a second ``Mittag-Leffler'' type theorem (Theorem
\ref{ML2}.)

In this whole section, $\cA$ is a separated Weierstrass system and
any separated analytic $\cA$-structure $\sigma$ is assumed to
satisfy $\ker(\sigma_{0,0}) = (0)$. Hence, if $K$ has analytic
$\cA$-structure, then it has analytic $\cA(K)$-structure,
cf.~Definition \ref{expand}, Theorem \ref{constants}.

With some care, all results and definitions  in this section and
section \ref{seccell} can be adapted to the case of strictly
convergent Weierstrass systems ($\pi\not=1$), where one uses the
maximal algebraic unramified extension in place of $\Kalg$, see
\cite{CLR1}.

\subsection{Definitions and Notation}

\begin{defn}[$K$-annulus]
\label{Rational definitionA} Let $K$ be a Henselian valued field.

(a) A {\em $K$-annulus formula} is a formula $\phi$ of the form
$$
|p_0(x)| \Box_0 \epsilon_0 \wedge \bigwedge_{i=1}^L \epsilon_i
\Box_i |p_i(x)|,
$$
where the $p_i \in K^\circ[x]$ are monic and irreducible, the
$\epsilon_i \in \sqrt{|K\setminus \{0\}|}$ and the $\Box_i \in
\{<,\le\}$. Define $\overline\Box_i$ by $\{\Box_i,
\overline\Box_i\}=\{<, \le\}$. We require further that the sets
$$\cH_i:=\{x\in \Kalg : |p_i(x)| \overline\Box_i \epsilon_i\},\ i=1,\dots,L,$$
be disjoint and contained in $\{x\in
\Kalg:|p_0(x)|\Box_0 \epsilon_0\}$.

(b) The corresponding {\em $K$-annulus} is
$$
\cU_\phi:= \{x \in \Kalg:\phi(x)\}
$$
(If $K_1 \supset \Kalg$ is a field then $\phi$ also defines a subset
of $K_1$. We shall also refer to this as $\cU_\phi$ and also call it
a $K$-annulus. No confusion will result.) The {\em $K$-holes of
$\phi$}, sometimes called the holes of $\cU_\phi$, are the sets $\cH_i$.
A $K$-annulus of the form $\{x\in \Kalg:|p_0(x)|\Box_0 \epsilon_0\}$
is called a {\em $K$--disc} or just a {\em disc}.

(c) A $K$-annulus formula $\phi$ and the $K$-annulus $\cU_\phi$ are
called {\em linear} if the $p_i$ are all linear and the
$\epsilon_i\in |K|\setminus \{0\}$.

(d) A $K$-annulus fomula $\phi$ and the $K$-annulus $\cU_\phi$ are
called \emph{closed} (resp.~\emph{open}) if all the $\Box_i$ are
$\le$ (resp.~$<$).

(e) A $K$-annulus formula is called \emph{good} if the $p_i$ are of
lowest possible degrees among all $K$-annulus formulas defining the
same $K$-annulus.

\end{defn}

Instead of requiring in a) that the $p_i$ be irreducible and
allowing the $\epsilon_i \in \sqrt{|K\setminus \{0\}|}$, we could
require that the $\epsilon_i \in |K \setminus \{0\}|$ and allow
the $p_i$ to be powers of irreducible monic polynomials.

We shall often let $\epsilon$ denote an element of
$\abs{K\setminus\{0\}}$ or an element of $K$ of that size, which
will be clear from the context.

\begin{lem}
\label{annulus_properties}Let $K$ be Henselian.
\begin{itemize}
\item[(i)]  Let $p\in K[x]$ be irreducible and
let $\Box\in\{<,\le\}$. Then for every
$\delta\in\sqrt{\abs{K\setminus\{0\}}}$ there is an
$\epsilon\in\sqrt{\abs{K\setminus\{0\}}}$ such that for every
$x\in\Kalg$, $\abs{p(x)}\Box\varepsilon$ if, and only if, for some
zero $\alpha$ of $p$, $\abs{x-\alpha}\Box\delta$.\
\item[(ii)] A $K$-annulus is a unique finite union of isomophic (and linear)
$\Kalg$-annuli. If $K=\Kalg$ then all $K$-annulus formulas are
linear. \
\item[(iii)] Any two $K$-discs (cf.~Definition~\ref{Rational
definitionA}(b)) $\cU_1$ and $\cU_2$ are either disjoint or one is
contained in the other.\
\item[(iv)] For any two $K$-annuli $\cU_1$ and $\cU_2$, if $\cU_1 \cap
\cU_2 \not=\emptyset$ then $\cU_1 \cap \cU_2$ is a $K$-annulus.\
\item[(v)] The complement of a $K$-annulus $\cU$ (i.e. $\Ko_{alg} \setminus \cU$) is a finite union of $K$-annuli.\
\item[(vi)] Every set of the form
$$
\cU=\left\{x\in\Kalgo: |p_0(x)|\, \Box_0\, \epsilon_0 \wedge
\bigwedge^s_{i=1}\,\epsilon_i\,\Box_i|p_i(x)|\right\}
$$
with the $p_i$ irreducible and the $\epsilon_i \in \sqrt{|K\setminus \{0\}|}$ is described by a $K$-annulus formula.\
\item[(vii)] Every $K$-annulus is described by a good $K$-annulus formula.\
\item[(viii)] If for a finite collection of $K$-annulus formulas $\varphi_i$ the $K$-annuli  $\cU_{\varphi_i}$ cover $\Ko_{alg}$,\
then for any $K' \supset K$ the corresponding $K'$-annuli cover
$K'^\circ$ (cf.~Definition~\ref{Rational definitionA}(b)).\
\end{itemize}
\end{lem}

\begin{proof}
Exercise (or see \cite{CLR1}, Lemma 3.2.)
\end{proof}

\begin{defn}[Rings of analytic functions]\label{def:holo}Let $K$ have separated analytic $\cA$-structure, and let
$\phi$ be a $K$-annulus formula as in Definition \ref{Rational
definitionA} (a). Define the corresponding {\em generalized rings of
fractions} over $\Ko$, resp.~$K$, by
$$
\cO_K(\phi):= A_{m+1,n}(K)/ (p_0^{\ell_0}(x) - a_0z_0,
  p_1^{\ell_1}(x)z_1 -
a_1,\dots,p_L^{\ell_L}(x) z_L - a_L),
$$
and
$$
\cO^\dag_K(\phi) := K\otimes_{\Ko} \cO_K(\phi),
$$
where $a_i \in K^\circ$, $|a_i|= \epsilon_i^{\ell_i}$, $m+n=L+1$,
$\{z_0, \dots, z_L\}$  is the set $\{\xi_2, \dots, \xi_{m+1},
\rho_1, \dots, \rho_n\}$ and $x$ is $\xi_1$ and $z_i$ is a $\xi$ or
$\rho$ variable depending, respectively, on whether $\Box_i$ is
$\le$ or $<$. By Weierstrass Division, each $f\in \cO^\dag_K(\phi)$ defines a function
$\cU_\phi \to \Kalg$ via the analytic structure on $\Kalg$ given by
Theorem \ref{algextn}. Denote this function by $f^\sigma$. Let
$\cO^\sigma_K(\phi)$ be the image of $\cO^\dag_K(\phi)$ under
$f\mapsto f^\sigma$ and call $\cO^\sigma_K(\phi)$ the \em{ring of
analytic functions on $\cU_K(\phi)$}.
\end{defn}

Clearly, $f^\sigma$ maps $\cU_\varphi\cap K'$ into $K'$ for
any field $K\subset K'\subset \Kalg$, cf.~Theorem \ref{algextn}.

The mapping $f\mapsto f^\sigma$ is obviously a surjective $K$--algebra homomorphism.  We do not know if it is always injective, though we do prove it is for some annuli of particularly simple forms (Lemmas
 \ref{MLdisc} and \ref{thin} and Corollary \ref{inj}).  We show in section \ref{Strongsystems} (Theorem \ref{MLL})  that under the stronger assumption that $K$ has \emph{strong} separated analytic structure, this homomorphism is injective for all $K$--annuli.

\begin{defn}[Units]\label{strongunit}
Let $f^\sigma$ be a unit in $\cO^\sigma_K(\phi)$. Suppose that there
is some $\ell\in\mathbb{N}$ and $c\in K$ such that
$\abs{(f^\sigma(x))^\ell}=\abs c$ for all $x\in\cU_\phi$. Suppose also that
there exists a nonzero polynomial $P(\xi)\in\Kt[\xi]$ such that
$P\left(\left(\frac{1}{c}(f^\sigma(x))^\ell\right)^\sim\right)=0$ for all
$x\in\cU_\phi$, where $^\sim:\Kalgo\to \Kalgt$ is the natural
projection to the residue field. Then we call $f$ a \emph{strong
unit}. We call $f$ a \emph{very strong} unit if moreover
$\abs{f(x)}=1$ and $\left(f(x)\right)^\sim=1$ for all $x \in
\cU_{\phi}$.  We call $f \in \cO^\dag_K(\phi)$ a (very) strong unit
if $f^\sigma$ is.
\end{defn}

\begin{exs}
\item[(i)]
If $f \in \cO_K(\phi)$, $f = 1 + g$ with
$$
g \in (A_{m+1,n}(K))^\circ / (p_0^{\ell_0}(x) - a_0z_0,
  p_1^{\ell_1}(x)z_1 -
a_1,\dots,p_L^{\ell_L}(x) z_L - a_L),
$$
then $f$ is a unit by Remark \ref{unit} and since $K$ has
analytic $\cA(K)$-structure. Hence, in this case $f$ is a very strong unit.

\item[(ii)] Let $\varphi$ be the $K$-annulus formula $|x|<1$, and
$\cU_\varphi=\Kalgoo$ be the corresponding open disc. Then
$\cO_K(\varphi)= A_{0,1}$, as expected. Namely, by definition,
$\cO_K(\varphi)=A_{1,1}(K)/(x-\rho_1)$, with $x$ a $\xi$-variable.
But since $x-\rho_1$ is regular in $x$ of degree $1$ (cf.~Definition
\ref{regular}(i)), for any $g$ in $A_{1,1}$ there exist unique $q\in
A_{1,1}$ and $r\in A_{0,1}$ such that $g=q(x-\rho_1)+r$. This shows
that there is a well defined surjective map from $\cO_K(\varphi)$ to
$A_{0,1}$, namely sending the class of $g$ to the class of $r$, with
inverse just the projection from $A_{0,1}$ to $\cO_K(\varphi)$.

\item[(iii)] Let $\varphi$ be the $K$-annulus formula $|x|\leq 1$, and
$\cU_\varphi=\Kalgo$ be the corresponding closed disc. Then clearly
$\cO_K(\varphi)= A_{1,0}$.

\item[(iv)] An example of a strong unit which is not a constant
plus a small function goes as follows. Let $\widetilde K = \bR$. The annulus formula
$|x^2+1|<1$ gives a $K$-annulus $\cU$ on which the identity function
$x\mapsto x$ is a strong unit. The $K$--disc $\cU$ consists of two open $K_{alg}$--discs centered at $i$ and $-i$. For details, see \cite{CLR1}.
\end{exs}

\subsection*{Notation}
With the notation from Definition \ref{def:holo}, in particular,
$\rho=(\rho_1, \dots, \rho_n)$, we define
\begin{eqnarray*}
\cO_K(\phi)^\circ &:=& (\Koo,\rho)\cO_K(\phi)\\
\cO^\dag_K(\phi)^\circ &:=& \cO_K(\phi)\\
\cO^\dag_K(\phi)^{\circ\circ} &:=& \cO_K(\phi)^\circ.
\end{eqnarray*}
We will use the suggestive notation
$f=\sum_\nu{a_\nu(x)({{p_0(x)^{\ell_0}} \over a_0})^{\nu_0}({a_1
\over {p_1(x)^{\ell_1}}})^{\nu_1}\cdots}({a_n \over
{p_n(x)^{\ell_n}}})^{\nu_n}$ to denote the image of
$f=\sum_\nu{a_\nu(\xi_1)z_0^{\nu_0}z_0^{\nu_0}z_1^{\nu_1}\cdots}z_n^{\nu_n}
\in A_{m+1,n}$ in $\cO^\dag_K(\phi)$ or $\cO^\sigma_K(\phi).$

We need to study the rings of functions on annuli of some particularly simple forms.

\begin{defn}
\label{thin-Laurent}
\item[(i)] A linear $K$-annulus of the form
$$
\{x\in\Kalg:\abs{x-a_0}\le\epsilon\ \mathrm{and} \bigwedge_{ i=0}^n \abs{x-a_i}\ge\epsilon\}
$$
for some $\epsilon\in\sqrt{\abs{K\setminus\{0\}}}$, $\epsilon\le1$,
and $a_i\in\Ko$ is called {\em thin}. A general $K$-annulus $\cU$ is
called {\em thin} if each each of the linear $K_{alg}$-annuli
$\cU_i$ corresponding to $\cU$ as in Lemma~\ref{annulus_properties}
(ii) is a thin linear $\Kalg$-annulus.

\item[(ii)] A $K$-annulus $\cU$ of the form
$$
\{x\in\Kalg:\epsilon_1<\abs{p(x)}<\epsilon_0\},
$$
where $p\in K[x]$ is irreducible is called a {\em Laurent annulus}.
We call a Laurent $K$-annulus $\cU$  \emph{simple} if each of the
linear $K_{alg}$-annuli $\cU_i$ corresponding to $\cU$ as in
Lemma~\ref{annulus_properties} (ii) is a Laurent $K_{alg}$-annulus.
A linear Laurent annulus is thus necessarily simple.

\item[(iii)]  If the  linear $K$-annulus
$$
\{x\in\Kalg:\abs{x-a_0}\le\epsilon\ \mathrm{and} \bigwedge_{ i=0}^n \abs{x-a_i}\ge\epsilon\}
$$
is thin and $\epsilon' < \epsilon < \epsilon''$, and $\cU'$, $\cU''$  are defined by
\begin{eqnarray*}
\cU' &:=& \{x\in\Kalg:\abs{x-a_0}\le\epsilon\ \mathrm{and}\abs{x-a_0}\ge\epsilon'\}\ \mathrm{and} \bigwedge_{ i=1}^n\abs{x-a_i}\ge\epsilon\}\\
\cU'' &:=& \{x\in\Kalg:\abs{x-a_0}\le\epsilon''\ \mathrm{and} \bigwedge_{ i=0}^n \abs{x-a_i}\ge\epsilon\},
\end{eqnarray*}
we call $\cU'$, $\cU''$ {\em almost thin}.   A general $K$-annulus $\cU$ is
called {\em almost thin} if each each of the linear $K_{alg}$-annuli
$\cU_i$ corresponding to $\cU$ as in Lemma~\ref{annulus_properties}
(ii) is an almost thin linear $\Kalg$-annulus.  We denote the corresponding thin annuli (obtained by replacing $\epsilon'$ and $\epsilon''$ in the above definitions by $\epsilon$) by $\cU'_t$ and $\cU''_t$ respectively.

\item[(iv)] We call two $K$--annuli \emph{adjacent} if their union is a $K$--annulus.  We say that two $K$--annuli \emph{overlap} if their intersection is nonempty.

\item[(v)] We call a finite cover $\{\cU_i\}$ of $\cU$ by annuli $\cU_i$ \emph{rigid} if for each $i,j$ there is a finite sequence
$\cU_i = \cU_{i_1}, \cdots , \cU_{i_m} = \cU_j$ from the cover such that for each $\ell = 1, \cdots , m-1$ we have that $\cU_{i_\ell}$ and $\cU_{i_{\ell +1}}$ overlap.
\end{defn}

\begin{rem}\label{secondcase}\item[(i)]An open $K$-annulus with only one $K$-hole is Laurent, by
Lemma \ref{annulus_properties}(i).
\item[(ii)] An almost thin annulus is a thin annulus ($\cU$) that has been ``thickened'' either in its hole around $a_0$ ($\cU'$) or in its ``hole at $\infty$'' ($\cU''$).
\end{rem}

The two types of almost thin annuli are equivalent for many purposes. The proof of the following lemma is a simple calculation.

\begin{lem}\label{equiv} The two types of almost thin linear annuli ($\cU'$, $\cU''$) in the definition, and their rings of functions ($\cO(\cU'), \cO(\cU'')$), are related by the change of variables
$$
y = (x-a_1)^{-1}.
$$
\end{lem}

\begin{lem}
\label{cover1}
Let $R(x) \in K(x)$ and $\epsilon \in
\sqrt{|K|\setminus\{0\}}$, let $\Box \in \{<,\le\}$, and let
$$
\cU:=\{x\in\Kalgo : |R(x)|\, \Box \, \epsilon\}.
$$
There are finitely many $K$-annuli $\cU_i$, $i=1,\dots,L$ and a finite set $S \subset \Ko_{alg}$, such that
$\cU \setminus S = \big(\bigcup^L_{i=1}\cU_i \big) \setminus S$ and for each $i$ there is an $f_i \in  \cO^\dag_K(\cU_i)$ such that
$$
R(x)|_{\cU_i \setminus S} = f_i|_{\cU_i \setminus S}.
$$
\end{lem}

\begin{proof} We may assume that $R(x)=x^{n_0}\,\prod^s_{i=1}\, p_i(x)^{n_i}$, where the $p_i\in \Ko[x]$ are monic, irreducible and mutually
prime and the $n_i\in \ZZ$. The lemma is proved in \cite{CLR1}, Lemma 3.16.
\end{proof}

\subsection{Results from \cite{BGR}}\label{fromBGR}
In this subsection we summarize some of the results of \cite{BGR}
that apply without change in our context. Only fairly special cases of these results are actually used
(for example in the proof of Proposition \ref{phipsi}.)

Observe that a $K$-annulus is not always a $K_{alg}$-annulus.
However, by Lemma 5.2 (ii), a $K$-annulus is the disjoint union of a
finite set of $K_{alg}$-annuli (and the corresponding rings of
functions are related as in Proposition \ref{ext}.) These linear $K_{alg}$--annuli are isomorphic via automorphisms of $K_{alg}$ over $K$.  We are thus led
to consider more general domains.

The following is used in Definition \ref{domain} (iii) and
afterwards in Section \ref{gen2}.
\begin{defn}[Residue norm]\label{residuenorm}
With the notation from Definition \ref{def:holo}, for $f \in
\cO_K(\phi)$,
$$
\|f\| \leq 1
$$
means that there exists $g\in  A_{m+1,n}(K)$ whose class in
$\cO_K(\phi)$ is $f$ and such that $\|g\|\leq 1,$ with $\|\cdot\|$
the gauss-norm (see just above Remark \ref{Gauss-norm}).  This usage extends naturally to the case
(Definition \ref{domain} below) that $\varphi$ is a domain formula.
\end{defn}

\begin{defn}\label{domain}
We define the concepts of \emph{a domain formula $\varphi$}, the
corresponding \emph{domain} $\cal U_\varphi$,  \emph{generalized
ring of $\Ko$-fractions} $\cal O_K(\varphi)$,  \emph{generalized
ring of $K$-fractions} $\cal O^\dag_K(\varphi)$ and \emph{ring of
analytic functions } $\cO_K^\sigma(\cU)$ inductively as follows.
\item[(i)] The formula $\varphi\colon=\bigwedge^m_{i=1}(|\xi_i|\leq 1)\wedge
\bigwedge_{i=1}^n (|\rho_i| < 1)$ is a domain formula. The
corresponding domain is $(K^\circ_{alg})^m\times
(K^{\circ\circ}_{alg})^n$, its generalized ring of $\Ko$-fractions
is $A_{m,n}(K)$, its generalized ring of $K$-fractions is
$A^\dag_{m,n} (K):= K \otimes_{\Ko} A_{m,n} (K) $.
\item[(ii)] If $\varphi,\ \cal U_\varphi$,  $\cal O_K(\varphi)$ and $\cal O^\dag_K(\varphi)$
are a domain formula and the corresponding domain and generalized ring of
fractions, and $f_1,\ldots,f_M,\ g_1,\ldots,g_N,\ h\in\cal O^\dag_K(\varphi)$
generate the unit ideal, then
$$
\varphi' :=\varphi\wedge \bigwedge\limits_{i=1}^M (|f_i| \leq
|h|)\wedge \bigwedge\limits^N_{i=1} (|g_i|<|h|)
$$
is a domain formula. The corresponding domain is
$$
 \{x\in \cal U_\varphi \colon \bigwedge\limits_{i=1}^M
 (|f_i (x)|\leq |h(x)|)\text{ and }\bigwedge\limits_{i=1}^N (| g_i(x)| <
|h(x)|)\}
$$
and if $\cal O_K(\phi)= A_{m,n}(K)/I$ then
$$
\cal O_K(\phi')= A_{m+M,n+N}(K)/ I',
$$
where
$$
I'=(I, f_1-\eta_1 h,\ldots,f_M-\eta_M h,\ g_1-\tau_1
h,\ldots,g_N-\tau_N h).
$$
\item[(iii)]We define the corresponding generalized ring of ($K$-valued) fractions $\cO^\dag_K(\phi)$
on $\cal U_\varphi$ by
$$
\cO_K^\dag(\phi):= K\otimes_{\Ko}\cO_K(\phi).
$$
This is consistent with the notation $A_{m,n}^\dag(K) =
K\otimes_{\Ko}A_{m,n}(K)$ from (i). We define
$$
\cO_K^\dag(\phi)^\circ := \{f \in \cO_K(\phi): \|f\| \leq 1\},
$$
and
$$
\cO_K^\dag(\phi)^{\circ\circ} := (\Koo,\rho)\cO_K^\dag(\phi)^\circ,
$$
where $\cal O_K(\phi)= A_{m,n}(K)/ I$ and $\rho = (\rho_1,\cdots, \rho_n).$
\item[(iv)]The corresponding ring of functions $\cO_K^\sigma(\phi)$ is the image of  $\cO_K^\dag(\phi)$ in the ring of functions  $\cU_\phi \to K_{alg}$  under the mapping provided by the analytic $\cA(K_{alg})$-structure on $K_{alg}$.
\item[(v)]We call a domain \emph{closed} if all its defining inequalities are
weak (i.e.~$\leq$) and we call a domain \emph{open} if all its
defining inequalities are strict (i.e.~$<$). In the case that $\cal
U$ is closed its ring of $\Ko$-fractions is of the form $A_{m,0}/ I$
and in the case it is open its ring of $K$-fractions  is of the form
$A_{0,n}/ I$.
\end{defn}

The following definition generalizes the notion of domains, and will
only be used in the proofs later on in section \ref{gen2}.  No
confusion should result with the notation of Definition
\ref{domain}.
\begin{defn}\label{domain(vi)}
If $J$ is an ideal in $A^\dag_{m,n}(K),$ define
$$
V(J):= \{(a,b) \in (\Ko_{alg})^m\times(\Koo_{alg})^n \colon
f^\sigma(a,b) = 0, \text{ for all } f \in J\} \text{ and }
$$
$\cO_K^\dag(J) := A^\dag_{m,n}(K)/J$. The elements of $\cO^\dag(J)$
define functions on $V(J) \to K_{alg}$ via $\sigma$ in a natural
way, and we denote this ring of functions by $\cO_K^\sigma(J)$.
\end{defn}

\begin{rem}(i) If $\varphi$ is a $K$-annulus formula and $F$ is an algebraic extension of $K$, then $\varphi$ is a $F$-domain formula, even though $\varphi$ may not be an $F$--annulus formula.

(ii) Many of the results  for domains, of \cite{BGR} in the affinoid
case, and of \cite{LR1} in the quasi-affinoid case, hold also in the
more general setting of Henselian fields with analytic structure,
often with (slightly) modified proofs. To check exactly how proofs
must be modified would be a cumbersome task. For example Proposition
\ref{ext} (ii) below would follow from the analogue of \cite{BGR}
Lemma 7.2.2.8 in the affinoid case. But the proof  of that Lemma in
\cite{BGR} uses much of that book. We will give a self-contained
treatment, though we will quote some theorems from \cite{BGR} when
the proofs apply without modification.

(iii) From the Strong Noetherian Property we see that $A^\dag_{m,n}(K)^\circ = A_{m,n}(K)$ and
$A^\dag_{m,n}(K)^{\circ\circ} = A_{m,n}(K)^\circ.$  It is easy to see that for $f \in A^\dag_{m,n}(K)$ we have
$f \in A^\dag_{m,n}(K)^\circ$ if and only if $|f^\sigma(a,b)| \leq 1$ for all $(a,b) \in (\Ko_{alg})^m\times(\Koo_{alg})^n,$ which is the case if, and only if, $\|f\| \leq 1,$
and that $f \in A^\dag_{m,n}(K)^{\circ\circ}$ if, and only if, $|f^\sigma(a,b)| < 1$ for all
$(a,b) \in (\Ko_{alg})^m\times(\Koo_{alg})^n,$ which is the case if, and only if,
$f \in (\Koo,\rho)A_{m,n}(K)^\circ$.
\end{rem}

From Theorem \ref{constants} we know that the rings $A_{m,0}(K)$ and
$A_{0,n}(K)$ have Weierstrass Preparation, are closed under
Weierstrass changes of variables, and (from the Strong Noetherian
Property) have the property that for any $0 \neq f \in A_{m,0}(K)$
(or $A_{0,n}(K)$), there is an $a \in K$ such that $af \in
A_{m,0}(K)$ (or $A_{0,n}(K)$) satisfies $\|af\| = 1$ and is
preregular (and hence after a Weierstrass change of variables,
regular) of some degree.

\begin{prop} \label{Ruck}The R\"uckert Theory of \cite{BGR} sections 5.2.5 and 5.2.6 applies to the rings $A_{m,0}^\dag(K)$ and to the rings $A_{0,n}^\dag(K)$.
Hence these rings are Noetherian, factorial and normal (integrally
closed in their fields of fractions.)
\end{prop}

\begin{proof}The proofs of \cite{BGR} apply.
\end{proof}

From the strong Noetherian Property and Weierstrass Preparation we
have (cf. \cite{BGR} Corollary 6.1.2.2)

\begin{prop}(Normalization)  \label{norm}(i)\ Let $J$ be an ideal in $A^\dag_{m,0}$.
There is a nonnegative integer $d$ and a finite monomorphism
$$
A^\dag_{d,0}(K)\hookrightarrow A^\dag_{m,0}(K)/ J.
$$
(ii)\ Let $J$ be an ideal in $A^\dag_{0,n}$. There is a nonnegative
integer $d$ and a finite monomorphism
$$
A^\dag_{0,d}(K)\hookrightarrow A^\dag_{0,n}(K)/ J.
$$
\end{prop}

As in \cite{BGR} section 7.1.2 we obtain

\begin{prop}\label{Null}The rings $A^\dag_{m,0}(K)$ and $A^\dag_{0,n}(K)$ satisfy the Nullstellensatz.
There is a one-to-one correspondence between the maximal
ideals of $A^\dag_{m,0}$ (resp.~$A^\dag_{0,n})$ and the orbits of
$(K^\circ_{alg})^m$ (resp.~$(K^{\circ\circ}_{alg})^n$) under the
Galois group of $K_{alg}$ over $K$.
\end{prop}

\begin{rem}[Norms] Let $J$ be an ideal in $A^\dag_{m,n}(K)$.
Then the elements of $A^\dag_{m,n}/ J$ define functions on
$V(J)=\{x\in (K^\circ_{alg})^m\times (K^{\circ\circ}_{alg})^n\colon
f(x)=0\ \text{ for all } \ f\in J\}$. Since the proof of \cite{LL2}
or \cite{LR2} (see also \cite{LR3}, Theorem 5.2) shows that
$K_{alg}$ admits quantifier elimination in the language with
function symbols for the elements of $\cS=\bigcup_{m,n}
A^\dag_{m,n}$, we see that the supremum seminorm on $V(J)$,
$\|\cdot\|_{\sup}$ is well defined and takes values in $|K_{alg}|$
(where $\|f\|_{\sup}$ for $f\in A^\dag_{m,n}/ J$ is defined as the
supremum over all $x$ in $V(J)$ of the $|f(x)|$). If
$A^\dag_{m,n}/J$ is reduced then $\|\cdot\|_{sup}$ is a norm. The
residue norm $\|\cdot\|_J$ on $A^\dag_{m,n}/ J$ is well defined in
the standard examples (where $\|f\|_J$ is defined as the infimum of
the gauss-norms of all representatives of $f$ in $A^\dag_{m,n}$). We
don't know if our assumption on $\cS$ are sufficient to guarantee
that residue norms are always defined and, if $J$ is reduced,
equivalent to the supremum norm. However, the notation of Definition
\ref{residuenorm} always makes sense, even if the residue norm is
not defined. Clearly, for $f\in A^\dag_{m,n}$, the supremum norm of
$f$ equals the gauss-norm of $f$.
\end{rem}

\begin{rem} Normalization fails in general for ideals $J \subset
A^\dag_{m,n}(K)$ when $m,n > 0$ (cf. \cite{LR1} Example 2.3.5.)  It
is likely that the standard properties (the Nullstellensatz, unique
factorization, etc.) could be established for these rings by
adapting the proofs of \cite{LR1}. This could be quite nontrivial to
carry out and is not needed for the results of this paper.
\end{rem}

\begin{defn}With the notation from Definition
\ref{domain(vi)}, we shall call $f\in\cal O^\dag(J) (=
A^\dag_{m,n}/J$) \emph{power bounded} if $|f^\sigma(x)|\leq 1$ for
all $x\in V(J)$ and we shall call $f$ \emph{topologically nilpotent}
if $|f^\sigma(x)|<1$ for all $x\in V(J)$. We shall say that
$f\in\cal O^\dag(V(J))$ is \emph {strongly power bounded} if $f$
satisfies
$$
f^n=A_1 f^{n-1}+A_2 f^{n-2}+\ldots+A_n
$$
where for each $i$ there is an $a_i\in A_{m,n}(K)$ such that $a_i\in
A_i+J$, with $A_i$ in $A^\dag_{m,n} (K) $. We say that $f$ is \emph{
strongly topologically nilpotent} if the $a_i\in A_{m,n}(K)^{\circ}
= (K^{\circ\circ},\rho)A_{m,n}(K)$. We call $f^\sigma$ (strongly)
power bounded or (strongly) topologically nilpotent if $f$ is.
\end{defn}

\begin{rem}Observe that $f\in A^\dag_{m,n}(K)$ is power bounded if and only if $f\in A_{m,n}
(K)$ and $f$ is topologically nilpotent if and only if $f\in
A_{m,n}(K)^{\circ}$.  Since we allow $K$ to have rank $>1$, it
is possible that there are $\alpha, \beta \in K$ with $1 < |\alpha|$
and $|\alpha^n| < |\beta|$ for all $n \in \bN.$  By our definition
$\alpha$ is not power bounded.
\end{rem}

We have

\begin{prop}\label{int}(\cite{BGR} Proposition 3.8.1.7) Let $\Phi\colon B\to A$ be an integral torsion free
$K$-algebra monomorphism between two $K$-algebras $A$ and $B$, where
$B$ is an integrally closed integral domain. Then one has
\item[(a)]$|f|_{\sup}=\max\limits_{1\leq i\leq n} |b_i|^{1/i}_{\sup}$ for $f\in A$ where
$$
f^n+\Phi(b_1) f^{n-1}+\ldots+\Phi(b_n)=0
$$
is the (unique) integral equation of minimal degree for $f$ over
$\Phi(B)$.

\item[(b)]In the case that $B=A^\dag_{m,0}$, or $B = A^\dag_{0,n}$, if $f$ is power bounded, then $f$ is strongly power bounded and if $f$ is topologically nilpotent, then $f$ is strongly topologically nilpotent.
\end{prop}

\begin{prop}\label{stronglybounded}Let $\phi$ be a domain formula and let $f_1,\ldots,f_M,g_1,\ldots,g_N,\ h\in\cal O^\dag_K(\varphi)$ and suppose that $h$ is a unit in $\cal O^\dag_K(\varphi)$, each ${f_i\over h}$ is strongly power bounded and each ${g_i\over h}$ is strongly topologically nilpotent.
Let $\psi=\varphi\wedge\bigwedge^M_{i=1} |f_i|\leq
|h|\wedge\bigwedge^N_{j=1} |g_j| < |h|$. Then $\cal O^\dag_K(\psi)=\cal
O^\dag_K(\varphi)$.
\end{prop}

\begin{proof}\cite{BGR} Proposition 6.1.4.3 or \cite{LR1} Proposition 5.3.2.
\end{proof}

From Propositions \ref{Ruck}, \ref{norm}, \ref{int} and
\ref{stronglybounded} we have the following Corollary.  In Corollary \ref{canon} we prove a similar result for all $K$--annuli.

\begin{cor}\label{canon}If $\varphi$ is an open or a closed domain formula then $\cal O_K(\varphi)=\cal O_K(\cal U_\varphi)$ -- i.e.~$\cal O_K(\varphi)$ depends only on the domain $\cal U_\varphi$ and not on the particular description $\varphi$.
\end{cor}

\begin{rem}\label{bigdisc}  In case 2 of the proof of Proposition \ref{phipsi} we shall
use annuli or discs of radius $>1$ and their associated rings. To
give an example of such a disc, if $a \in K$, $|a|>1$, the annulus
formula $\phi := (|x|<|a|)\wedge(\epsilon_1<|p_1(x)|)$ defines the
annulus $\cU_\phi = \{x\in K_{alg}\colon \phi(x)\}$ and generalized
ring of fractions
$$
\cO_K^\dag(\phi):= \Big \{f\Big({\rho_1 \over a}, {a_1 \over p_1 ^{\ell_1}(\rho_1)}\Big)\colon f \in A_{0,2}\Big \}.
$$
Here $|a_1| = \epsilon_1^{\ell_1} \in |K|.$ It is clear that
$$
\cO_K^\dag(\phi) \hookrightarrow \cO_K^\dag(\epsilon_1 < |p_1(x)|) = \cO_K^\dag(|x| \leq 1 \wedge \epsilon_1 < |p_1(x)|)
$$
by $\rho_1 \mapsto \xi_1$. In this example $\phi$ is an open annulus
formula, $\cU_\phi$ is open and Propositions \ref{Ruck}, \ref{norm},
\ref{Null}, \ref{int} and Corollary \ref{canon} apply to the rings
$\cO_K^\dag(\phi)$ and  $\cO_K^\sigma(\phi)$.
\end{rem}

\subsection{General $K$-annuli, part 1}\label{general}

In this subsection we prove some preliminary results about $K$--annuli.

\begin{prop}\label{zeros}  Let
$$
\varphi := |p_0(x)| \Box_0 \epsilon_0 \wedge \bigwedge_{i=1}^L \epsilon_i
\Box_i |p_i(x)|,
$$
be a good annulus formula and suppose that $P(x) \in \Ko[x]$ is monic and irreducible and that $\alpha \in \cU_K(\varphi)$ is a zero of $P(x)$.  Let $f \in \cO_K^\dag(\varphi)$ and suppose that $f^\sigma(\alpha) = 0$.  Then there is a $g \in \cO_K^\dag(\varphi)$ such that $f(x) = P(x)\cdot g(x)$, i.e. $P(x)$ divides $f(x)$ in
$\cO_K^\dag(\varphi)$.
\end{prop}
\begin{proof}  We may suppose that
$$
f=\sum_\nu{b_\nu(x)\Big({{p_0(x)^{\ell_0}} \over a_0}\Big)^{\nu_0}\Big({a_1
\over {p_1(x)^{\ell_1}}}\Big)^{\nu_1}\cdots}\Big({a_n \over
{p_n(x)^{\ell_n}}}\Big)^{\nu_n}
$$
where the $b_\nu(x)$ are of degrees $< \ell_0 \deg(p_0)$.  Let $Q_0 \in\Ko[x]$ be the minimal polynomial of
${p_0^{\ell_0}(\alpha) \over a_0}$, and let $Q_i$ be the minimal polynomial of ${a_i \over p_i^{\ell_i}(\alpha)}$, for $i = 1, \cdots, n$.
Then each $Q_i$ is regular in the appropriate sense for a variable corresponding to the inequality $\Box_i$, and by Weierstrass Division
$$
f(x) = R(x) + g_0 \cdot Q_0\Big({{p_0(x)^{\ell_0}} \over a_0}\Big) + \sum_{i=1}^n{g_i \cdot Q_i\Big({a_i \over p_i(x)^{\ell_i}}\Big)}
$$
where $R(x)$ is a rational function in $K(x)$ whose denominator is a product of powers of the $p_i(x), i = 1,\cdots, n$.  Let $R(x) \equiv r(x) \mod P(x)$ where $r(x) \in K[x]$ has degree $< \deg(P)$ and note that
$$
Q_0\Big({{p_0(x)^{\ell_0}} \over a_0}\Big) \equiv 0 \mod P(x)
$$
and, for $i=1, \cdots,n$,  that
$$
Q_i\Big({a_i \over p_i(x)^{\ell_i}}\Big) \equiv 0 \mod P(x)
$$
in $\cO^\dag_K(\varphi)$.  Thus we have written
$$
f(x) = r(x) + g(x)\cdot P(x)
$$
where $g(x) \in \cO^\dag_K(\varphi)$ and $\deg(r) < \deg(P)$.  Then $r(\alpha) = 0$ and hence $r(x) \equiv 0$.
\end{proof}

\begin{prop}\label{phipsi}
If $\phi$ and $\psi$ are $K$-annulus formulas and $\cal
U_\phi\subseteq\cal U_\psi$ then $\cal O_K^\dag(\psi)\subseteq\cal
O_K^\dag(\phi)$.
\end{prop}

\begin{proof}Let
\begin{eqnarray*}
\varphi&=&p_0 (x)\square_0\epsilon_0\wedge\bigwedge\limits^n_{i=1}\epsilon_i\square_i p_i(x)\\
\psi&=&p'_0(x)\square'_0\epsilon'_0\wedge\bigwedge\limits^{n'}_{i=1}\epsilon'_i\square'_i p'_i (x).\\
\end{eqnarray*}
By Corollary \ref{stronglybounded} it is enough to show:
if $\square_0'$ is $\leq$ then ${p_0'^{\ell_0'} \over a_0'}$ is strongly power bounded in
$\cO_K^\dag(\phi)$;  if $\square_0'$ is $<$ then ${p_0'^{\ell_0'} \over a_0'}$ is strongly topologically nilpotent in
$\cO_K^\dag(\phi)$;
if $\square_i'$ is $\leq$ then
$\frac{a_i'} {p_i'^{\ell_i'}}$ is strongly power bounded in
$\cO_K^\dag(\phi)$;  and if $\square_i'$ is $<$ then
$\frac{a_i'}{p_i'^{\ell_i'}}$ is strongly topologically nilpotent in
$\cO_K^\dag(\phi)$.  We will check some of these cases and leave the rest
to the reader.

\begin{case}$\square_0$ is $\leq$.
\emph{Consider $\varphi_0 := |p_0 (x)|\leq\epsilon_0$. Then $\cal
O^\dag_K(\varphi_0)=A^\dag_{2,0}/(p_0^{\ell_0} (\xi_1)-a_0 \eta_1)$.  (Here
$|a_0^{\ell_0}|=\epsilon_0$.) Since $p_0(\xi_i)\in K^\circ[x]$ is
monic and $|a_0| \leq 1$, $\cal O^\dag_K(\varphi_0)$ is a finite extension of $A^\dag_{1,0}$ (in
the variable $\eta_1$) and hence  $p'_0(\xi_1)$ is integral over $A^\dag_{1,0}$. If $\square'_0$ is $\leq$ then ${(p'_0)^{\ell'_0}\over a'_0}$ is power bounded and hence, by
Proposition \ref{int}, strongly
power bounded in $\cal O^\dag_K(\varphi_0)\subset\cal O^\dag_K(\varphi)$; and if $\square'_0$ is $<$ then ${(p'_0)^{\ell'_0}\over a'_0}$ is topologically nilpotent and hence, by
Proposition \ref{int}, strongly
topologically nilpotent in $\cal O^\dag_K(\varphi_0)\subset\cal O^\dag_K(\varphi)$.
The result follows from Corollary \ref{stronglybounded}.}
\end{case}

\begin{case}$\square_1$ is $<$ and the ``$K$-hole"
$\cH_1' = \{x \colon |p_1'(x) \overline{\square}_1' \epsilon_1'\}$
is contained in the ``$K$-hole" $\cH_1 = \{x \colon |p_1(x) \leq
\epsilon_1'\}$.
 \emph{ Let $a \in K$, $1 < |a|$ and consider the annulus formula (cf. Remark \ref{bigdisc})
$\phi_1 := (|x| < |a|) \wedge (\epsilon_1 < |p_1 (x)|)$. Then, using
that Remark, $\cO_K^\dag(\phi_1) \subset \cO_K^\dag(\phi)$ and
$\phi_1$ is open. Hence by Proposition \ref{Null} $\frac{p_1'^{\ell_1'}}{a_1'}$
is a unit in $\cO_K^\dag(\phi_1)$ and thus also in
$\cO_K^\dag(\phi)$. Thus $\frac{a_1'}{p_1'^{\ell_1'}} \in
\cO_K^\dag(\phi_1)$, which is topologically nilpotent, is strongly
topologically nilpotent by Proposition \ref{int}.}
\end{case}
The other cases are similar.
\end{proof}

From Proposition \ref{phipsi}  we have

\begin{cor}\label{canon2}If $\varphi$ is a $K$-annulus formula then $\cal O_K^\dag(\varphi)$ depends only on the underlying $K$-annulus $\cal U_\varphi$ and not on the particular presentation.
If $\varphi$, $\psi$ are two $K$-annulus formulas and $\cal
U_\varphi\subseteq\cal U_\psi$ then $\cal O_K^\dag(\cal
U_\psi)\subseteq\cal O_K^\dag(\cal U_\varphi)$.
\end{cor}

The following Proposition is a special case of Corollary
\ref{canon2}.  The proof in the linear case is direct and does not
use Proposition \ref{phipsi}, so we include it as a concrete
example.

\begin{prop}\label{linearcanon}For a linear annulus formula $\varphi$, $\cal O^\dag_K(\varphi)$ depends only on $\cal U_\varphi$ and not on the particular description $\varphi$.
\end{prop}

\begin{proof}If $\alpha_0'$ is another center of the disc $\{x\colon|x-\alpha_0| \square_0 \epsilon_0\}$ then ${x-\alpha'_0\over\epsilon_0}={x-\alpha_0\over\epsilon_0}+{\alpha'_0-\alpha_0\over\epsilon}$ and ${\alpha'_0-\alpha_0\over\epsilon_0} \leq 1$ if $\square_0$ is $\leq$, and ${\alpha'_0-\alpha_0\over\epsilon_0} < 1$ if $\square_0$ is
$<$. Similarly for a hole
$\{x\colon\epsilon_1\square_1|x-\alpha_1|\}$, if $\alpha'_1$ is
another center then we have
$$
{\epsilon_1\over x-\alpha'_1}={\epsilon_1\over
(x-\alpha_1)+(\alpha_1-\alpha'_1)}={\epsilon_1\over x-\alpha_1}
\Big(1-{\alpha'_1-\alpha_1\over x-\alpha_1 }\Big)^{-1}.
$$
If $\square_1$ is $\leq$ then $|\alpha'_1-\alpha_1| <
\epsilon_1$ and if $\square_1$ is $<$ then
$|\alpha'_1-\alpha_1|\leq\epsilon_1$.
\end{proof}

The following Lemma is a small extension of \cite{CLR1}, Lemma 3.11.

\begin{lem}\label{Annulusdecomp}
Every $K$-annulus is a finite union of
\begin{itemize}
\item[(i)]thin $K$-annuli with good descriptions of the form
$$
|p_0(x)|\leq\epsilon_0\wedge\epsilon_0 \leq |p_0(x)| \wedge \bigwedge\limits^n_{i=2}\epsilon_i\leq |p_i(x)|,
$$
\item[(ii)]Simple Laurent $K$-annuli (cf. Definition \ref{thin-Laurent}) (with good descriptions of the form
$|p_0(x)| < \epsilon_0 \wedge \epsilon_1 < |p_1(x)|$), and
\item[(iii)]Open $K$-discs (with good descriptions of the form
$|p_0(x)| < \epsilon_0$.)
\end{itemize}
\end{lem}

Next we prove Mittag-Leffler type decompositions for $K$--annuli of some special types, namely discs (Lemma \ref{MLdisc}) and thin annuli (Lemma \ref{thin}).

\begin{lem}\label{MLdisc}Let $\cU$ be a $K$--disc and let $0 \neq f \in \cO_K^\dag(\cU)$.  Then there is a unique monic polynomial $P(x) \in K[x]$ and strong unit $E \in \cO_K^\dag(\cU)$ such that $f = P(x) \cdot E$.  In particular, the mapping $f \mapsto f^\sigma$ in injective.
\end{lem}
\begin{proof}  The case of a linear $K$--disc is immediate by Weierstrass Preparation, since in that case
$f = \sum a_i ({x - \alpha \over \epsilon})^i$ with the $a_i$ constants.

For $\cU$ an arbitrary $K$--disc we write by Weierstrass Division
$$
f = \sum_i{a_i(x)\Big({p(x) \over \epsilon}\Big)^i}
$$
where the $a_i$ have degrees $< \deg(p(x))$ and hence are either zero or strong units on $\cU$.  Considering $f$ on each of the $K_{alg}$--discs into which $\cU$ decomposes, we see that $\|f\|_{sup} = \max_i\{\|a_i\|_{sup}\}$.
Hence $f \mapsto f^\sigma$ is injective and from the linear case we see that $f^\sigma$ has only finitely many zeros.  The result now follows from Proposition \ref{zeros} and the Nullstellensatz \ref{Null}.  A unit on a disc is necessarily a strong unit.
\end{proof}

\begin{lem}\label{thin} If $\cal U_\varphi$ is a thin $K$-annulus described by a good $K$-annulus formula
$$
\phi :=|p_0(x)|\leq\epsilon_0\wedge\epsilon_0 \leq |p_0(x)| \wedge \bigwedge\limits^n_{i=2}\epsilon_i\leq |p_i(x)|
$$
and $0 \neq f \in \cal O_K^\dag(\cal U_\phi)$, there is a monic $P(x)\in K[x]$ all of whose zeros lie in $\cal U_\phi$, a strong unit $E \in \cal O^\dag_K(\cal U_\phi)$ and integers $n_i$ such that
\begin{equation}\label{mlK1}
f=P(x)\cdot\prod^n_{i=1} p_i (x)^{n_i}\cdot E.
\end{equation}
Furthermore, this representation is unique.  The mapping $f \mapsto f^\sigma$ is injective.
\end{lem}

\begin{proof}
Write
$$
f= \sum_{i,j,\nu}{a_{ij\nu}(x)\Big({p_0(x)\over\epsilon_0}\Big)^i\Big({\epsilon_0 \over p_0(x)}\Big)^j\Big({\epsilon\over p}\Big)^{\nu'}}
$$
where $\nu'=(\nu_2,\dots,\nu_n)$ is a multi-index and
$({\epsilon\over p})^{\nu'} = ({\epsilon_2\over p_2})^{\nu_2} \cdots ({\epsilon_n\over p_n})^{\nu_n}.$ (We are
implicitly assuming that $\epsilon_i\in |K|-\{0\}$. The general
case, with $\ell_i > 1$ is only notationally more cumbersome. We are also abusing notation by using $\epsilon_i$ for both an element of $|K|$ and an element of $K$ of that size.)
Let $p(x) := \prod_{i=1}^np_i(x)$, $\epsilon := \prod_{i=1}^n\epsilon_i$ and
$N := \sum_{i=1}^n n_i$, where $n_i$ is the degree of $p_i(x)$.  (We are taking $p_1(x) = p_0(x)$ and $\epsilon_1 = \epsilon_0$.)

Observe that $\cal U_{\phi}$ is also defined by the condition $|p(x)| = \epsilon$ and that
$p(x) \in \Ko[x]$ is monic.  (Outside the $K$--disc $\{x: |p_0(x)| \leq \epsilon_0\}$ we have $|p(x)| > \epsilon$
and in the holes of $\cal U_{\phi}$ we have $|p(x)| < \epsilon$.)  Hence, by Proposition \ref{phipsi} (in fact a small extension to the case that $\psi$ is presented by $|p(x)| = \epsilon$) we have that

$$
f = \sum_{ij}a_{ij}(x) \Big({p(x) \over \epsilon}\Big)^i \Big({\epsilon \over p(x)}\Big)^j.
$$
Dividing $\sum a_{ij}\xi_0^i\xi_1^j$ by $p(x) - \epsilon\xi_0$, which is regular in $x$ of degree $N$,
we may assume that the $a_{ij}(x)$ all have degrees $< N$.

This representation is far from canonical because of the terms
$({p(x)\over\epsilon})^i({\epsilon \over p(x)})^j$.  If our
Weierstrass system were \emph{strong} (cf. Section
\ref{Strongsystems}) we could work with the canonical representation
$f = f_1(x,{p\over\epsilon}) + {\epsilon\over p}f_2(x,{\epsilon\over
p})$, with
coefficients of degrees less than $N$. The proof is much easier in this case  --  see Remark \ref{strongthin}
below.   In the
absence of the strongness assumption, we proceed as follows to get a
canonical representation.  We have
$$
f\equiv \sum_{i,j}{a_{ij}(x)\xi_0^i\xi_1^j}  \text{ modulo }(\epsilon \xi_0 - p(x), \xi_1p(x) - \epsilon).
$$
Doing Weierstrass division by $\xi_0\xi_1 - 1$ in the variables $\xi_0$ and $\xi_0 - \xi_1$ (i.e. writing
$\eta:=  \xi_0 - \xi_1$, so $\xi_1 = \eta + \xi_0$ and $\xi_0\xi_1 - 1 = \xi_0^2 + \xi_0\eta -1$, which is regular in $\xi_0$ of degree $2$) we obtain
$$
f\equiv \sum_{i,}{b_{i}(x)(\xi_0-\xi_1)^i}  +   \xi_0 \sum_{i}{c_{i}(x)(\xi_0-\xi_1)^i}
$$
modulo $(\epsilon\xi_0 - p(x), \xi_1p(x) - \epsilon_0, \xi_0\xi_1 - 1)$, where the $b_{i}$ and $c_{i}$ have degrees less than $N$.
Next we observe that for representations of this form
$$
\|f\|_{sup} = max_{i,}\{\|b_{i}\|_{sup}, \|c_{i}\|_{sup}\},
$$
and that
$\|\cdot\|_{sup}$ is a multiplicative norm.
From the Strong Noetherian Property there are only finitely many ``biggest" terms.
Consider $|b_i(x)({p(x) \over \epsilon} - {\epsilon \over p(x)})^i|$ and
$|c_j(x){p(x) \over \epsilon}({p(x) \over \epsilon} - {\epsilon \over p(x)})^j|$ for $x$ ``just" outside the outer edge of $\cal U_{\phi}$ and ``just" inside the holes in $\cal U_{\phi}$, and the fact that the degrees of the $b_i, c_j$ are $< N$, to see that there cannot be any cancellation among ``biggest" terms.  This proves injectivity.

By the Strong Noetherian Property
and Corollary \ref{VSSNP} we can write
$$
f=\sum_{i \leq s}{b_{i}(x)\Big({p(x)\over \epsilon} - {\epsilon \over p(x)}\Big)^i(1 + g_{i})} +
 {p(x)\over\epsilon}\sum_{i \leq s}{c_{i}(x)\Big({p(x) \over \epsilon} - {\epsilon \over p(x)}\Big)^i(1 + g'_{i})}
$$
where the $g_{i}, g'_{i} \in \cal O^\circ_K(\cal U)$.

Let
$$
R:= \sum_{i \leq s}{b_{i}(x)\Big({p(x)\over \epsilon} - {\epsilon \over p(x)}\Big)^i} +
 {p(x)\over\epsilon}\sum_{i \leq s}{c_{i}(x)\Big({p(x) \over \epsilon} - {\epsilon \over p(x)}\Big)^i.}
$$
Then $\|f(x) - R(x)\|_{sup} <  \|R\|_{sup} = \|f\|_{sup}$, and hence, except on finitely many discs $\cU'_j$ contained in $\cU$ (around the zeros of $R$) we have for  $x \in \cU$ (i.e. for all $x \in \cU \setminus \bigcup_j \cU'_j$) that $|f(x) - R(x)| < |f(x)| = |R(x)| = \|f\|_{sup} =\|R\|_{sup}$.  Clearly, $f^\sigma$ is not identically zero on any of these discs $\cU'_j$. Hence, by Lemma \ref{MLdisc} $f^\sigma$ has only finitely many zeros in these discs.  Hence, $f$ has only finitely many zeros and by Proposition \ref{zeros} we may divide by a suitable polynomial $P(x) \in K[x]$ and reduce to the case that $f$ has no zeros and hence, by the Nullstellensatz (Proposition \ref{Null}) is a unit.  Since a unit on a disc is a strong unit, we see that now
$|f(x)| = \|f\|_{sup}$ for all $x \in \cU$.
Hence $R$ is also a unit, and we have that $f = R\cdot E$ for $E$ a unit in $\cal O_K(\cal U)$.  Since $|f(x) - R(x)| < |f(x)|$ for all $x \in \cU$, $E$ is a strong unit.
Note that $R(x)$ is a rational function with denominator of the form $\prod p_i(x)^{m_i}$.  The numerator has no zeros in $\cal U$, since $f$ is a unit.  Let $Q(x)$ be an irreducible factor of the numerator.  If one of (and hence all of) the zeros of $Q$ lies outside the disc $\{x: |p_0(x)| \leq \epsilon_0\}$, $Q$ is a strong unit on $\cal U$.   If one of (and hence all of) the zeros of $Q$ lies inside the hole  $\{x: |p_i(x)| < \epsilon_i\}$, then by iterated use of the following Lemma \ref{lemthin}, there is a strong unit $E'$ and an $\ell \in \bN$ such that
$Q(x) = (p_i(x))^{\ell} \cdot E'$, and this completes the proof of existence.  Uniqueness follows from the observation that $P$ is determined by the zeros of $f$ and that $\prod_i{p_i^{n_i}}$ is a strong unit only when $n_i = 0$ for all $i$.
\end{proof}

\begin{lem}\label{lemthin} Let $|p(x)| < \epsilon$ be a good description of a $K$--disc $\cal H$, and let $P(x) \in \Ko[x]$ have all its zeros in $\cal H$.  Then there is a strong unit $E$ on the annulus $\cal U = \{x: |p(x)| = \epsilon\}$ such that $P(x) = p(x)\cdot q(x)\cdot E$ for some $q(x) \in K[x]$ of degree less than that of $P(x)$.  Indeed, the conclusion is true on the whole annulus $\cal U'' = \{x: |p(x)| \geq \epsilon\}$.
\end{lem}
\begin{proof}It is sufficient to consider $P$ irreducible.  By Euclidean division
$$
P(x) = p(x)\cdot q(x) + r(x),
$$
where degree $r(x)$ $<$ degree $p(x)$.  Hence either $r(x) = 0$ or $r(x)$ is a strong unit on $\cal H$.  $\cal H$ is also described by $|P(x)| < \epsilon' \leq \epsilon$, since $\deg P(x) \geq \deg p(x)$.
Let $\alpha$ be a zero of $p(x)$.  Then
$\epsilon \geq \epsilon' > |P(\alpha)| = |r(\alpha)|$.  So $\|r\|_{sup_{\cal H}} < \epsilon'$, where
$\|\cdot\|_{sup_{\cal H}}$ is the supremum norm on $\cal H$.  Hence also $\|r\|_{sup_{\cal U}} < \epsilon'$, where
$\|\cdot\|_{sup_{\cal U}}$ is the supremum norm on $\cal U$.  Hence $E:= {P(x) \over P(x) - r(x)}$ is a strong unit on $\cal U$, i.e.
\begin{eqnarray*}
P(x) &=& [P(x) - r(x)] E\\
&=& p(x) q(x)  E.\\
\end{eqnarray*}
Since the final conclusion of the Lemma is not used, we leave its proof to the reader.
\end{proof}

\begin{rem}\label{strongthin}  As remarked in the proof of Lemma \ref{thin}, the proof is simpler when the Weierstrass system is strong.  In that case we have that $f = f_1(x,{p\over\epsilon}) + {\epsilon\over p}f_2(x,{\epsilon \over p})$.
Multiplying by s suitable power of ${\epsilon\over p}$, dividing by a strong unit (and using the Strong Noetherian Property), we may assume that $f$ is regular in ${\epsilon\over p}$, of degree $s$, say.  Hence,
by Weierstrass Preparation, multiplying by a strong unit $E$, we are reduced to the case that
$E\cdot p^s \cdot f =  f_1(x,{p\over\epsilon})$.  This case is the same as that of a disc, handled in Lemma \ref{MLdisc}.
\end{rem}

We do not use the following

\begin{cor}\label{inj}  We showed above that when $\cU$ is thin, and when $\cU$ is a $K$--disc, the mapping $f \mapsto f^\sigma$ for $f \in \cO(\cU)$ is injective.  The similar result, with a similar proof, also holds for simple Laurent annuli.
\end{cor}

We include the following Proposition and its Corollary for
completeness - we do not use it except in the proof of Theorem \ref{ML2}.  The slightly weaker version of that theorem in which $f$ is replaced by $f^\sigma$ does not require Proposition \ref{ext}.

\begin{prop} \label{ext}Let $\varphi$ be a $K$-annulus formula, and let $F$ be an algebraic extension of $K$ over which all the polynomials in $\varphi$ split and containing the $a_i$ (as above $a_i^{\ell_i} = \epsilon_i$).
Then
\begin{itemize}
\item[(i)]$\cal O_K^\dag(\varphi)\hookrightarrow\cal O_K^\dag(\varphi)\otimes_K F$ is a faithfully flat extension.
\item[(ii)]$\cal O_K^\dag (\varphi)\otimes_K F\simeq\bigoplus\limits^N_{i=1}\cal O_F^\dag(\cal U_i)$ where the $\cal U_i$ are the linear $K_{alg}$ annuli which make up $\cal U_\varphi$ (see Lemma 5.2(ii)).
\end{itemize}
\end{prop}

\begin{proof}Since $A_{m,n}(F)$ is an integral extension of $A_{m,n}(K)$, (i) is immediate.
Part (ii) follows by induction and rescaling from the following
Lemma.
\end{proof}

\begin{lem}Let $g=\prod^n_{i=1} (x-\alpha_i)\in F[x]$ and assume for all $i,j$ that $|\alpha_i|=1$, $|\alpha_i-\alpha_j|=1$ if $i\neq j$ and that on $\cal U_\phi$, $|g(x)|\square\epsilon$, where either $\square$ is $\leq$ and $\epsilon < 1$ or $\square$ is $<$ and $\epsilon\leq 1$.
Then in $\cal O_F^\dag(\varphi)$ (indeed in $\cal O_F^\dag(\varphi_0)$)
$$
g(\xi_1)-\epsilon z_1=\prod^n_{i=1} (\xi_1-\alpha_i-h_i)
$$
where $h_i\in A^\circ_{2,0}$ or $A^\circ_{1,1}$
(according as $\square$ is $\leq$ or $<$).
\end{lem}

\begin{proof}As above $z_1={g(\xi_1)\over\epsilon}$ is strongly power bounded (resp.~topologically nilpotent).
Consider the case that $\square$ is $\leq$. The other case is
similar. Let $H$ be a new type 1 variable.  By Taylor's theorem we have
\begin{eqnarray*}
(g - \epsilon z_1)(\alpha_1 + \epsilon H) &=& g(\alpha_1) - \epsilon
z_1 + g'(\alpha_1)\epsilon H
 +{g''(\alpha_1)\over {2!}} {(\epsilon H)^2}   + \cdots\\
&=& \epsilon z_1 +g'(\alpha_1)\epsilon H  + {g''(\alpha_1) \over {2!}}{(\epsilon H)^2}  + \cdots\\
&=& \epsilon [z_1 +  g'(\alpha_1)H +  {g''(\alpha_1) \over {2!}} {\epsilon H^2} + \cdots]\\
\end{eqnarray*}
observe that $|g'(\alpha_1)|=1$, hence by Weierstrass Preparation
there is an $A\in\cal O_F(\varphi_0)$ such that
$$
(g-\epsilon z_1) (\alpha_1+\epsilon H)=g'(\alpha_1) [H-A] Q.
$$
Then $\alpha_1+\epsilon A$ is the required zero of $g-\epsilon z_1$.
\end{proof}

\subsection{Linear $K$--annuli}\label{linear} In this subsection we prove some basic results
(in particular Proposition \ref{MLalgcl}) for \emph{linear} $K$--annuli.  In subsequent subsections we will extend several of these results to general $K$--annuli.
The following lemma, which is a special case of Lemma \ref{MLdisc}, is immediate by Weierstrass Preparation.
\begin{lem}\label{disc} Let $\cU$ be a linear $K$--disc and let $f(x) \in \cO(\cU)$.  There is a polynomial $P(x) \in K[x]$ all of whose zeros lie in $\cU$ and a strong unit $E \in \cO(\cU)$ such that
$$
f(x) = P(x) E(x).
$$
\end{lem}

The case of a linear Laurent annulus is a little more complicated.

\begin{lem} \label{linearLaurent} Let $\cU$ be a linear Laurent $K$--annulus described by an annulus formula $\epsilon_1 < |x -
\alpha| < \epsilon_0$ and let $0 \neq f(x) \in \cO(\cU)$.  There is a monic polynomial $P(x)$ all of whose zeros lie in $\cU$,  an $n \in \bN$ and a strong unit $E \in \cO(\cU)$ such that
$$
f = P(x)(x-\alpha)^nE.
$$
\end{lem}
\begin{proof} As in the proof of Lemma \ref{thin}, doing
Weierstrass division by $x-\alpha - \epsilon_0({x - \alpha
\over\epsilon_0})$,  we may assume that
$$
f= \sum_{i,j}{a_{ij}\Big({x - \alpha \over\epsilon_0}\Big)^i\Big({\epsilon_1\over x - \alpha}\Big)^j}
$$
where the $a_{ij}$ are of degree $< 1$, i.e. constants.  Such a
representation is again far from canonical  --  for example we could
have $f= ({\epsilon_1\over \epsilon_0})({x - \alpha
\over\epsilon_0})({\epsilon_1\over x - \alpha}) - ({x - \alpha
\over\epsilon_0})^2({\epsilon_1\over x - \alpha})^2$, which is
actually the zero function.

As in the proof of Lemma \ref{thin} (using the relation $\rho_0\rho_1 - {\epsilon_1 \over \epsilon_0} = 0$ instead of $\xi_0\xi_1 - 1 = 0$) we can write
$$
f= \sum_{i}{b_{i}\Big({x - \alpha \over\epsilon_0} - {\epsilon_1 \over x - \alpha}\Big)^i} + {x - \alpha \over\epsilon_0}\sum_{i}{c_{i}\Big({x - \alpha \over\epsilon_0} - {\epsilon_1 \over x - \alpha}\Big)^i}
$$
where the $b_i, c_i$ are constants.

Using the Strong Noetherian Property, Theorem \ref{SepSNP}
and Corollary \ref{VSSNP}) we see that there are only finitely many biggest  $b_i$ and  $c_i$.  Among the biggest terms let those of lowest degrees in the two sums be $b_{i}({x-\alpha \over\epsilon_0} - {\epsilon_1 \over x-\alpha})^i$ and $c_{j}{x-\alpha \over\epsilon_0}({x-\alpha\over\epsilon_0} - {\epsilon_0 \over x-\alpha})^j$. Of course, one may be missing. If $i > j+1$, after dividing by $b_i$, $f$ is regular in ${x-\alpha \over\epsilon_0}$ of degree $i$.  Similarly, if $i < j+1$, after dividing by $c_j$, $f$ is regular in ${x-\alpha \over\epsilon_0}$ of degree $j+1$.  If $i=j+1$, after dividing by $b_i$, $f$ is regular in ${\epsilon_1 \over x-\alpha}$ of degree $i$.  Hence we may assume that $f$ is regular in either
${x-\alpha \over\epsilon_0}$  or ${\epsilon_1 \over x-\alpha}$.  Assume that $f$ is regular in
${x-\alpha \over\epsilon_0}$.  The other case is similar.  By Weierstrass Preparation (multiplying by a strong unit) we may assume that $f$ is actually a \emph{polynomial} in ${x-\alpha \over\epsilon_0}$, of degree $s$, say.  Hence
$f\cdot ({x-\alpha \over\epsilon_0})^{-s} = g({\epsilon_1 \over x-\alpha})$, and a second use of Weierstrass Preparation completes the proof.
\end{proof}

For linear almost thin annuli we have the following lemma.

\begin{lem} \label{almostthin} Let $\cU$ be an almost thin linear $K$--annulus with annulus formula either
$$
\abs{x-a_0}\le\epsilon' \wedge \bigwedge_{i=0}^n \abs{x-a_i}\ge\epsilon
$$
where
$\epsilon' > \epsilon$, the $\abs{a_i} \geq \epsilon$ and $\abs{a_i - a_j} = \epsilon$ for $i \neq j$, or
$$
\abs{x-a_0}\ge\epsilon \wedge \bigwedge_{i=0}^n \abs{x-a_i}\le\epsilon' \wedge \bigwedge_{i=1}^n \abs{x-a_i}\ge\epsilon'
$$
where $\epsilon' > \epsilon$, the $\abs{a_i} \geq \epsilon'$ and $\abs{a_i - a_j} = \epsilon'$ for $i \neq j$.
Let $f \in \cO^\dag(\cU)$ with $f^\sigma \neq 0$.  There is an $\epsilon''$ with $\epsilon < \epsilon'' < \epsilon'$ such that,
denoting by $\cU''$ the almost thin annulus  defined, respectively, by the annulus formula
$$
\abs{x-a_0}\le\epsilon'' \wedge \bigwedge_{i=0}^N \abs{x-a_i}\ge\epsilon,
$$
or
$$
\abs{x-a_0}\ge\epsilon'' \wedge \bigwedge_{i=0}^n \abs{x-a_i}\le\epsilon' \wedge \bigwedge_{i=1}^n \abs{x-a_i}\ge\epsilon'
$$
there is a monic polynomial $P(x)$ all of whose zeros lie in the corresponding thin annulus $\cU''_t = \cU_t$ (Definition \ref{thin-Laurent}(iii)) integers $n_i$ and a strong unit $E \in \cO^\dag(\cU'')$ such that
$$
f^\sigma|_{\cU''} = P(x) \cdot \prod_{i=0}^n(x - a_i)^{n_i} \cdot E^\sigma.
$$
The polynomial $P$, the integers $n_i$ and the strong unit $E$ are unique.
\end{lem}

\begin{proof} By Lemma \ref{equiv} the two cases are equivalent, so we need only consider the first case. Then the annulus $\cU$ is also defined by the formula
$$
\abs{p(x)} \geq \epsilon^{n+1} \wedge \abs{p(x)} \leq (\epsilon')^{n+1}
$$
where
$$
p(x) := \prod_{i=0}^n(x-a_i).
$$
Hence, by  Propositions \ref{int} and \ref{stronglybounded} or \ref{phipsi} and Weierstrass division, we can write
$$
f(x) = \sum_{i=0}^n\sum_{j,k \geq 0} {a_{ijk}' x^i \Big({p(x)\over (\epsilon')^{n+1}}\Big)^j \Big({\epsilon^{n+1} \over p(x)}\Big)^k}
$$
where the $a_{ijk}' \in \Ko$.  Rescaling, we may assume that $\epsilon' = 1$, so
$$
f(x) = \sum_{i=0}^n\sum_{j,k \geq 0} {a_{ijk} x^i \big(p(x)\big)^j \Big({\epsilon^{n+1} \over p(x)}\Big)^k}
$$
where the $a_{ijk} \in \Ko$.  As in the proof of Lemma \ref{thin}, using the relations
$$
p(x) \cdot {\epsilon^{n+1} \over p(x)} = \epsilon^{n+1}
$$
(i.e. $\xi_1 \cdot \xi_2 = \epsilon^{n+1}$, $\eta = \xi_1 + \xi_2$ and $\xi_1^2 - \eta\xi_1 - \epsilon^{n+1} = 0$) we have
\begin{eqnarray*}
f(x) &=& \sum_{i=0}^n\sum_{j=0}^\infty {a_{ij} x^i \Big(p(x) + {\epsilon^{n+1} \over p(x)}\Big)^j} +p(x) \sum_{i=0}^n\sum_{j=0}^\infty {b_{ij} x^i \Big(p(x) + {\epsilon^{n+1} \over p(x)}\Big)^j} \\
&=& \sum_{i=0}^n x^i \Big[\sum_{j=0}^\infty {a_{ij} \Big(p(x) + {\epsilon^{n+1} \over p(x)}\Big)^j} + p(x) \sum_{j=0}^\infty {b_{ij} \Big(p(x) + {\epsilon^{n+1} \over p(x)}\Big)^j}\Big]\\
&=&  \sum_{i=0}^n{x^i h_i(x)},
\end{eqnarray*}
say.  Without loss of generality we may assume $\|f\| = 1$, i.e. $\max\{\abs{a_{ij}}, \abs{b_{ij}}\} = 1$.

Case (i):  some $h_i$, with $\abs{h_i}=1$ is regular in $p(x)$, and hence also in $x$.  Then $f(x)$ is regular in $x$ and by Weierstrass Preparation we can write
$$
f = U \cdot \Big[x^s + A_1 x^{s-1} + \cdots + A_s \Big]
$$
where each $A_j$ is a power series in ${\epsilon^{n+1} \over p(x)}$ and $U$ is a (strong) unit.
Hence, by Euclidean division by $p(x)$ we are reduced to the case that
$$
f(x) = \sum_{i=0}^n \sum_{j=0}^\infty{a_{ij} x^i   \Big({\epsilon^{n+1} \over p(x)}\Big)^j}.
$$
There are finitely many terms that are biggest in the supremum norm on $\cU_t$, the corresponding
\emph{thin} annulus defined by
$$
\abs{x-a_0}\le\epsilon \wedge \bigwedge_{i=0}^n \abs{x-a_i}\ge\epsilon.
$$
These occur for  $j < L$, say.  Since these terms are all of different degrees $(i-(n+1)j)$, $0 \leq i \leq n$, $0 \leq j < L$, there can be no cancellation in supremum norm, and exactly one of these terms will be biggest on an annulus $\cU^*$ defined by
$\epsilon^{n+1} < \abs{p(x)} \leq (\epsilon'')^{n+1}$ for some $\epsilon < \epsilon'' < 1$.  Hence, shrinking $\epsilon''$ if necessary, we may assume (i) that $f(x) = R(x) + g(x)$ where $R(x)$ is a rational function whose denominator is a power of $p(x)$, and (ii) that $R(x)$ has all its zeros $\beta_i$ in $\cU_t$, and (iii) that outside the open discs of radius $\epsilon$ around the $\beta_i$ we have $\abs{R(x)} > \abs{g(x)}$.  Applying Proposition \ref{zeros} we may further assume that $f$ is a has no zeros in $\cU''$ and hence is a strong unit on every disc contained in $\cU''$.  Then $R(x)$ is a unit and $f \cdot R^{-1} = 1 + R^{-1} \cdot g$ is a strong unit and the proof of existence is complete in this case.  Uniqueness follows exactly as in the proof of Lemma \ref{thin} using Lemma \ref{lemthin}.

Case (ii):  no $h_i$ with $\|h_i\| = 1$ is regular in $p(x)$.  Hence, as in the proof of Lemma \ref{thin}, all are ``regular" in ${\epsilon^{n+1} \over p(x)}$.  Then the ``biggest" terms in the supremum norm on $\cU_t$ are all of the form $a_{ij}x^i\big( {\epsilon ^{n+1} \over p(x)}\big)^j$ with $0 \leq i \leq n$ and $0 \leq j \leq L$ for some $L \in \bN$ and
$\abs{a_{ij}} = 1$.  Indeed,
$$
\|a x^i \Big( {\epsilon ^{n+1} \over p(x)}\Big)^j \big(p(x)\big)^k\|_{sup_{\cU_t}}  =  \abs{a} \abs{\epsilon}^i \abs{\epsilon}^{(n+1)k}
$$
and hence none of the terms with $k > 0$ can be biggest.  We now complete the argument as in case (i).
\end{proof}

\begin{lem}\label{cover2} Let $\cU$ be a (necessarily linear) $K_{alg}$--annulus.  Then either $\cU$ is thin, or there is a finite rigid (Definition \ref{thin-Laurent}) cover of $\cU$ by $K_{alg}$--discs, almost thin $K_{alg}$--annuli and open Laurent $K_{alg}$--annuli $\cU_i$.  Indeed, we can ensure that if each almost thin annulus
$$
\abs{x-a_0}\le\epsilon' \wedge \bigwedge_{i=0}^n \abs{x-a_i}\ge\epsilon
$$
where
$\epsilon' > \epsilon$, the $\abs{a_i} \geq \epsilon$ and $\abs{a_i - a_j} = \epsilon$ for $i \neq j$, or
$$
\abs{x-a_0}\ge\epsilon \wedge \bigwedge_{i=0}^n \abs{x-a_i}\le\epsilon' \wedge \bigwedge_{i=1}^n \abs{x-a_i}\ge\epsilon'
$$
where $\epsilon' > \epsilon$, the $\abs{a_i} \geq \epsilon'$ and $\abs{a_i - a_j} = \epsilon'$ for $i \neq j$ in the cover is replaced by a ``thinner'' annulus
$$
\abs{x-a_0}\le\epsilon'' \wedge \bigwedge_{i=0}^N \abs{x-a_i}\ge\epsilon,
$$
or
$$
\abs{x-a_0}\ge\epsilon'' \wedge \bigwedge_{i=0}^n \abs{x-a_i}\le\epsilon' \wedge \bigwedge_{i=1}^n \abs{x-a_i}\ge\epsilon'
$$
respectively, for any $\epsilon''$ with $\epsilon < \epsilon'' < \epsilon'$, the resulting annuli still form a rigid cover of $\cU$.
\end{lem}
\begin{proof} The proof is an easy induction on the number of holes in $\cU$.
\end{proof}

Next we prove the Mittag-Leffler decomposition for analytic functions on $K_{alg}$--annuli.

\begin{prop}\label{MLalgcl} Let $\cU$ be a $K_{alg}$--annulus and let $\{\cU_i\}$ be the rigid cover of $\cU$ provided by Lemma \ref{cover2}.  Let $f \in \cU$ and assume that for each $i$ we have the unique representation of $f^\sigma|_{\cU_i}$ provided by Lemmas \ref{disc},\ref{linearLaurent} and \ref{almostthin}.  Then, if $f^\sigma \neq 0$, $f^\sigma$ has a unique representation  of the form
$$
f^\sigma  = P(x) \cdot \prod_{i=0}^n(x - a_i)^{n_i} \cdot E^\sigma
$$
where $P$ is a monic polynomial all of whose zeros lie in $\cU$, the $a_i$ are (preselected) centers of the holes in $\cU$, the $n_i \in \bZ$ and $E \in \cO(\cU)$ is a strong unit.
\end{prop}
\begin{proof} We paste together the representations provided by Lemmas \ref{disc}, \ref{linearLaurent} and \ref{almostthin} for $f^\sigma|_{\cU_i}$.  Consider the case that $\cU_1$ is an almost thin annulus with annulus formula
$$
\abs{x-a_0}\ge\epsilon \wedge \bigwedge_{i=0}^n \abs{x-a_i}\le\epsilon' \wedge \bigwedge_{i=1}^n \abs{x-a_i}\ge\epsilon'
$$
and $\cU_2$ is a Laurent annulus with annulus formula
$$
\epsilon_1 < |x - \alpha| < \epsilon'
$$
and that $\epsilon_1 < \epsilon <  \epsilon'$, so that $\cU_1$ and $\cU_2$ overlap.  We may assume by Lemma \ref{lemthin} that $\alpha = a_0$ and that $f^\sigma$ has no zeros in the annulus $\epsilon \leq \abs{x-a_0} < \epsilon'$.  Let
\begin{eqnarray*}
f^\sigma|_{\cU_1} &=& P_1(x) \cdot \prod_{i=0}^n(x - a_i)^{n_i} \cdot E_1^\sigma\\
f^\sigma|_{\cU_2} &=& P_2(x) \cdot(x - a_0)^m \cdot E_2^\sigma.
\end{eqnarray*}
Then all the zeros of $P_2$ lie in the hole $\abs{x-a_0} < \epsilon$ of $\cU_1$ and by the uniqueness of the representations and Lemma \ref{lemthin} we have that $n_0 - m = \deg P_2$ and $P_2(x) \cdot (x - a_0)^{m - n_0}$ is a strong unit on $\cU_1$.  Hence
$$
f^\sigma |_{\cU_1 \cup \cU_2} = P_1(x) \cdot P_2(x) \cdot (x - a_0)^m \cdot \prod_{i=1}^n(x - a_i)^{n_i} \cdot E_1^\sigma.
$$
The other cases are similar.
\end{proof}

\subsection{General $K$--annuli, part 2}\label{gen2}

Now let $\cU$ be a (not necessarily linear) $K$-annulus, defined by
a $K$-annulus formula $\varphi$, let $\cU = \bigcup \cU_j$ be the decomposition of $\cU$ into (linear) $K_{alg}$--annuli, and let $f\in\cal O_K(\varphi)$.
After we have (arbitrarily) chosen a center $\alpha_{ijk}$ for each
hole in $\cal U_j$ corresponding to the inequality $\epsilon_i
\Box_i |p_i(x)|$, say a zero of $p_i$ in that hole, by Proposition \ref{MLalgcl}, the function $f^\sigma$, if it is nonzero, has a (unique) representation on each $\cU_j$
$$
f^\sigma|_{\cal U_j}= P_j(x)\cdot \prod_{i,k}
(x-\alpha_{ijk})^{n_{ijk}}\cdot E^\sigma_j.
$$
Here $P_j\in F[x]$ is monic with zeros only in $\cal U_j$ and $E_j$
is a strong unit in $\cal O_F(\cal U_j)$. $F$ is an algebraic extension of $K$ containing the $\alpha_{ijk}$.

Note that for $\ell \neq j$ the function $P_\ell \cdot \prod_{i,k}
(x-\alpha_{i \ell k})^{n_{i \ell k}}$ is a strong unit on $\cal
U_j$. Take $R:=\prod_j [P_j \cdot \prod_{i,k}
(x-\alpha_{ijk})^{n_{ijk}}]$. Then $f\cdot R^{-1}$ is a strong unit
on each $\cal U_j$ and hence a strong unit on $\cal U_\varphi$. Thus
we have the decomposition
\begin{equation}\label{mlF}
f^\sigma = P(x)\cdot\prod_{i,j,k}(x-\alpha_{ijk})^{n_{ijk}}\cdot E^\sigma,
\end{equation}
where $P \in F[x]$ is monic and has zeros only in $\cal U_\phi$, and
the $\alpha_{ijk}$ are ``centers" of the $F$-holes in $\cal U$.  Considering automorphisms of $K_{alg}$
over $K$ that permute the zeros of $p_i$ we see that $n_{ijk} =
n_{ij'k'}$.

By Proposition \ref{zeros} we see that $P \in K[x]$.

If $K$ is of equicharacteristic zero, we may assume that each $p_i$
occurring in the definition of $\cal U_\varphi$ has only one zero in
each disc or hole. Indeed, taking the $p_i$ of lowest possbile
degree (i.e. taking a good description of $\cU_\phi$) will ensure
that, for if $p$ has $m > 1$ zeros in each of its holes, say
$\alpha_1,\ldots,\alpha_m$ in one of the holes, then $\beta={1\over
m}(\alpha_1+\ldots+\alpha_m)$ is also in that hole and the minimal
polynomial of $\beta$ will have lower degree than $p$. Hence in
equicharacteristic zero we have in fact written
$$
f^\sigma=P\cdot\prod_i p_i^{n_i}\cdot E^\sigma
$$
where $P\in K[x]$ is monic and has all its zeros in $\cal U_\phi$
the $n_i\in\Bbb Z$ and $E$ is a strong unit. The uniqueness of this representation follows from uniqueness in the linear case.  Hence we have proved Theorem \ref{ML} in the equicharacteristic zero case. To obtain the similar result in the general (non-equicharacteristic-zero) case will take a little more work.

\begin{defn}\label{surround} Let $\cal U_\varphi$ be an annulus and let $f \in \cO(\cal U_\varphi)$.
\item[(i)]We say that the thin annulus $\cal U_\psi$ \emph{surrounds} the hole
$\cal H_i = \{x\colon |p_i(x)| < \epsilon_i\}$ in $\cal U_\varphi$ if
$\cal H_i$ is a hole of $\cU_{\psi}$ and $\cal U_\psi \subset \cal U_\varphi$.
\item[(ii)]We say that the thin annulus $\cal U_\psi$ \emph{surrounds} the hole
$\cal H'_i = \{x\colon |p_i(x)| \leq \epsilon_i\}$ in $\cal U_\varphi$ if
$\cal U_\psi = \{x\colon  \abs{p_i(x)} = \delta\} \subset  \cal U_\varphi$,  $\epsilon _i < \delta$
and $f$ has no zeros in the annulus $\{x\colon \epsilon_i < \abs{p_i(x)} \leq \delta \}$ and no other hole of $\cU_\varphi$ is contained in the hole of $\cU_\psi$.
\item[(iii)]We say that the Laurent annulus $\cal U_\psi$ \emph{surrounds} the hole
$\cal H'_i = \{x\colon |p_i(x)| \leq \epsilon_i\}$ in $\cal U_\varphi$ if
$\cal U_\psi = \{x\colon  \epsilon_i < \abs{p_i(x)} < \delta\} \subset  \cal U_\varphi$,  $\epsilon _i < \delta$
and $f$ has no zeros in the annulus $\{x\colon \epsilon_i < \abs{p_i(x)} < \delta \}$

\end{defn}

It is clear if $\cal U_\phi$ is a $K$-annulus that for each of the holes of $\cal U_\phi$ there is a
thin $K$-annulus surrounding that hole, and if the hole is of the form $\cal H'_i = \{x\colon |p_i(x)| \leq \epsilon_i\}$ there is a Laurent annulus surrounding that hole in $\cal U_\varphi$.  We will complete the proof of Theorem \ref{ML} by
comparing the representations we have obtained above for $f$ on $\cal
U_\varphi$ but over $F$ (equation \ref{mlF})  with the representations we have from Lemma
\ref{thin} for $f$ on these thin annuli surrounding the holes of $\cU_\varphi$.
\begin{thm}(Mittag-Leffler Decomposition.)\label{ML}
Let
$$
\varphi :=|p_0
(x)|\square_0\epsilon_0\wedge\bigwedge\limits^n_{i=1}\epsilon_i\square_i
|p_i(x)|
$$
 be a good $K$-annulus formula and let $f\in\cal O_K^\dag(\varphi)$.
Then, if $f^\sigma \neq 0$, there exist a monic polynomial $P(x)$ with zeros only in $\cal
U_\varphi$, integers $n_i$ and a strong unit $E\in\cal O_K^\dag(\varphi)$
such that
\begin{equation}\label{mlK}
f^\sigma=P(x)\cdot\prod^n_{i=1} p_i (x)^{n_i}\cdot E^\sigma.
\end{equation}
$P,E$ and the $n_i$ are uniquely determined by $f$ (and $\varphi$).
\end{thm}

\begin{proof}
We already observed that $P \in K[x]$. We have to show that
$$
\prod_{i,j,k}(x-\alpha_{ijk})^{n_{ijk}}\cdot E^\sigma
$$
 can be written in the form
 $$
 \prod^n_{i=1} p_i (x)^{n_i}\cdot (E')^\sigma.
 $$
Hence, we need to see, in the notation of equation \ref{mlF}, that  $n_{ijk}$
is a multiple of the number of zeros that $p_i$ has in each hole in
each $\cal U_j$.  But this follows by comparing the representation
\ref{mlF} with the representations (\ref{mlK1}) that come (via Lemma
\ref{thin}) from the thin annuli surrounding the holes
as in
Definition \ref{surround}.  Finally, observe that if $p_i$ has $\ell$
zeros $\alpha = \beta_1, \cdots , \beta_{\ell}$ in a hole, then
$(x-\alpha)^\ell = \prod^\ell_{k=1}{(x-\beta_{k})}\cdot E'$ for
some strong unit $E'$.

Uniqueness  of the representation \ref{mlK} follows from the observations that $P(x)$ is determined by the zeros of $f$ and that $\prod^n_{i=1} p_i (x)^{n_i}$ is a strong unit only when  $n_i = 0$ for all $i$.
\end{proof}

The following Theorem is what we need for model-theoretic applications.

\begin{thm}\label{terms}
Let $K$ be a valued field with separated analytic $\cA$-structure,
and let $\cL_{\cA(K)}$ be the language of valued fields, $\langle 0,
1, +, \cdot, (\cdot)^{-1}, |\cdot| \rangle$, augmented with function
symbols for all the elements of $\bigcup_{m,n}A_{m,n}(K)$. (We
extend functions $f \in A_{m,n}(K)$ by zero outside
$(\Ko)^m\times(\Koo)^n$.) Let x be one variable, and let $\tau(x)$
be a term of $\cL_{\cA(K)}$.  There is a finite set $S \subset
\Ko_{alg}$ and a finite cover of $\Ko_{alg}$ by $K$-annuli $\cU_i$
such that for each $i$ there is rational function $R_i \in K(x)$ and
a strong unit  $E_i \in \cO_K^\dag(\cU_i)$ with
$$
\tau|_{\cU_i \setminus S} = R_i \cdot E^\sigma_i|_{\cU_i \setminus S}
$$
i.e. $\tau$ and $R\cdot E_i$ define the same function on $\cU_i
\setminus S$. Observe that $K_{alg}$ also has analytic
$\cA(K)$-structure (Theorem \ref{algextn}), $\tau$ is also a term of
$\cL_{\cA(K_{alg})}$ and hence defines a function $\Ko_{alg} \to
K_{alg}$.
\end{thm}

\begin{proof} This is proved by induction on terms (cf. \cite{CLR1} Theorem 5.1.)  First consider
$$
\tau(x) = f(\tau_1(x),\cdots,\tau_{m+n}(x)),
$$
where $f \in A_{m,n}(K)$. Then, by induction, we may assume that we
have restricted to a $K$-annulus $\cU$, and that there are rational
functions $R_j$ and strong units $E_j \in \cO_K^\dag(\cU)$ such that
for each $j$
$$
\tau_j|_{\cU} = R_j \cdot E^\sigma_j|_{\cU}.
$$
(We ignore the finite set $S$.) By Lemma \ref{cover1} we can cover
$\cU$ with finitely many $K$-annuli $\cU_i'$ such that on each
$\cU_i'$ for each $1 \leq j \leq m$, $|R_j(x)E_j(x)| \leq 1$ for all
$x \in \cU_i'$ or $|R_j(x)E_j(x)| > 1$ for all $x \in \cU_i'$; and
for each $m+1 \leq i \leq m+n$, $|R_j(x)E_j(x)| < 1$ for all $x \in
\cU_i'$ or $|R_j(x)E_j(x)| \geq 1$ for all $x \in \cU_i'$.  Then on
each $\cU_i'$ there is an $f_j' \in \cO_K^\dag(\cU_i')$ that defines
the same function as $\tau$, and the result follows from Theorem
\ref{ML} and induction.

Next consider $\tau(x) = \frac{\tau_1(x)}{\tau_2(x)}$.  As above we
may assume that on $\cU$ we have $\tau_i|_{\cU \setminus S} =
R_i \cdot E_i|_{\cU \setminus S}$ for $i=1,2$.
Then $\tau(x)|_{\cU \setminus S} = R_1\cdot R_2^{-1}\cdot E_1\cdot E_2^{-1}.$

The case $\tau(x) = \tau_1(x)\cdot\tau_2(x)$ is similar.

Finally consider $\tau(x) = \tau_1(x) + \tau_2(x)$ on an annulus $\cU$.  Breaking into sub-annuli it is sufficient to consider the case that $|\tau_1(x)| \geq |\tau_2(x)| \text{ for all }x \in \cU.$
 Write $\tau = \tau_1(1 +  \frac{\tau_2}{\tau_1})$.  Since $|\tau_1(x)| \geq |\tau_2(x)|$ for all $x \in \cU$, it follows from Theorem \ref{ML} that if $\tau_1$ is not identically zero, then $\frac{\tau_2}{\tau_1} \in \cO(\cU)$.  Hence, also
$1+ \frac{\tau_2}{\tau_1} \in \cO(\cU)$, and the result follows from the previous case.
\end{proof}

\begin{rem}\label{moreterms}
With the notation from Remark \ref{constantsbis}, Theorem
\ref{terms} also holds for $\cA(K')$-terms, in which case the rational
functions can be taken over $K'$ and the annuli are $K'$-annuli.
\end{rem}

\begin{rem}\label{minimum}The following statement (A Piecewise Mittag-Leffler Theorem) which is weaker than Theorem \ref{ML} suffices for the application Theorem \ref{terms}:

\emph{Let $\cU$ be a $K$-annulus, $f  \in \cO^\dag_K(\cU)$. There is a
finite cover of $\cU$ by (thin and Laurent) $K$-annuli and open
$K$-discs $\cU_i$ such that for each $i$ there are polynomials $P_i
\in K[x]$, integers $n_{ij} $ and strong units $E_i \in
\cO^\dag_{K}(\cU_i)$ such that
$$
f^\sigma|_{\cU_i}=P_i\prod_{j=1}^{n_i}p_{ij}^{n_{ij}}E^\sigma_i
$$
where the $p_{ij}$ are the polynomials occurring in a good
description of $\cU_i$.}

The proof is easier than that of the full Mittag-Leffler Theorem (Theorem \ref{ML}).
It is easy to see that there is a cover (not necessarily rigid) of $\cU$ by thin and Laurent annuli and discs. The thin annuli are handled by Lemma \ref{thin}.  Linear Laurent annuli can be handled using the canonical representation
$$
f(x) = \sum_i{a_i\Big({x-\alpha \over \epsilon_0} - {\epsilon_1 \over x-\alpha}\Big)^i} + \Big({x-\alpha \over \epsilon_0}\Big) \sum_i{b_i\Big({x-\alpha \over \epsilon_0} - {\epsilon_1 \over x-\alpha}\Big)^i}
$$
where the $a_i, b_i \in \Ko$.  The result for Laurent $K$--annuli then follows as in the proof of Theorem \ref{ML} by comparing representations on thin $K$--annuli surrounding the holes.  Discs are handled in Lemma \ref{MLdisc}.
\end{rem}

\subsection{Strong Weierstrass systems}\label{Strongsystems}

By imposing extra axioms on the Weierstrass system, we can obtain
stronger results than those of the previous section, and also prove some of the results of the previous section more easily and in a more elementary way.  We do not need these results for the model-theoretic applications in Section \ref{seccell}, but we present them here for completeness.  Most of the examples of Weierstrass systems  given in \ref{sepex} satisfy these additional axioms.

\begin{defn}\label{stronggood}

We call a separated Weierstrass system $\{A_{m,n}\}$ a \emph{strong
separated Weierstrass system} if it satisfies the following five
conditions:
\begin{itemize}
  \item[(i)] If $f(\xi,\eta_1,\eta_2,\rho) \in A_{m+2,n}$, there are
  $f_1(\xi,\eta_1,\zeta_1,\zeta_2,\rho), f_2(\xi,\eta_2,\zeta_1,\zeta_2,\rho)$ and
  $Q(\xi,\eta_1,\eta_2,\zeta_1,\zeta_2,\rho) \in A_{m+4,n}$ such that
  \begin{eqnarray*}
  f(\xi,\eta_1,\eta_2,\rho) = f_1(\xi,\eta_1,\zeta_1,\zeta_2,\rho) &+& f_2(\xi,\eta_2,\zeta_1,\zeta_2,\rho) + \\
  & &+ \quad  Q\cdot(\eta_1\eta_2 -\zeta_1\eta_1 - \zeta_2\eta_2).
 \end{eqnarray*}
\item[(ii)] If $f(\xi,\rho,\lambda_1,\lambda_2) \in A_{m,n+2}$, there are
  $f_1(\xi,\rho,\lambda_1,\tau_1,\tau_2)$, $f_2(\xi,\rho,\lambda_2,\tau_1,\tau_2)$ and
  $Q(\xi,\rho,\lambda_1,\lambda_2,\tau_1,\tau_2) \in A_{m,n+4}$ such that
   \begin{eqnarray*}
 f(\xi,\rho,\lambda_1,\lambda_2) = f_1(\xi,\rho,\lambda_1,\tau_1,\tau_2) &+& f_2(\xi,\rho,\lambda_2,\tau_1,\tau_2) + \\
 & &+ \quad Q\cdot(\lambda_1\lambda_2 -\tau_1\lambda_1 - \tau_2\lambda_2).
  \end{eqnarray*}
 \item[(iii)] If $f(\xi,\eta_1,\rho,\lambda_1) \in A_{m+1,n+1}$, there are
  $f_1(\xi,\eta_1,\zeta_1,\rho,\tau_1)$, $f_2(\xi,\zeta_1,\rho,\lambda_1,\tau_1)$ and
  $Q(\xi,\eta_1,\zeta_1,\rho,\lambda_1,\tau_1) \in A_{m+2,n+2}$ such that
   \begin{eqnarray*}
f(\xi,\eta_1,\rho,\lambda_1) =f_1(\xi,\eta_1, \zeta_1,\rho,\tau_1) &+& f_2(\xi,\zeta_1,\rho,\lambda_1, \tau_1)+\\
& & + \quad Q\cdot(\eta_1\lambda_1 - \tau_1\eta_1 - \zeta_1\lambda_1).
  \end{eqnarray*}
 \item[(iv)] If $f(\xi,\eta_1,\eta_2,\rho) \in A_{m+2,n}$, there are
  $f_1(\xi,\eta_1,\eta_3,\rho)$, $f_2(\xi,\eta_2,\eta_3,\rho)$ and \\
  $Q(\xi,\eta_1,\eta_2,\eta_3,\rho) \in A_{m+3,n}$ such that
   \begin{eqnarray*}
  f(\xi,\eta_1,\eta_2,\rho) = f_1(\xi,\eta_1,\eta_3,\rho) &+& \eta_2f_2(\xi,\eta_2,\eta_3,\rho) + \\
  & & + \quad Q\cdot(\eta_1\eta_2 - \eta_3).
 \end{eqnarray*}
\item[(v)] If $f(\xi,\rho,\lambda_1,\lambda_2) \in A_{m,n+2}$, there are
  $f_1(\xi,\rho,\lambda_1,\lambda_3)$, $f_2(\xi,\rho,\lambda_2,\lambda_3)$ and
  $Q(\xi,\rho,\lambda_1,\lambda_2,\lambda_3) \in A_{m,n+3}$ such that
   \begin{eqnarray*}
  f(\xi,\rho,\lambda_1,\lambda_2) = f_1(\xi,\rho,\lambda_1,\lambda_3) &+& \lambda_2f_2(\xi,\rho,\lambda_2,\lambda_3) + \\
 & &+ \quad Q\cdot(\lambda_1\lambda_2 - \lambda_3).
 \end{eqnarray*}
\end{itemize}

We call a strictly convergent Weierstrass system $\{A_{m}\}$ a
\emph{strong strictly convergent Weierstrass system} if it satisfies
properties (i) and (iv) with $n=0$.
\end{defn}

\begin{rem}\label{strong}  The conditions
of Definition \ref{stronggood} ensure that we can perform the formal
operations on the power series in $A_{m,n}$ that are needed in this
subsection. For example, parts (iv) and (v) allow us to write
$f(X,X^{-1}) = f_1(X) + f_2(X^{-1})$.  Other parts will be used
(suppressing extraneous variables) to allow us to use a relation of
the form $XY = aX + bY$, where $a$ and $b$ are constants, to write
$f(X,Y) = f_1(X) + f_2(Y)$.  This is what we need for
partial fractions expansions. These conditions are not automatically
satisfied for Weierstrass systems, see Example \ref{sepex}(9),
though they are satisfied by most of the natural examples (and are a consequence of completeness).
\end{rem}

For the rest of this subsection we assume that the field $K$ has
strong separated analytic $\cA$-structure, and hence strong
separated analytic $\cA(K)$-structure, coming from a strong
separated Weierstrass system $\cA = \{A_{m,n}\}$.

We will first reprove the basic Mittag-Leffler decomposition
for \emph{linear} annuli (Theorem \ref{MLL}.) We give this proof in the ``strong" case as it is much simpler than the proof given above when $\cA$ is not necessarily strong.

Let $\phi$ be a linear $K$-annulus formula, and let $f \in
\cO^\dag_K(\phi)$ . Then
$$
f=\sum_\nu{a_\nu(x)\Big({{x-\alpha_0} \over
\epsilon_0}\Big)^{\nu_0}\Big({\epsilon_1 \over
{x-\alpha_1}}\Big)^{\nu_1}\cdots}\Big({\epsilon_n \over
{x-\alpha_n}}\Big)^{\nu_n}.
$$
(Recall the slight abuse of notation about
$\epsilon$ above Lemma \ref{annulus_properties}.) Doing Weierstrass
Division by $x-\alpha_0 -\epsilon_0({x-\alpha_0 \over \epsilon_0})$
we may assume that the $a_\nu$ are constants (i.e. of degree $0$).
Next we observe that we can write
$$
f= f_0({x-\alpha_0}) + \hat f\Big({\epsilon_1 \over {x-\alpha_1}},
\cdots ,{\epsilon_n \over {x-\alpha_n}}\Big).
$$
To see this, notice that by Weierstrass Division by $x-\alpha_0 -
\epsilon_0[{x-\alpha_1 \over \epsilon_0} + {\alpha_1 - \alpha_0
\over \epsilon_0}]$ we can write
$$
f = g\Big({{x-\alpha_1} \over \epsilon_0},{\epsilon_1 \over
{x-\alpha_1}}, \cdots ,{\epsilon_n \over {x-\alpha_n}}\Big)
$$
and using Definition \ref{stronggood}(iv) (see Remark \ref{strong}), we
can write $g$ as
$$
g_1\Big({{x-\alpha_1} \over \epsilon_0},{\epsilon_2 \over
{x-\alpha_2}}, \cdots ,{\epsilon_n \over {x-\alpha_n}}\Big) +
{\epsilon_1 \over {x-\alpha_1}}g_2\Big({\epsilon_1 \over {x-\alpha_1}}, \cdots ,{\epsilon_n \over
{x-\alpha_n}}\Big),
$$
and we may proceed by induction.  Finally we use the observation
that for $i \neq j$
$$
{\epsilon_i \over {x-\alpha_i}}{\epsilon_j \over {x-\alpha_j}} = {A
\over {x-\epsilon_i}} + {B \over {x-\epsilon_j}}
$$
where $A= {\epsilon_1\epsilon_2 \over {\alpha_1 - \alpha_2}}$ and
$B= {\epsilon_1\epsilon_2 \over {\alpha_2 - \alpha_1}}$.  We have
$|\alpha_1 - \alpha_2| \geq |\epsilon_1|,|\epsilon_2|$, and if at
least one of the holes is ``open" (i.e. defined by a strict
inequality) then $|\alpha_1 - \alpha_2| >
|\epsilon_1|,|\epsilon_2|.$ Now Definition \ref{stronggood} parts
(i)-(iii) (see Remark \ref{strong}) allow us to carry out the
partial fractions expansion term by term to write
$$
f = f_0\big({{x-\alpha_0} \over \epsilon_0}\big) +
f_1\big({\epsilon_1 \over {x-\alpha_1}}\big) + \cdots +
f_n\big({\epsilon_n \over {x-\alpha_n}}\big).
$$
By convention we put the constant term in $f_0$, so $f_1(\infty)=
\cdots = f_n(\infty) = 0.$ The uniqueness of this expansion is then
proved exactly as in \cite{FP} Proposition 2.2.6. This establishes
the first part of the following

\begin{thm}{(Mittag-Leffler Theorem for linear annuli)}\label{MLL}
Let $\phi$ be a linear $K$-annulus formula, and let $f \in
\cO_K^\dag(\phi)$.  Then

(i) there is a unique $f_0 \in \cO_K ^\dag(|x-\alpha_0| \Box_0
\epsilon_0),$ and unique   $f_i \in \cO_K ^\dag(\epsilon_i  \Box_i
|x-\alpha_i|)$ without constant terms, such that
$$
f = f_0 + \cdots + f_n,
$$

(ii) if $f \neq 0$ there is a unique rational function $R$ of the form
$P(x)\prod_i{(x-\alpha_i)^{n_i}},$  where $P(x)$ is a monic polynomial
all of whose zeros are in $\cU_\phi$ and the $n_i \in \bZ$, and a
unique strong unit $E \in \cO_K ^\dag(\phi)$ such that
$$
f = R \cdot E.
$$
\end{thm}

\begin{proof}
(cf.~\cite{DHM}.) We must prove part (ii). By part (i) we may assume that
$$
f=f_0\big({x-\alpha_0\over\epsilon_0}\big)+\sum^n_{i=1} f_i
\big({\epsilon_i\over x-\alpha_i}\big),
$$
or more precisely that
$$
f\equiv f_0(z_0)+\sum^n_{i=1} f_i (z_i)\text{ modulo
}(x-\alpha_0-\epsilon_0 z_0,\ldots,z_n(x-\alpha_n)-\epsilon_n)
$$
where each $z_i$ is either a $\xi$ or a $\rho$ variable, and
$$
f_0(z_0)=\sum^\infty_{j=0} a_{0j} z_0^j,\quad f_i(z_i)=\sum^n_{j=1}
a_{ij} z_i^j.
$$
First we prove existence.

We proceed by induction on $n$, the number of nonzero $f_i,\ i\geq
1$. If $n=0$, applying the Strong Noetherian Property and
multiplying by a constant, we may assume that $f_0$ is regular.
Hence, after multiplying by a strong unit we may assume that $f_0$
is a polynomial. In the case that $n>0$, by using the Strong
Noetherian Property and after multiplying by a constant, we may
assume that $f$ is regular in $z_i$ of degree $N$ say for some $i
\geq 0$. If $i=0$, we may assume after multiplying by a strong unit
that $f_0$ is a polynomial.  Rewrite $f_0$ as a polynomial in
$x-\alpha_1\over\epsilon_1$ and multiply by
$({x-\alpha_1\over\epsilon_1})^{-N}$ to reduce to the case that
$f_0$ is $0$.  Hence we need only consider the case $i>0$. Apply
Weierstrass Preparation to multiply by a strong unit and reduce to
the case that $f$ is a polynomial of degree $N$, say in ${\epsilon_i \over x-\alpha_i}$. Multiplying
by $({x-\alpha_i\over\epsilon_i})^N$ (or $z_i^{-N}$) and using part
(i) reduces $n$. Hence we have written $f$ in the form
$$
f=P(x)\cdot\prod_{i=1}^n (x-\alpha_i)^{n_i}\cdot E
$$
where $E$ is a strong unit.

Next we show that if $P(x)$ is
irreducible and has no zeros in the disc $\{x \colon
|x-\alpha_0|\square_0\epsilon_0\}$ then $P(x)$ is a strong unit. We
may make a change of variable so that $\alpha_0=0$ and
$\epsilon_0=1$. We may assume that $P(x) = a_mx^m+a_{m-1}x^{m-1}+,
\cdots, a_0$. If $\square_0$ is $\leq$ then all the zeros of $P(x)$
are $>1$ and hence $|a_0| > |a_i|$ for $i=1,\cdots,m$. Thus
$a_0^{-1}P(x)$ is a very strong unit (Definition \ref{strongunit}), as it is regular of degree
$0$. If $\square_0$ is $<$ then all the zeros of $P(x)$ are $\geq 1$
and, since $|x| < 1$ (i.e. is a $\rho$ variable) we again see that
$a_0^{-1}P(x)$ is a very strong unit.
Next we observe (cf. Proposition \ref{linearcanon}) that if $P$ is irreducible and has one of (and hence
all of) its zeros in the hole $\{x:|x-\alpha_i| \overline \square_i \epsilon_i\}$, and $P$ has degree $\ell$,
then  $P\cdot(x-\alpha_i)^{-\ell}$ is a strong unit.
Finally, to prove uniqueness
we must see that the only time that an expression of the form
$P(x)\cdot\prod_{i=1}^n (x-\alpha_i)^{n_i}$ with all zeros of $P$
lying in $\cal U$, is a strong unit is when $P(x)$ is a constant and
$n_i=0$ for all $i$. We may assume that $\alpha_0=0$,
$\epsilon_0=1$. Again considering the two cases $\square_0$ is
$\leq$ and $\square_0$ is $<$ separately, this is clear.

\end{proof}

The following is an immediate consequence of the Mittag-Leffler Theorem for linear annuli, Theorem \ref{MLL}.

\begin{cor}\label{MLLCor} If \ $\cal U$ is a linear $K$-annulus then $\cal O^\dag_K(\cal U)$ is a principal ideal domain and
$\cO_K^\dag(\cU) \simeq \cO_K^\sigma(\cU)$.
\end{cor}
\begin{proof} Using the notation of Theorem \ref{MLL}, exactly as in \cite{FP} Proposition 2.2.6, one has that
$\|f\|_{sup} = \max_i \|f_i\|_{sup}$, and $\|f_i\| = \|f_i\|_{sup}$.
\end{proof}

The full Mittag-Leffler Theorem in the strong case is:

\begin{thm} \label{ML2}If $\varphi$ is a good $K$-annulus formula and $f\in\cal O_K^\dag(\varphi)$ Then:
\item[(i)]there are unique
\begin{eqnarray}
f_0&=\sum^\infty_{j=0} a_{0j} (x) ({p_0(x)\over\epsilon_0})^j\quad\text{with degree }a_{ij}< \text{ degree } p_0\\
f_i&=\sum^\infty_{j=1} a_{ij}(x) ({\epsilon_i\over
p_i(x)})^j\quad\text{with degree }a_{ij} < \text{ degree } p_i
\end{eqnarray}
such that
$$
f=f_0 + f_1 + \ldots + f_n.
$$
(This type of decomposition can fail to exist if the Weierstrass
system is not strong, for example, in the structure described in
\ref{sepex} (9) above).
\item[(ii)] there is a unique representation of the form
$$
f(x) = P(x)\cdot \prod_i{p_i(x)^{n_i}} \cdot E,
$$
where $P(x)$ is a monic polynomial
all of whose zeros are in $\cU_\phi$, the $n_i \in \bZ$, and $E \in \cO_K ^\dag(\phi)$ is a strong unit.
\end{thm}

\begin{proof} As observed in Corollary \ref{MLLCor}, from the representation in (i) it follows easily that the mapping $f \mapsto f^\sigma$ is injective.  Then (ii) follows from the linear case (Theorem \ref{MLL}) exactly as in the proof Theorem \ref{ML} by considering thin or Laurent annuli surrounding the holes.  Recall (Remark \ref{strongthin}) that the proof of Lemma \ref{thin} is much easier in the strong case.

We outline the proof of (i).  If $\cH_i$ is a hole of the form $|p_i(x)| \leq \epsilon_i$ in $\cU_K(\varphi)$, then
considering $f$ on a Laurent annulus $\cU_i = \{x \in K_{alg}: \epsilon_i < |p_i(x)| < \delta_i\}$ surrounding
this hole, gives
$$
f|_{\cU_i} = \sum{b_{jk}(x) \Big({\epsilon_i\over p_i(x)} \Big)^j \Big({p_i(x) \over \delta_i} \Big)^k}
$$
with the $b_{jk}(x)$ of degrees $<$ the degree of $p_i(x)$, which, using Definition \ref{stronggood} and
the relation ${\epsilon_i \over p_i}{p_i \over \delta_i} = {\epsilon_i \over \delta_i}$, can be rewritten as
$$
f|_{\cU_i} = \sum_{j \geq 1}{a_{ij}(x) \Big({\epsilon_i\over p_i(x)} \Big)^j} + \sum_{j \geq 0}{a'_{ij}(x) \Big({p_i(x) \over \delta_i} \Big)^j}.
$$
Take $f_i :=  \sum{a_{ij}(x)({\epsilon_i\over p_i(x)})^j}$.

If $\cH_i$ is a hole of the form $|p_i(x)| < \epsilon_i$ in
$\cU_K(\varphi)$, consider a thin $K$-annulus $\cU_i$ surrounding
$\cH_i$.  On $\cU_i$ $f$ has a representation
$$
f|_{\cU_i} = \sum{a_{ijk\nu}(x) \Big({p_i(x) \over \epsilon_i} \Big)^j \Big({\epsilon_i\over p_i(x)} \Big)^k \Big({\epsilon'\over p'(x)} \Big)^{\nu}}.
$$
Use a relation of the form
$$
{\epsilon_i\over p_i(x)}{\epsilon_j\over p_j(x)} = A(x) {\epsilon_i\over p_i(x)} + B(x) {\epsilon_j\over p_j(x)}
$$
where the degree of $B$ is less than the degree of $p_j$, the degree of $A$ is less than the degree of $p_i$
and $A$ and $B$ have supremum norm $\leq 1$ on $\cU_i$, to separate the terms ${\epsilon_i\over p_i(x)}$ from ${\epsilon_j\over p_j(x)}$.  Use the relation $({\epsilon_i\over p_i(x)})({p_i(x) \over \epsilon_i}) = 1$ to separate the terms ${\epsilon_i\over p_i(x)}$ and ${p_i(x) \over \epsilon_i}$.  Then we have written
$$
f = f_i + g
$$
where $f_i = \sum{a_{ij}(x)({\epsilon_i\over p_i(x)})^j}$ with the degrees of the $a_{ij}$ less than the degree of $p_i$, and $g = \sum{a_{k\nu'}(x)\Big({p_i(x)\over \epsilon_i} \Big)^k \Big({\epsilon'\over p'(x)} \Big)^{\nu'}}$, where $\nu' = (\nu_1, \dots \nu_{i-1}, \nu_{i+1}, \dots, \nu_n)$.
Use a thin or Laurent annulus at the ``outer" edge of $\cU$ to obtain $f_0$.

Finally, considering $f- (f_0 + f_1 + \cdots + f_n)$ on the linear
$K_{alg}$--annuli into which $\cU$ decomposes over $K_{alg}$,
using part (i) of Theorem \ref{MLL}, we see that  $\|f - (f_0 + f_1 +
\cdots + f_n)\|_{sup} = 0$.  The result now follows from Proposition \ref{ext}(i).
\end{proof}

\begin{cor}If \ $\cal U$ is a $K$-annulus then $\cal O^\dag_K(\cal U)$ is a principal ideal domain and
$\cO_K^\dag(\cU) \simeq \cO_K^\sigma(\cU)$.
\end{cor}

\section{Cell decomposition}\label{seccell}

In this section, a cell decomposition, as a consequence of
$b$-minimality, is obtained for analytic structures on Henselian
valued fields of section \ref{sepanstruct}, see Remark \ref{remmt}.
The history of cell decomposition for Henselian fields goes back to
Cohen \cite{Co} and Denef \cite{D1}, \cite{D2} for the algebraic
case, generalized by Pas \cite{Pas1}, \cite{Pas2} to more general
Henselian fields, and by the authors to analytic expansions
\cite{C}, \cite{CLR1}, suited for $p$-adic integration. The
definitions of cells and cell decomposition have been simplified in
\cite{CL} and put in an axiomatic framework in \cite{CLb}. In
\cite{CL}, cell decomposition is used to define motivic integrals;
historically, cell decomposition has been used to calculate many
types of $p$-adic integrals \cite{D1}, \cite{Pas1}. In this section
we generalize the cell decomposition of \cite{CLR1} to the
generality of this paper, using the axiomatic formulation of
\cite{CLb}. We also establish ``preservation of balls'' and the
``Jacobian property'' for these structures, notions useful for
change of variables formulas for integrals. In this section we shall
use the additively written order $\ord$ as well as the
multiplicative norm $|\cdot|$.

\subsection{The semialgebraic language}

Let $\Hens$ be the collection of all Henselian valued fields of
characteristic zero (hence mixed characteristic and
equicharacteristic zero fields are included).

For $K$ in $\Hens$, write $\Ko$ for the valuation ring, $\Gamma_K$
for the value group, $\ord:K^\times\to \Gamma_K$ for the (additively
written) valuation, $M_K$ or $\Koo$ for the maximal ideal of $\Ko$,
and $\Kt$ for the residue field.

For any integer $n>0$, write $$rv_{n}:K^\times\to K^\times/1+nM_K$$
for the natural group morphism, with $nM_K=\{nm\mid m\in M_K\}$, and
extend it to a map $rv_n:K\to (K^\times/1+nM_K)\cup \{0\}$ by
sending zero to zero. Write $RV_{n}$, or $RV_n(K)$, for
$(K^\times/1+nM_K)\cup \{0\}$ for integers $n>0$. Write also $\ord$
for the natural maps $\ord:K^\times/1+nM_K\to \Gamma_K$. We
sometimes abbreviate $rv_1$ to $rv$ and $RV_1$ to $RV$. Note that in
equicharacteristic zero the $RV_n$ all are the same as $RV_1$.

The sorts $RV$ and $RV_{n}$ are called auxiliary. The valued field
sort is the main sort. There are no other sorts. We write $\Val$ for
the valued field sort.

\begin{defn}
Let $\LHens$ be the language of rings $(+,-,\cdot,0,1)$ for the
valued field sort, together with function symbols $rv_{n}$ for all
$n>0$, and the inclusion language as defined below on the auxiliary
sorts.
\end{defn}

 Let $\THens$ be the theory of all fields in $\Hens$ in the language
$\LHens$.

We will sometimes write $K$ to denote both a model of $\THens$ and
the (universe of the) valued field sort of that model.

\subsubsection{The inclusion language on the $RV_n$}\label{induced}
Let $K$ be in $\Hens$. For $a_i\in RV_{n_i}(K)$, $i=1,\ldots,n$, and
for $f,g$ polynomials over $\ZZ$ in $n$ variables, we let the
expression
$$
f(a_1,\ldots,a_{n})
$$
correspond to  the set
$$
f(a):=\{x\in K\mid (\exists y\in K^{n}) \big( f(y)=x \wedgem_i
rv_{n_i}(y_i)=a_i\big)\}.
 $$
By an inclusion
\begin{equation}\label{inc}
f(a)\subset g(a)
\end{equation}
 of such expressions, we shall mean the inclusion
of the corresponding sets.

The inclusion language $\cL_{RV}$ on the sorts $RV_{n}$, $n>0$,
consists of the three symbols $+,\cdot,\subset$, interpreted as the
relations explained in (\ref{inc}). (There are no terms in this
language, only relations of the form $f(x)\subset g(x)$, for $f$ and
$g$ polynomials in $x_1,\ldots,x_n$ formed with $+$ and $\cdot$,
with the $x_i$ variables of the sorts $RV_{n_i}$ for some integers
$n_i>0$.

\begin{rem}
\item[(i)] For an alternative (but essentially equivalent) language on the auxiliary sorts, see
\cite{BK} and \cite{Scanlon}.

\item[(ii)] Note that $\ord(x) < \ord(y)$ is valued field quantifier free
definable using $rv$, for example, by $rv(x) = rv(x+y)$.
\end{rem}

\begin{defn}
 By an \emph{open, resp.~closed, ball} in a valued field $K$ is meant a set
of the form
$$
\{x\in K \mid \ord (x-b) >\ord (a)\}, \mbox{ resp. }\{x\in K \mid
\ord (x-b) \geq \ord (a)\},
$$
with $b\in K$, $a\in K^\times$. By a \emph{point} is meant a
singleton.
\end{defn}

\begin{rem}\label{ballsrv}
For any $n>0$, any nonzero $\xi\in RV_n$, and any $h\in K$, the set
\begin{equation}\label{dball}
X:=\{x\in K\mid rv_n(x-h)=\xi\}
\end{equation}
is an open ball of the form
$$
\{x\in K \mid \ord (x-b) > \alpha\}
$$
for any $b\in X$ and $\alpha= \ord(n(b-h))$. Often, none of the
points $b$ is definable (over a certain set of parameters) while $h$
and $\xi$ are definable. This is the advantage of the description
(\ref{dball}) of the open ball $X$.
\end{rem}

The following is a consequence of Hensel's Lemma, see also
\cite{Co}, \cite{D2}, \cite{Pas1}, \cite{Pas2}.
\begin{lem}\label{hm} Let $K$ be in
$\Hens$. Let
 $$f(y)=\sum_{i=0}^m a_iy^i
 $$
be a polynomial in $y$ with coefficients in $K$, let $n>0$ an
integer, and let $x_0\not=0$ be in $RV_{n}(K)$.
 Suppose that there exist $i_0>0$ and $x\in K$ with
\begin{equation}\label{h0}
rv_n(x)=x_0 \mbox{,  and, $\ \ord (a_{i_0}x^{i_0})$ is minimal among
the $\ord (a_{i}x^{i})$,}
\end{equation}
in the sense that
$$
\min_{0\leq i\leq m} \ord (a_{i}x^{i}) = \ord (a_{i_0}x^{i_0}),
$$
 and such that
\begin{equation}\label{h1}
\ord (f(x)) > \ord (n^2  a_{i_0}x^{i_0})
\end{equation}
and
\begin{equation}\label{h2}
\ord (f'(x)) \leq  \ord (n a_{i_0}x^{i_0-1}).
\end{equation}
Then there exists a unique $b\in K$ with
\begin{equation}\label{h3}
f(b)=0 \ \mbox{ and }\  rv_{n}(b)=x_0.
\end{equation}
\end{lem}
\begin{proof}
The case that the $a_i$ lie in the valuation ring $K^\circ$ and that
$a_{i_0}$ and $x$ are units in $K^\circ$ follows from Hensel's
lemma.
 The general case follows after changing coordinates. See
 \cite{Pas1}, Lemma 3.5,
 or \cite{Pas2} for explicit change of variables.
\end{proof}

\begin{defn}[Henselian functions]\label{hm2}
Let $K$ be in $\Hens$. For each integers $m\geq 0$, $n>0$, define
the function
 $$h_{m,n}:K^{m+1}\times RV_{n}(K) \to K$$ as the function sending
 the tuple
$(a_0,\ldots,a_m,x_0)$ with nonzero $x_0$ to $b$ if there exist
$i_0$ and $x$ that satisfy the conditions (\ref{h0}), (\ref{h1}),
and (\ref{h2}) of Lemma \ref{hm} and where $b$ is the unique element
satisfying (\ref{h3}), and sending $(a_0,\ldots,a_m,x_0)$ to $0$ in
all other cases.

Define $\LHens^*$ as the union of the language $\LHens$ together
with all the function symbols $h_{m,n}$. The functions $h_{m,n}$ are
similar to those in \cite{CLR1}.

\end{defn}

\subsection{The analytic languages}

Let $\cA = \{A_{m,n}\}$ be a separated Weierstrass system, as
defined in section \ref{sepanstruct}. Define $\cL_{\Hens,\, \cA}$ as
the language $\LHens$ together with function symbols for all the
elements of $\bigcup_{m,n}A_{m,n}$, with the field inverse
$(\cdot)^{-1}$ on the valued field sort extended by $0^{-1}=0$, and
together with the induced language on the sorts $RV_n$. The analytic
function symbols are interpreted as zero outside their natural
domains of products of the valuation ring and the maximal ideal. On
their natural domains, they are interpreted via an analytic
$\cA$-structure.

Let $\cT_{\Hens,\, \cA}$ be the $\cL_{\Hens,\, \cA}$-theory of all
Henselian valued fields in $\Hens$ with analytic $\cA$-structure.

Likewise, one can give  definitions of analytic languages and
analytic theories arising from strictly convergent Weierstrass
systems and strictly convergent analytic structures with
$\pi\not=1$, as defined in section \ref{strictlyconv} and Definition
\ref{seps} (i). We will not treat this separately.

\subsection{$b$-minimality}\label{secbmin}

To obtain cell decomposition, the criterion of $b$-minimality of
\cite{CLb} is used. From $b$-minimality many properties, like cell
decomposition, dimension theory, etc., follow immediately, see
\cite{CLb}. In this paper we need only prove $b$-minimality, because
cell decomposition follows by \cite{CLb}, see Remark \ref{remmt}.

First we recall the definitions for $b$-minimality.

By an expansion of a theory $\cT$ in a language $\cL$ is meant a
theory $\cT'$ in a language expanding $\cL$ such that $\cT'$
contains $\cT$.  By \emph{definable} is meant definable with
parameters, unless we specify the parameters, for example by saying
$A$-definable. A definable set is called \emph{auxiliary} if it is a
subset of a finite Cartesian product of (the universes of) auxiliary
sorts.

\begin{defn}[$b$-minimality for expansions of $\THens$]\label{defbmin} Call an expansion of $\THens$ $b$-minimal when
the following three conditions are satisfied for every model $K$,
any set of parameters $A$ (the elements of $A$ can belong to any of
the sorts), for all $A$-definable subsets $X$ and $Y$ of (the valued
field) $K$, and for every $A$-definable function $F:X\to Y$,
\begin{enumerate}
 \item[(b1)]
 there exists an $A$-definable function
$ f:X\to S$ with $S$ auxiliary such that for each $s\in f(X)$ the
fiber $f^{-1}(s)$ is a point or an open ball;
 \item[(b2)] if $g$ is a definable function from an auxiliary set
 to a ball, then $g$ is not surjective;
  \item[(b3)] there exists an $A$-definable function $f : X
\rightarrow S$ with $S$ auxiliary such that for every $s$ in $S$ the
restriction
 $F_{|f^{-1}(s)}$ is either injective or constant.
\end{enumerate}
Call $f$ as in (b1) a \textit{$b$-map} on $X$. (The $f$ and $S$ of
(b1) and (b3) are allowed to be different.)
\end{defn}

\begin{defn}[Centers for expansions of $\THens$]\label{defcenters}
Say that a $b$-minimal  expansion of $\THens$ has centers when, in
(b1) of Definition \ref{defbmin}, one can choose $f$ such that
moreover there exist $n>0$ and an $A$-definable function
$$
h:f(X)\to K
$$
such that for each $s\in f(X)$ with $f^{-1}(s)$ a ball, there exists
a (necessarily unique and nonzero) $\xi$ such that
$$
f^{-1}(s)=\{x\in K\mid rv_n(x-h(s))=\xi\}.
$$

Following \cite{CLb}, call $h$ a \textit{$B_n$-center} of $f$.
\end{defn}
\begin{rem}
The definition of centers should be compared with Remark
\ref{ballsrv}.
\end{rem}

\begin{defn}[Preservation of all balls for expansions of $\THens$]\label{defball}
Say that a $b$-minimal expansion of $\THens$ \textit{preserves all
balls} when for every $F$ as in Definition \ref{defbmin}, $f:X\to S$
can be taken as in (b3) such that moreover for each $s\in S$ and for
each open ball $Z$ with
$$Z\subset f^{-1}(s),$$
$F(Z)$ is either an open ball or a point.
\end{defn}

A function on an open subset of a valued field is called $C^1$ when
it is continuously differentiable in the sense of Cauchy's
$\varepsilon$, $\delta$-definition.

\begin{defn}\label{defjacprop}
Let $F:B_1\to B_2$ be a bijection between open balls $B_1,B_2\subset
K$, with $K$ a Henselian valued field. Say that $F$ \textit{has the
Jacobian property} if the following conditions a) to c) hold

\item[a)] $F$ is $C^1$ on $B_1$; write ${\rm Jac} F$ for $\partial
F/\partial x:B_1\to K$;

\item[b)] $rv ({\rm Jac} F)$ is constant on $B_1$; write $rv ({\rm Jac  } F)=a$;

\item[c)] for all $x,y\in B_1$ with $x\not=y$, one has
$$
\ord(a)+\ord(x-y)=\ord(F(x)-F(y)).
$$
\end{defn}

The Jacobian property of the following definition is useful for
change of variables formulas, for example for $p$-adic or motivic
integrals.

\begin{defn}\label{defjacprop2}
Say that an expansion $\cT$ of $\THens$  \emph{has the Jacobian
property} if and only if the following holds in any model $K$ of
$\cT$ and any set $A$ of parameters:

For any $A$-definable bijection
$$
F:X\subset K\to Y\subset K
$$
there exists a $A$-definable $b$-map
$$
f:X\to S
$$
such that for each $s\in f(X)$ such that $f^{-1}(s)$ is a ball, the
restriction of $F$ to $f^{-1}(s)$ has the Jacobian property.
\end{defn}

The main results of this section are the following two theorems:
\begin{thm}[$b$-minimality, ${\rm Char}(K)=0$]\label{mt}
Let $\cA= \{A_{m,n}\} $ be a separated Weierstrass system, as
defined in section \ref{sepanstruct}. The theory $\cT_{\Hens,\,
\cA}$ eliminates valued field quantifiers, is $b$-minimal with
centers and preserves all balls. Moreover, $\cT_{\Hens,\, \cA}$ has
the Jacobian property.
\end{thm}

Define $\cL_{\Hens,\, \cA}^*$ as the language $\cL_{\Hens,\, \cA}$
together with all the functions $h_{m,n}$.

\begin{thm}[Term structure, ${\rm Char}(K)=0$]\label{thens}
Let $K$ be a $\cT_{\Hens,\, \cA}$-model. Let $X \subset K^n$ be
definable and let $f:X\to K$ be a $\cL_{\Hens,\, \cA}(A)$-definable
function for some set of parameters $A$. Then there exists a
$\cL_{\Hens,\, \cA}(A)$-definable function $g:X\to S$ with $S$
auxiliary such that
\begin{equation}\label{et}
f(x)=t(x,g(x))
\end{equation}
for each $x\in X$ and where $t$ is a $\cL_{\Hens,\, \cA}^*(A)$-term.
\end{thm}

The rest of section \ref{secbmin} is devoted to the proofs of these
theorems. We begin with five lemma's, for $K$ in $\Hens$. We will
abuse notation and write $f$ for $f^\sigma$ when $f$ is in some
$\cO^\dag_K(\cdot)$.
\begin{lem}[${\rm Char}(K)=0$]\label{ballWD}
Let $B$ be the open ball $\Kalgoo$. Note that $B$ is a $K$-annulus.
Let $g$ be in $\cO^\dag_K(B)$. Suppose that $|g'(x)| = 1$ for all
$x\in \Koo$ and that $g(0)=0$. Then there exists an integer $n_0$
such that the mapping $x \mapsto g(x)$ is a bi-analytic isometry
between $n_0\Koo$ and itself.
\end{lem}
\begin{proof}

Let $c\in K$ be such that $\|g'\|=|1/c|$. Since $cg'$ lives in
$A_{0,1}(K)$, it is regular of some degree, hence, by Weierstrass
Preparation for $A_{0,1}(K)$, $cg'$ equals a monic polynomial $p$
over $\Ko$ times a very strong unit $E$ in $A_{0,1}(K)$, for well
chosen $c$. Suppose that $c$ is infinitesimal (in the sense that
$|c|<|n|$ for all integers $n>0$). Let $K'$ be the fraction field of
the quotient of $\Ko$ by the ideal $J$ consisting of the
infinitesimals in $\Ko$. By example (11) in section \ref{sepex},
$K'$ has analytic $\cA(K)/J$--structure. Let $p_J$, resp.~$E_J$ be
the image of $p$, resp.~$E$, in $A_{0,1}/JA_{0,1}$. Then $E_J$ is a
very strong unit, hence has no zeros, and $p_J$ is a monic
polynomial. Thus $p_JE_J$ is nonzero in $A_{0,1}/JA_{0,1}$ but it
has infinitely many zeros in $\Kalg'$, hence $p_J$ has infinitely
many zeros in $\Kalg'$, which is a contradiction. Hence, $c$ is not
infinitesimal and there exists an integer $n_0>0$ such that $|c|\geq
|n_0|$.
 Let $q$ be the image of $p$ in $\cO_K(n_0\Kalgoo)$. Then $q$ is a
strong unit of constant size $|q(x)|\geq |n_0|$ for $x\in
n_0\Kalgoo$, by definition of $\cO_K(n_0\Kalgoo)$ and by Remark
\ref{unit}.
 Hence, if $|n_0|$ is small enough (we can always make $n_0$ more divisible if needed), then the image
$h$ of $g$ in $\cO_K(n_0\Kalgoo)$ is regular of degree $1$. (Indeed,
automatically $h/\|h\|$ is regular of some degree $\alpha$, and then
one can replace $n_0$ by $\alpha \cdot n_0$.) (Regularity of $h$ of
degree $1$ means that $h(n_0\rho)$ is regular of degree $1$ in
$A_{0,1}(K)$.) It then follows for such $n_0$ that $h:n_0\Kalgoo\to
n_0\Kalgoo$ and also $h:n_0\Koo\to n_0 \Koo$ are isometries.
\end{proof}

\begin{rem}\label{rv-Jac}
Suppose that there is $n>0$ such that $rv_n(g')=1$ on $B$, then we
can choose $n_0$ in Lemma \ref{ballWD} such that moreover
$$
rv_n(x-y)=rv_n(g(x)-g(y))
$$
for all $x\not = y$ in $n_0\Koo$. This is stronger than saying that
$x \mapsto g(x)$ is an isometry between $n_0\Koo$ and itself. This
property also follows from the regularity of degree $1$ of $h$ at
the end of the proof of Lemma \ref{ballWD}.
\end{rem}

\begin{lem}[${\rm Char}(K)=0$]\label{ballterm}
Fix elements $b_i$ in $K$ for $i=0,\ldots,m$, fix an integer $n>0$,
and $x$ in $K$. Let $B$ be the Cartesian product
$$
\prod_{i=0}^m (b_i+n^2b_i\Koo),
$$
where $b_i+n^2b_i\Koo=\{0\}$ whenever $b_i=0$. Then the map
$$
F:B\to K:a\mapsto h_{m,n}(a,rv_n(x))
$$
is given by some $f$ in $\cO^\dag(B)$.\\

\end{lem}
\begin{proof}
If $F$ is constantly zero on $B$ there is nothing to prove. In the
other case, $F$ is nowhere zero on $B$, by the definition of
$h_{m,n}$. Suppose thus that $F$ is nowhere zero on $B$.
 First suppose that the $b_i$ lie in $\Ko$, that $b_{i_0}$ is a unit
and that $\ord x=0$, where $i_0$ in $\{1,\ldots,m\}$ is such that
$\ord b_{i_0}$ is minimal among the $\ord b_{i}$ in the sense that
$$
\min_{0\leq i\leq m} \ord b_{i} = \ord b_{i_0}.
$$

 Define
\begin{eqnarray*}
g(\lambda_1,\rho_0,\ldots,\rho_m) & := & \sum_{0=1}^m
(n^2b_i\rho_i+b_i) (n\lambda_1-x)^i\\
\end{eqnarray*}
in $A_{0,m+2}(K)$, where $\lambda_1$ runs over $\Koo$ and $\rho$
over $(\Koo)^{m+1}$. Since $b_{i_0}$ and $x$ are units, since $F$ is
nonzero, by conditions (\ref{h1}) and (\ref{h2}), and by the
definition of $h_{m,n}$, one has for all $(\lambda_1,\rho)\in
(\Kalgoo)^{m+2}$ that
$$
\ord (g(\lambda_1,\rho))>\ord ( n^2)
$$
and that
$$
\ord (g'(\lambda_1,\rho))\leq \ord ( n),
$$ with $g'$ the derivative of
$g$ w.r.t.~$\lambda_1$.
 Rewrite
$$
 g(\lambda_1,\rho)=\sum_{i=0}^m h_i(\rho)\lambda_1^i,
$$
Then, by Taylor's Theorem, $|h_1(\rho)|\geq |n|$ for all $\rho\in
\Kalgoo$, thus $h_1(\rho)^{-1}$ comes from an element of
$A_{0,m+1}^\dag(K)$ and $h_1(\rho)^{-1}\cdot g$ is regular of degree
$1$ in $\lambda_1$ in $A_{0,m+2}(K)$. Hence, by Weierstrass
preparation for $A_{0,m+2}(K)$ one has that
$$g= E(\lambda_1,\rho)(\lambda_1 + h(\rho)   ),
 $$
with $h(\rho)$ in $A_{0,m+1}(K)$, $E$ a unit in $A_{0,m+2}(K)$, and
with $\lambda_1 + h(\rho)$ regular of degree $1$ in $\lambda_1$. So,
by Weierstrass Division one can compose $g$ and $-h$. Thus, in this
case, one can take $-h(\rho)$ for $f$.
 In the general case, one can perform a change of variables as in the
proof of Lemma \ref{hm} to reduce to the above special case, by
using that the rings $\cO^\dag(\cdot)$ are closed under meaningful
composition.

\end{proof}

\begin{lem}[${\rm Char}(K)=0$]\label{rvunits}
Let $n>0$ be an integer. Let $B$ be the open ball $\Kalgoo$. Note
that $B$ is a $K$-annulus. Let $E$ in $\cO^\dag_K(B)$ be a strong
unit, cf.~Definition \ref{strongunit}. Then one has that
$rv_n(E^\sigma)(x)$ only depends on $rv_n(x)$ when  $x$ varies over
$\Kalgoo$.
\end{lem}
\begin{proof}
The unit $E$ in $\cO^\dag_K(B)=A_{0,1}(K)\otimes_{\Ko} K$ is
automatically of the form
$$
E(\rho_1)=c+\rho_1g(\rho_1)
$$
for some $c\in K$ and $g\in \cO^\dag_K(B)$ such that $|c|\geq
\|g\|$, cf.~Remark \ref{unit}.
 Now suppose that $x_1,x_2\in\Kalgoo$ with $rv_n(x_1)=rv_n(x_2)$. Then  $x_1=(1+na)x_2$ for some $a\in\Kalgoo$. It is
enough to treat the case $c=1$ (after dividing by $c$). Then
\begin{eqnarray*}
E(x_1) & = & 1 + (1+na)x_2 g((1+na)x_2)
\\
& =& 1 + ( 1+na )x_2 g(x_2)+ nah_1
\\
 & = &1+ x_2g(x_2) + nah_2
\\
 &=&(1+ x_2g(x_2))(1+ nah_2(1+ x_2g(x_2) )^{-1})
\\
 &=&E(x_2)(1+ nah_2(1+ x_2g(x_2) )^{-1}),
\end{eqnarray*}
with $h_1$ and $h_2$ in $\Ko$, by Taylor expansion of $g((1+na)x_2)$
around $x_2$. Now we are done since $nah_2(1+ x_2g(x_2) )^{-1}$ lies
in $n\Koo$.
\end{proof}

\begin{rem}\label{remCLb}
The algebraic analogue of Lemma \ref{lcda}, namely, with
$\cT_{\Hens}$ instead of $\cT_{\Hens,\, \cA}$ and $\cL_{\Hens}(A)$
instead of $\cL_{\Hens,\, \cA}(A)$, follows from \cite{CLb}, Lemma 7.2.11 and Proposition 7.2.4, or from Theorems 7.2.6 and 7.2.9 of \cite{CLb}.
\end{rem}

\begin{lem}[${\rm Char}(K)=0$]\label{lcda}
Let $K$ be a $\cT_{\Hens,\, \cA}$-model, let $n>0$ be an integer,
and let $f_i(y)$ be $\cL_{\Hens,\, \cA}(A)$-valued field terms in
the $K$-variable $y$ for $i=1\ldots,k$, with $A$ a set of
parameters. Then:
\begin{enumerate}
\item[(i)] There exists a $\cL_{\Hens,\, \cA}(A)$-definable
$b$-map (cf. Definition \ref{defbmin})
$$
\lambda:K\to S
$$
with $B_n$-center
$$
c:\lambda(K)\to K
$$
such that for each $i$
$$
rv_{n}\circ f_i
$$
is a component function of $\lambda$ (that is, $\lambda$ composed
with a coordinate projection gives $rv_n(f_i)$).

\item[(ii)] One can ensure that
the image of $\lambda$ is defined by an $\cL_{\Hens,\,
\cA}(A)$-formula without valued field quantifiers and that $c$ is
given by an $\cL_{\Hens,\, \cA}^\star(A)$-term.
\end{enumerate}
\end{lem}
\begin{proof}
We may suppose that $A\subset K$, since $\cL_{\Hens,\,
\cA}(A)$-valued field terms can not involve auxiliary constants. Let
$K'$ be the valued subfield of $K$ which given by all $\cL_{\Hens,\,
\cA}(A)$-valued field terms. The field $K'$ clearly has analytic
$\cA$-structure. Then $K$ has analytic $\cA(K')$-structure by Remark
\ref{constantsbis}, extending the $\cA$-structure. Since the
$\cT_{\Hens,\, \cA}(A)$-valued field terms are the same as the
$\cT_{\Hens,\, \cA(K')}$-valued field terms, we may suppose that
$\cA(K')=\cA$ and that $A$ is empty.

\par

Apply Theorem \ref{terms} and Remark \ref{moreterms} to each of the
$f_i$, yielding a finite number of polynomials over $K'$, rational
functions over $K'$ and strong units. Write $p_\ell$ for the
polynomials (including the denominators and numerators of the
rational functions). Now by Remark \ref{remCLb} applied to the
polynomials $p_\ell$, Remark \ref{moreterms}, and by Lemma
\ref{rvunits}, and since on the auxiliary sorts lies the full
induced language, there exist functions $\lambda$ and $c$ as
desired.
\end{proof}

The next lemma is only needed for the ball preservation property
statement of Theorem \ref{mt}; it is not needed to prove
$b$-minimality and we use $b$-minimality once in the proof.

\begin{lem}[${\rm Char}(K)=0$]\label{termtoanalytic}
Let $t(x,u)$ be a $\cL_{\Hens,\, \cA}^*(A)$-valued field term in the
valued field variable $x$, with $u$ a fixed tuple of elements of $A$
of auxiliary sorts. Let $n>0$ be an integer. Then there exists a
$b$-map $f:K\to S$ for some auxiliary $S$ such that for each ball
$B$ of the form $f^{-1}(s)$ the restriction of the function
$x\mapsto t(x,u)$ to $B$ lies in $\cO^\dag_K(B)$.
\end{lem}
\begin{proof}
We give a proof by induction on the complexity of the term $t$.
Suppose that terms $t_i(x,u)$, $i=0,\ldots,m$ with $m>0$ have a
$b$-map $f_i:K\to S_i$ as in the lemma. Define $f$ as
$$
f:K\to \prod_i S_i:x\mapsto (f_i(x))_i.
$$
Then, since $\cO^\dag_K(B)$ is a ring for any ball $B$ and since the
intersection of two open balls is either empty or an open ball, the
lemma is satisfied for this $f$ when $t$ is the term $t_1+t_2$ and
likewise for the product $t_1\cdot t_2$. Finally we treat the term
$h_{m,n}(t_0,\ldots,t_m,u_1)$, for $n>0$. By $b$-minimality and a
compactness argument (see, for example, the section on cell
decomposition in \cite{CLb}), we may suppose that the
$rv_{n^2}(t_i(x,u))$ are constant on each fiber $f^{-1}(s)$. Apply
Lemma \ref{ballterm} and the fact that the rings $\cO(B)^\dag(K)$
are closed under (meaningful) composition to see that this $f$ is as
desired. The term $(\cdot)^{-1}$ is a special case of the functions
$h_{m,n}$ (see the proof of Theorem \ref{thens}) so the lemma is
proved.
\end{proof}

\begin{proof}[Proof of Theorem \ref{mt}]
Write $\cL$ for $\cL_{\Hens,\, \cA}$. Although valued field
quantifier elimination can be proven for $\cL$ in the same way as it
is historically done for many of the examples in section 4, we give
a slightly different proof. So, let us first prove elimination of
valued field quantifiers. Let $K$ be a $\cT_{\Hens,\, \cA}$-model.
Let $\varphi(\xi,x,w)$ be a valued-field-quantifier-free
$\cL$-formula with auxiliary variables $\xi$ running over say $S$,
valued field variables $x$ running over $K^m$ and one valued field
variable $w$. We have to prove that there exists a valued field
quantifier free formula $\theta$ such that
$$
\{(\xi,x)\mid K\models \exists w \varphi(\xi,x,w)\} = \{(\xi,x)\mid K\models \theta(\xi,x)\}.
$$

Let $\tau_i$ be the valued field terms occurring in $\varphi$. By
rewriting $\varphi$, we may suppose that these terms only occur in
the form $rv_n(\tau_i)$ for some $n$.
 By Lemma \ref{lcda} and by compactness, there exists a definable
function
$$
\lambda:S\times K^{m+1}\to S'\times S\times K^m
$$
with $S'$ auxiliary such that the $rv_n(\tau_i)$ are component
functions of $\lambda$ (that is, $\lambda$ composed with a
coordinate projection gives $rv_n(\tau_i)$) and such that the image
of $\lambda$ is given by a valued field quantifier free formula. But
then it is easy to construct $\theta$, by quantifying over $S'$.
 This proves the quantifier elimination statement.

Next we prove (b1) and the property about centers. Let $X\subset K$
be a $\cL(A)$-definable set, given by a valued-field-quantifier-free
$\cL(A)$-formula $\varphi(x)$. Let the $f_i$ be all the valued field
terms appearing in $\varphi$.  Lemma \ref{lcda} applied to the $f_i$
implies (b1) and the property about centers.

Property (b2) follows from the quantifier elimination statement.
Namely, consider a valued field quantifier free formula $\varphi$ in
one valued field variable $x$ and auxiliary variables $\xi$ running
over $S$, giving the graph of a surjection from an auxiliary set $S$
to a ball. Let $\tau_i(x)$ be the nonzero valued field terms
appearing in $\varphi$. We may suppose that they all appear in the
form $rv_n(\tau_i(x))$ for finitely many $i$, since $\tau_i(x)=0$ is
equivalent to $rv_n(\tau_i(x))=0$. Since $\varphi$ associates to any
$\xi$ a unique $x$, it follows that $x$ satisfies
$\prod_i\tau_i(x)=0$. Since we may suppose that the $\tau_i$ are not
identically zero, (b2) follows.

For (b3) one uses Lemma 2.4.4 in section 2 of \cite{CLb}; property
($\ast$) there is clear by looking at quantifier free formulas in
two valued field variables which give a definable function, as for
(b2). Such a formula $\varphi$ has nonzero valued field terms
$\tau_i(x,y)$ in the valued field variables $x$ and $y$. We may
again suppose that they only appear in the form $rv_n(\tau_i(x,y))$
for finitely many $n$. Since $\varphi$ describes a graph of a
function $x\mapsto y$, this graph must lie in $\prod_i \tau_i=0$.
Now it is clear that either the image is finite, or some fiber is
finite. This proves ($\ast$) of \cite{CLb} and thus (b3).

The $b$-minimality is proven.

Next we will show that preservation of all balls (in correspondence
with the Jacobian) is a consequence of Theorem \ref{thens}, Lemmas
\ref{ballWD} and \ref{termtoanalytic}, and a compactness argument.

Let $F:X\subset K\to K$ be $A$-definable for some $A$. By Theorem
\ref{thens}, there exists a $\cL_{\Hens,\, \cA}(A)$-definable
function $g:X\to S$ with $S$ auxiliary such that
\begin{equation}\label{et2}
f(x)=t(x,g(x))
\end{equation}
for each $x\in X$ and where $t$ is a $\cL_{\Hens,\, \cA}^*(A)$-term.
By $b$-minimality with centers, we may suppose that $g$ is a $b$-map
which has a $B_{n_0}$-center $c$ for some $n_0$.
 By Lemma \ref{termtoanalytic} we may suppose that on each ball $B$
of the form $g^{-1}(s)$ the restriction of $x\mapsto f(x)$ to $B$ is
in $\cO^\dag_K(B)$. By $b$-minimal cell decomposition, we may
moreover suppose that $rv_n(f'(x))$ is constant on each such $B$.
Fix a ball $B$ of the form $g^{-1}(s)$, say, of size $\alpha_B$,
namely of the form $\{x\mid \ord(x-b)>\alpha_B\}$ for some $b$. Then
there exists by Lemma \ref{ballWD} an integer $n_B$ such that the
restriction of $af$  to any ball of size $\alpha_B+\ord(n_B)$,
contained in $B$, is an isometry, with $a$ any element of the same
size as $1/f'$ on $B$. By compactness, there exists a single integer
$n$ which can serve for all numbers $n_B$ for all such $B$. One can
easily refine $g$ such that $c$ becomes a $B_{n\cdot n_0}$-center,
in a way such that each ball $g^{-1}(s)$ is replaced by a union of
balls of size $\ord (n)$ smaller than $g^{-1}(s)$. The Jacobian
property then follows.
\end{proof}

\begin{rem}\label{rv-Jac2}
Remark \ref{rv-Jac} together with the above proof can be used to get
a stronger Jacobian property than the one given in Definition
\ref{defjacprop2}. Namely, given $n>0$, the function $f$ in
Definition \ref{defjacprop2} can be taken such that for each $s\in
f(X)$ such that $f^{-1}(s)$ is a ball, the restriction of $F$ to
$f^{-1}(s)$ has the Jacobian property and moreover, for all $s\in
f(X)$ and all $x,y\in f^{-1}(s)$,
$$
a(s) rv_n(x-y)=rv_n(F(x)-F(y))
$$
where $rv_n{\rm Jac}F$ is constant on $f^{-1}(s)$ and $a(s):=
rv_n{\rm Jac}F(f^{-1}(s))$. This property is a local form of the
monotonicity property of \cite{Schikh}, section 86.
\end{rem}

\begin{proof}[Proof of Theorem \ref{thens}]
This proof is similar to that of Theorem 7.2.9 of \cite{CLb} and
follows from $b$-minimality with centers and Lemma \ref{lcda}.
 Note that it is
enough to work piecewise. Namely, with the functions $h_{1,1}$ one
can make terms which are characteristic functions of $rv(1)$,
$rv(2)$, and so on, so, and hence one can always paste together a
finite number of terms on finitely many disjoint pieces. For
example, to obtain the characteristic function of $Y\subset X$ as a
term, define $g(x)\in RV$ as $rv(1)$ when $x\in Y$ and as $rv(0)$
else, and let $t(x,\xi)$ be the term $h_{1,1}(1,-1,\xi)$. Note that
also the valued field inverse can be given by a term, using
$h_{1,1}(1,-y,rv(1/y))$. See \cite{CLR1} for more detail.

\par
By Lemma \ref{lcda} and compactness as in the proof of Theorem 7.2.9
of \cite{CLb}, and with $b$-minimal terminology of \cite{CLb},
section 3 and 6 about cell decompositions, it follows that a cell
decomposition theorem holds where all the centers are given by
$\cL_{\Hens,\, \cA}^*(A)$-terms. Partitioning the graph of $f$ into
such cells yields the desired piecewise terms.
\end{proof}

\begin{rem}\label{remmt}
Cell decomposition (as well as other properties) for $\cT_{\Hens,\,
\cA}$ now follows immediately from \cite{CLb} and the $b$-minimality
of $\cT_{\Hens,\, \cA}$ established in Theorem \ref{mt}.
\end{rem}

\end{document}